\documentclass[10pt,reqno]{amsart}
\usepackage{hyperref}
\usepackage{latexsym,amsmath}
\usepackage{enumerate}
\usepackage{amsfonts}
\usepackage{amssymb}
\usepackage{geometry}
\usepackage{latexsym}
\usepackage{fixmath}
\usepackage[prependcaption]{todonotes}

\usepackage[T1]{fontenc}
\usepackage{fourier}
\usepackage{bbm}
\usepackage{color}
\usepackage{fullpage}

\renewcommand\ln{\log}
\renewcommand\Join{*}
\newcommand\msg{\cM}

\newcommand\MC{\mathrm{MC}}

\newcommand\disteq{\,\stacksign{d}=\,}

\newcommand\cutm{\Delta_{\Box}}
\newcommand\Cutm{\cD_{\Box}}
\newcommand\CUTM{D_{\Box}}

\newcommand{\BP}{\mathrm{BP}} 
\newcommand{\OPT}{\mathrm{OPT}} 

\newcommand{\vh}{\vec h}
\newcommand{\vt}{\vec t}
\newcommand{\vw}{\vec w}
\newcommand{\vc}{\vec c}
\newcommand{\vs}{\vec s}
\newcommand{\vz}{\vec z}
\newcommand{\vX}{\vec X}

\renewcommand{\epsilon}{\eps}

\newcommand\vI{\vec I}
\newcommand\vJ{\vec J}

\newcommand\MEAS{\fD^\star}
\newcommand\Meas{\fD}
\newcommand\meas{\fK}
\newcommand\states{\cK}
\newcommand\prior{\fp}

\newcommand{\GG}{\mathbb G}

\newcommand\MU{\vec\mu}
\newcommand\vY{\vec Y}

\newcommand\vm{\vec m}

\newcommand\NU{\vec\nu}

\newcommand\vU{\vec U}

\newcommand\OMEGA{\vec\omega}
\newcommand\PSI{\vec\psi}

\newcommand\VARPHI{\vec\varphi}

\newcommand\nix{\,\cdot\,}

\newcommand\vW{\vec W}
\newcommand\vS{\vec S}
\newcommand\lam{\lambda}

\newcommand\dd{{\mathrm d}}

\newcommand\G{\vec G}

\numberwithin{equation}{section}

\renewcommand{\vec}[1]{\boldsymbol{#1}}

\newcommand\SIGMA{\vec\sigma}
\newcommand\CHI{\vec\chi}
\newcommand\TAU{\vec\tau}

\newtheorem{definition}{Definition}[section]

\newtheorem{example}[definition]{Example}
\newtheorem{remark}[definition]{Remark}
\newtheorem{theorem}[definition]{Theorem}
\newtheorem{lemma}[definition]{Lemma}
\newtheorem{proposition}[definition]{Proposition}
\newtheorem{corollary}[definition]{Corollary}

\newtheorem{fact}[definition]{Fact}

\newcommand\fK{\mathfrak{K}}
\newcommand\fD{\mathfrak{D}}

\newcommand\fA{\mathfrak{A}}

\newcommand\fp{\mathfrak{p}}

\newcommand\cA{\mathcal{A}}
\newcommand\cB{\mathcal{B}}
\newcommand\cC{\mathcal{C}}
\newcommand\cD{\mathcal{D}}
\newcommand\cF{\mathcal{F}}

\newcommand\cE{\mathcal{E}}
\newcommand\cU{\mathcal{U}}
\newcommand\cN{\mathcal{N}}

\newcommand\cH{\mathcal{H}}
\newcommand\cS{\mathcal{S}}

\newcommand\cK{\mathcal{K}}

\newcommand\cM{\mathcal{M}}

\newcommand\cP{\mathcal{P}}
\newcommand\cX{\mathcal{X}}

\newcommand\cV{\mathcal{V}}

\def\cE{{\mathcal E}}

\newcommand\vu{\vec u}
\newcommand\vv{\vec v}
\newcommand\vx{\vec x}
\newcommand\vZ{\vec Z}

\newcommand\THETA{\vec\theta}

\newcommand{\beq}{\begin{equation}} \newcommand{\eeq}{\end{equation}}

\newcommand\thet{\vartheta}

\newcommand\eul{\mathrm{e}}
\newcommand\eps{\varepsilon}

\newcommand\Erw{\mathbb{E}}
\newcommand{\vecone}{\vec{1}}

\newcommand{\Po}{{\rm Po}}

\newcommand{\Be}{{\rm Be}}

\newcommand\TV[1]{\left\|{#1}\right\|_{\mathrm{TV}}}
\newcommand\tv[1]{\|{#1}\|_{\mathrm{TV}}}
\newcommand\dTV{d_{\mathrm{TV}}}

\newcommand{\bink}[2] {{\binom{#1}{#2}}}

\newcommand\bc[1]{\left({#1}\right)}
\newcommand\cbc[1]{\left\{{#1}\right\}}
\newcommand\bcfr[2]{\bc{\frac{#1}{#2}}}
\newcommand{\bck}[1]{\left\langle{#1}\right\rangle}
\newcommand\brk[1]{\left\lbrack{#1}\right\rbrack}
\newcommand\scal[2]{\bck{{#1},{#2}}}

\newcommand\abs[1]{\left|{#1}\right|}

\newcommand\RR{\mathbb{R}}

\newcommand{\Whp}{W.h.p.}
\newcommand{\whp}{w.h.p.}

\newcommand{\stacksign}[2]{{\stackrel{\mbox{\scriptsize #1}}{#2}}}
\newcommand{\tensor}{\otimes}

\newcommand{\Erdos}{Erd\H{o}s}
\newcommand{\Renyi}{R\'enyi}
\newcommand{\Lovasz}{Lov\'asz}

\newcommand{\Mezard}{M\'ezard}

\newcommand{\Szemeredi}{Szemer\'edi}

\newcommand\pr{\mathbb{P}}

\newcommand\Lem{Lemma}
\newcommand\Prop{Proposition}
\newcommand\Thm{Theorem}

\newcommand\Cor{Corollary}
\newcommand\Sec{Section}
\newcommand\Chap{Chapter}

\begin{document}

\title{Spin systems on Bethe lattices}

\author{Amin Coja-Oghlan, Will Perkins}

\address{Amin Coja-Oghlan, {\tt acoghlan@math.uni-frankfurt.de}, Goethe University, Mathematics Institute, 10 Robert Mayer St, Frankfurt 60325, Germany.}

\address{Will Perkins, {\tt math@willperkins.org}, Department of Mathematics, Statistics, and Computer Science, University of Illinois at Chicago, Chicago, Illinois, USA.}

\begin{abstract}
In an extremely influential paper \Mezard\ and Parisi put forward an analytic but non-rigorous approach called the cavity method for studying spin systems on the Bethe lattice, i.e., the random $d$-regular graph [Eur.\ Phys.\ J.\ B {\bf20} (2001) 217--233].
Their technique was based on certain hypotheses; most importantly, that the phase space decomposes into a number of Bethe states that are free from long-range correlations and whose marginals are given by a recurrence called Belief Propagation.
In this paper we establish this decomposition rigorously for a very general family of spin systems.
In addition, we show that the free energy can be computed from this decomposition.
We also derive a variational formula for the free energy.
The general results have interesting ramifications on several special cases.
\hfill{\em MSC: 05C80}
\end{abstract}

\maketitle

\section{Introduction}\label{Sec_intro}

\subsection{Disordered systems and the Bethe lattice}

In 2001 in a ground-breaking contribution \Mezard\ and Parisi proposed an analytic but non-rigorous technique that they called the cavity method for the study of spin glasses on the `Bethe lattice'%
		\footnote{Sometimes the $d$-regular infinite tree is referred to as the `Bethe lattice'.
	However, as \Mezard\ and Parisi point out, the $d$-regular infinite tree does not provide a particularly useful framework for the study of spin interactions because almost all sites belong to the boundary of the tree.
	The random $d$-regular graph, which they and hence we call the Bethe lattice, provides a useful way out: while the local geometry around a given vertex is just a $d$-regular tree, at long distances this tree `wraps around'.},
known in combinatorics as the random $d$-regular graph~\cite{MP1}.
\Mezard\ and Parisi argued that the Bethe lattice constitutes an attractive halfway point between classical `mean-field' models such as the Sherrington-Kirkpatrick model with complete interaction between all sites and spatial models such as the Edwards-Anderson model.
Indeed, the Bethe lattice induces a non-trivial metric on the sites, each of which interacts with only a bounded number of others.
But at the same time \Mezard\ and Parisi showed that the model is amenable to analytic methods, even though matters are significantly more complicated than in the fully connected case.
They went on to argue that the spin glass on the Bethe lattice exhibits many of the properties expected of real glassy systems, such as replica symmetry breaking and the proliferation of pure states.

From the original contribution~\cite{MP1} sprang a truly enormous body of work that has had a transformative impact on an astounding variety of subjects, ranging from physics to combinatorics to machine learning.
Many of the applications may appear unexpected, even surprising.
Almost all of them hinge on the cavity method.
Prominent success stories include the development of `low-density parity check codes', a rare example of a statistical physics idea leading directly to an eminently useful, and widely used, algorithm~\cite{RichardsonUrbanke}.
A further example is a new algorithm for the compressed sensing problem, a fundamental signal processing task~\cite{LF}.
Other important cavity method-based contributions pertain to classical problems in mathematics, such as phase transitions in random graphs and other random structures~\cite{pnas,MM,MPZ}.
The cavity method has also been used to put forward predictions in machine learning, including the capacity of the Hopfield model or on restricted Boltzmann machines~\cite{MezardHopfield}.

Due to these numerous ramifications, the task of vindicating the cavity method rigorously has become an important research task at the junction of mathematical physics, combinatorics and computer science.
There has been a lot of progress recently, e.g., \cite{Banks,DSS3,CKPZ,Mossel}; we shall review the literature in greater detail in \Sec~\ref{Sec_related}.
However, much of this work is concerned with special cases, mostly the `replica symmetric' scenario where there is just a single pure state.

The aim of the present paper is to move past such assumptions and special cases.
We directly confirm several of the key hypotheses of \Mezard\ and Parisi, particularly the decomposition into pure states and the validity of the Belief Propagation recurrence, the mainstay of the cavity calculations.
Further, we obtain a general variational formula for the free energy that is perfectly in line with the \Mezard-Parisi ansatz.
Additionally, we show that the free energy can be computed from the Belief Propagation representation of the pure states of the model.
We obtain these results not merely for a specific model, but for a broad family of models on the Bethe lattice.
The prime example is, of course, the diluted spin glass model.
But in addition, since the proof techniques that we develop are generic, the results apply to models that are of eminent interest in other areas, particularly combinatorics, such as the Potts antiferromagnet or the hard-core model.
Crucially, the results apply universally to all parameter values (such as degree, inverse temperature) of the respective models.

Technically the paper builds upon and continues two intertwined threads of prior work.
First, we bring to bear a variant of the `regularity method' from combinatorics that we developed recently~\cite{Victor,Bethe,Limits} in order to establish the pure state decomposition and to vindicate the Belief Propagation equations.
Second, we seize upon Panchenko's work on asymptotic Gibbs measures and the interpolation method, particularly in order to derive the variational formula for the free energy~\cite{Panchenko,Panchenko2}.
Both of these methods were previously applied with great success to random graphs of \Erdos-\Renyi\ type.
This line of work crucially exploited the relative geometric flexibility of the \Erdos-\Renyi\ model, whose Poisson degree distribution facilitates coupling arguments.
By contrast, the geometry of the Bethe lattice is rigid.
While this entails that the specification of the model, the cavity equations and their solution are quite `clean', 
the rigidity poses substantial technical challenges that the present paper resolves.

Before presenting the main results of the paper, which cover a broad family of problems that we call random factor graph models,  in \Sec~\ref{Sec_results}, we illustrate the results and the concepts around which they revolve with the spin glass model from the original contribution of \Mezard\ and Parisi.  We also work out an additional application to the hard-core model and the independence number of the random regular graph.  Several further applications, including the Potts model and the {\sc Max $q$-Cut} problem, are worked out in \Sec~\ref{Sec_app}.

\subsection{The diluted spin glass}\label{Sec_diluted_sg}

For integers $d\geq3$, $n>0$ such that $dn$ is even, let $\GG=\GG(n,d)$ be the uniformly random $d$-regular graph on the vertex set $V_n=\{v_1,\ldots,v_n\}$.
With each edge $e\in E(\GG)$ comes a standard Gaussian $J_e$.
The random variables $(J_e)_{e\in E(\GG)}$ are mutually independent.
For a given inverse temperature $\beta>0$, the diluted spin glass on $\GG$ is the probability distribution on $\{\pm1\}^{V_n}$ defined by
	\begin{align}\label{eqsg}
	\mu_{\GG}(\sigma)&=
	\frac1{Z(\GG)}\prod_{vw\in E(\GG)}\frac{1+\tanh(\beta J_{vw})\sigma_v\sigma_w}2,
	\end{align}
where the partition function $Z(\GG)$ ensures normalization.
	\footnote{The expression (\ref{eqsg}) is equivalent to the possibly more familiar formula
		$\mu_{\GG}(\sigma)\propto\exp\bc{\beta\sum_{vw}\sigma_v\sigma_w }$.}
Without the couplings $J_e$, this would just be the ferromagnetic Ising model on $\GG$.
But since the $J_e$ are independent Gaussians, some will be positive and others negative.
In effect, some edges induce ferromagnetic and others antiferromagnetic interactions, causing frustration.
Thus, $\mu_{\GG}$ is a spin glass model, the well-known diluted spin glass on the Bethe lattice.

There are two fundamental problems associated with this and numerous similar models:
first, to characterize the structure of the Boltzmann distribution $\mu_{\GG}$.
Does it exhibit long-range correlations?
Does it decompose into one or several `pure states', and if so, how can we characterize them?
Second, to calculate the quantity $\lim_{n\to\infty} \frac{1}{n}\Erw[\ln Z(\GG)]$, which we call the {\em free energy density}.  Its fundamental importance is due to the fact that other important observables derive from it.
Moreover, the singularities of the function $\beta\mapsto\lim_{n\to\infty} \frac{1}{n}\Erw[\ln Z(\GG)]$ constitute the phase transitions of the model.

\subsubsection*{Bethe states and the Boltzmann distribution}
With respect to the first problem, \Mezard\ and Parisi hypothesized that the Boltzmann distribution always decomposes into one or a moderate (albeit not necessarily bounded) number of pure states.
Further, they hypothesized that these pure states are characterized by fixed points of a recurrence called Belief Propagation.
Our first theorem confirms this hypothesis.

To be precise, writing $\partial v$ for the set of neighbors of a vertex $v$, let $\cM(\GG)$ be the set of all families $(\nu_{u\to v})_{u\in V_n,u\in\partial v}$ such that $\nu_{u\to v}\in[0,1]$.
We call $\nu_{u\to v}$ the {\em message} from $u$ to $v$.
The messages need not be symmetric, i.e., possibly $\nu_{u\to v}\neq\nu_{v\to u}$.
Furthermore, {\em Belief Propagation} is the operator $\BP:\cM(\GG)\to\cM(\GG)$, $\nu\mapsto\hat\nu$, where
	\begin{align}\label{eqBPsg}
	\hat\nu_{v\to u}&=\frac{\prod_{w\in\partial v\setminus u}1+2\tanh(\beta J_{vw})(\nu_{w\to v}-1/2)}
	{\sum_{\sigma\in\{\pm1\}}\prod_{w\in\partial v\setminus u}1+2\sigma\tanh(\beta J_{vw})(\nu_{w\to v}-1/2)}.
	\end{align}
The motivation behind this operator, and the origin of the name `cavity method', is this.
Suppose we fix a vertex $v$ in a $d$-regular graph along with a neighbor $u$.
Now suppose we remove the vertex $u$, thereby creating a `cavity'.
Then the `ideal' message $\mu_{\GG,u\to v}$ that we would like to compute is just the marginal probability $\mu_{\GG-v,u}(1)$ that $u$ takes spin $1$ in the subgraph obtained by removing $v$.
If the Boltzmann distribution $\mu_{\GG}$ is free from long-range correlations, then 
these ideal messages should plausibly be a fixed point of the BP operator.
Indeed, if we remove $v$, then very likely its former neighbors will be mutually far apart in the resulting graph.
In effect, the joint distribution of their spins should factorize.
If so, then a straightforward calculation verifies that the ideal messages are a fixed point of BP.
In fact this reasoning goes back to Bethe's classical work~\cite{Bethe}.

However, generally spin glass models do exhibit long-range correlations, a phenomenon called {\em replica symmetry breaking} (see, e.g., \cite{SoftCon,COZ} for proofs that replica symmetry breaking occurs in certain models).
Yet the fundamental hypothesis of \Mezard\ and Parisi holds that the phase space $\{\pm1\}^{V_n}$ always decomposes into {\em Bethe states} $S_1,\ldots,S_\ell$ in such a way that the conditional distributions $\mu_{\GG}[\nix|S_h]$ are free from long-range correlations.
Formally, this means that if we pick a pair of vertices $(v_i,v_j)$ uniformly at random, then typically the conditional joint distribution $\mu_{\GG,v_i,v_j}[\nix|S_h]$ of the spins of $v_i$ and $v_j$ is close to the product distribution $\mu_{\GG,v_i}(\nix|S_h)\tensor\mu_{\GG,v_j}(\nix|S_h)$, i.e.,
	\begin{align}\label{eqPureStateSg1}
	\frac1{n^2}\sum_{1\leq i<j\leq n}\TV{\mu_{\GG,v_i,v_j}(\nix|S_h)-\mu_{\GG,v_i}(\nix|S_h)\tensor\mu_{\GG,v_j}(\nix|S_h)}&=o(1).
	\end{align}
In effect, within each Bethe state the `ideal' messages are predicted to be an approximate fixed point of the BP operator.
To be precise, for adjacent vertices $u,v$ we write $\mu_{\GG,v\to u}[S_h]=\mu_{\GG-u,v}(1|S_h)$ for the conditional probability given $S_h$ that $v$ takes spin $1$  in the subgraph of $\GG$ with $u$ removed.
Then we expect that
	\begin{align}\label{eqPureStateSg2}
	\frac1n\sum_{i=1}^n\sum_{u\in\partial v_i}\TV{\mu_{\GG,v_i\to u}[S_h]
		-\hat\mu_{\GG,v_i\to u}[S_h]}&=o(1)&\mbox{where}\quad
		(\hat\mu_{\GG,v\to u}[S_h])_{v\in V_n,u\in\partial v}=\BP(\mu_{\GG,v\to u}[S_h])_{v\in V_n,u\in\partial v}.
	\end{align}
Further, the cavity method predicts that the Boltzmann marginals can be obtained from the messages by a formula quite similar to (\ref{eqBPsg}):
	\begin{align}\label{eqPureStateSg3}
	\frac1n\sum_{i=1}^n\abs{\mu_{\GG,v_i}(1|S_h)-
		\frac{\prod_{w\in\partial v_i}1+2\tanh(\beta J)(\mu_{\GG,w}[S_h]-1/2)}
			{\sum_{\sigma\in\{\pm1\}}\prod_{w\in\partial v_i}1+2\sigma\tanh(\beta J)(\mu_{\GG,w}[S_h]-1/2)}}&=o(1).
	\end{align}
The following theorem establishes these conjectures rigorously.
We say that $\GG$ enjoys a property {\em with high probability} (`\whp') if the probability that the property holds tends to one as $n\to\infty$.

\begin{theorem}\label{Thm_sg}
For any $d\geq3$, $\beta>0$ the following is true.
Let $L=L(n)\to\infty$ be any integer sequence that tends to infinity.
Then there exists a decomposition $S_0=S_0(\GG),S_1=S_1(\GG),\ldots,S_\ell=S_\ell(\GG)$, $\ell=\ell(\GG)\leq L$, of the phase space $\{\pm1\}^n$ into non-empty sets such that $\mu_{\GG}(S_0)=o(1)$ and such that with high probability
 (\ref{eqPureStateSg1})--(\ref{eqPureStateSg3}) are satisfied for $h=1,\ldots,\ell$.
\end{theorem}	

\noindent
Crucially, and in contrast to much prior work in this area, \Thm~\ref{Thm_sg} applies indiscriminately to all $d,\beta$.
While it is expected that in the `high-temperature' regime (small $\beta$) there is just a single pure state, it is widely conjectured that for large $d$ and $\beta$ the number of pure states is unbounded.
Thus, we do not expect that it will be possible to replace the unbounded $L$ in \Thm~\ref{Thm_sg} by a constant.
Yet \Thm~\ref{Thm_sg} shows that the number of states can be upper bounded by an {\em arbitrarily} slowly growing function $L(n)$.

\subsubsection*{The free energy}
The Bethe states and their associated messages contain all the information needed to compute the free energy.
To be precise, once more following the ideas of \Mezard\ and Parisi, we can set up a recurrence for computing the difference
$\Erw[\ln Z(\GG(n+1,d))]-\Erw[\ln Z(\GG(n,d))]$, which in turn enables us to write a formula for $\frac1n\Erw[\ln Z(\GG(n,d))]$ by telescoping.
To set up such a recurrence it is necessary to crack the rigid geometry of the random regular graph open a little bit.
To this end, we resort to the idea of creating a few `cavities'.
Specifically, we delete a few random vertices and edges from $\GG(n,d)$. 
Formally, let $\omega>0$ and let $\vX,\vY$ be two independent Poisson variables with mean $\omega$.
Moreover, let $\vu_1,\ldots,\vu_{\vX}$ and $\vv_1\vw_1,\ldots,\vv_{\vY}\vw_{\vY}$ be sequences of uniformly random vertices and edges of $\GG$, chosen independently.
With $S_1,\ldots,S_\ell$ the decomposition from \Thm~\ref{Thm_sg}, we introduce weights 
	\begin{align*}
	\vz_{\GG,h}=\mu_{\GG}(S_h)&\cdot\prod_{i=1}^{\vX}\bc{
		\sum_{\sigma\in\{\pm1\}}\prod_{v\in\partial\vu_i}1+2\tanh(\beta J_{v\vu_i})
					(\mu_{\GG,v\to\vu_i}[\sigma|S_h]-1/2)}^{-1}\\
		&\cdot\prod_{i=1}^{\vY}\bc{1+4\tanh(\beta J_{\vv_i\vw_i})
			(\mu_{\GG,\vv_i\to\vw_i}[1|S_h]-1/2)(\mu_{\GG,\vw_i\to\vv_i}[1|S_h]-1/2)}^{-1}
	\end{align*}
and $\vz_{\GG}=\sum_{h=1}^\ell\vz_{\GG,h}$.
Further, let $\cC(\GG)$ be the set of all vertices of degree less than $d$ in the graph $\GG_{n,\omega}$ obtained from $\GG$ by removing $\vu_1,\ldots,\vu_{\vX}$ and $\vv_1\vw_1,\ldots,\vv_{\vY}\vw_{\vY}$.
Then with high probability each $c\in\cC(\GG)$ has degree precisely $d-1$, and we  write $c'$ for the erstwhile $d$'th neighbor of $c$.
Further, with $\vc_1,\vc_2,\ldots$ a sequence of uniformly and independently chosen elements of $\cC(\GG)$ and $(\vJ_i)_{i\geq1}$ a sequence of independent standard Gaussians, we let
	\begin{align*}
	\cB(\GG)&=\Erw\brk{
		\ln{\sum_{h=1}^\ell\frac{\vz_{\GG,h}}{\vz_{\GG}}
			\sum_{\sigma\in\{\pm1\}}\prod_{i=1}^d1+2\sigma\tanh(\beta\vJ_{i})(\mu_{\GG,\vc_i\to\vc_i'}[S_h]-1/2)}			\Bigg|\GG}\\
		&\qquad\qquad-\frac d2\Erw\brk{\ln{1+4\tanh(\beta\vJ_1)\sum_{h=1}^\ell\frac{\vz_{\GG,h}}{\vz_{\GG}}
			(\mu_{\GG,\vc_1\to\vc_1'}[1|S_h]-1/2)
					(\mu_{\GG,\vc_2\to\vc_2'}[1|S_h]-1/2)}\Bigg|\GG}-\frac d2\ln 2.
	\end{align*}
The expression $\cB(\GG)$ mirrors our recurrence for the difference  $\Erw[\ln Z(\GG(n+1,d))]-\Erw[\ln Z(\GG(n,d))]$.
Having created a moderate number of cavities, we insert a new $(n+1)$st vertex, connected to $d$ randomly chosen `cavities'.
The first summand above represents the ensuing change in the free energy.
But this operation adds $d$ more edges, whereas a random regular graph with $n+1$ vertices only has $d/2$ more edges than one with $n$ vertices.
Therefore, a correction term is needed. 
Hence the second summand.

Crucially,  the functional $\cB(\GG)$ depends only on the pure state decomposition from \Thm~\ref{Thm_sg} and the associated messages.
The following theorem shows that this information suffices to compute the free energy.

\begin{theorem}\label{Thm_sgZ}
For all $d\geq3,\beta>0$ we have
	$$\lim_{n\to\infty}\frac1n\Erw[\ln Z(\GG)]=\liminf_{\omega\to\infty}\,\liminf_{n\to\infty}\,\Erw[\cB(\GG)].$$
\end{theorem}

Entirely in line with the ideas developed in~\cite{MP1},  \Thm~\ref{Thm_sgZ} establishes a direct conceptual link between Belief Propagation and the pure state decomposition from \Thm~\ref{Thm_sg} and the free energy for {\em all} $d,\beta$.
Of course, in order to evaluate $\cB(\GG)$ it is necessary to actually determine the pure state decomposition along with the corresponding Belief Propagation messages.
The shape of this decomposition, and the practical difficulty of computing it, will depend significantly on the parameters $d,\beta$.
Alternatively, as we see next, it is possible to derive a variational formula for the free energy.

\subsubsection*{A variational formula}

The variational formula comes in terms of an optimization problem on a space that resembles the graphon space from the theory of graph limits~\cite{Lovasz}.
To be precise, let $\nu:[0,1]^2\to[0,1]$, $(s,x)\mapsto\nu_{s,x}$ and $\nu':[0,1]^2\to[0,1]$, $(s,x)\mapsto\nu'_{s,x}$ be measurable maps.
We define the {\em cut distance} between $\nu,\nu'$  by
	\begin{align*}
	\Cutm(\nu,\nu')&=\inf_{\varphi,\varphi'}\,\sup_{\substack{S,X\subset[0,1]}}
		\abs{\int_S\int_X\nu_{s,x}(\omega)-\nu'_{\varphi(s),\varphi'(x)}(\omega)\dd x\,\dd s},
	\end{align*}
where $\varphi,\varphi':[0,1]\to[0,1]$ are measurable maps that preserve the Lebesgue measure and $S,X\subset[0,1]$ are measurable.
Obtain the space $\meas$ by identifying any $\nu,\nu'$ with $\Cutm(\nu,\nu')=0$.
Then $\meas$ endowed with the cut distance is a compact metric space.
In addition, write $\Meas$ for the space of probability measures on $\meas$.

The formula for the free energy comes as a variational problem on a subspace $\MEAS$ of $\Meas$, defined as follows.
Let $N,M\geq0$ be integers.
For $\mu\in\meas$ we define a randomly perturbed $\mu^{*(N,M)}\in\meas$ as follows.
Let $(\vx_{i,j})_{i,j\geq1}$ be a family of uniform random variables on $[0,1]$ and let $(\vJ_{i,j})_{i,j\geq1}$ be a family of standard Gaussians, all mutually independent.
Then for $s\in[0,1]$ we define
	\begin{align*}
	\vz_{s}&=\prod_{i=1}^N\bc{
				\sum_{\sigma\in\{\pm1\}}\prod_{j=1}^d1+2\tanh(\beta \vJ_{i,j})(\mu_{s,\vx_{i,j}}-1/2)}
				\prod_{i=1}^M\bc{1+4\tanh(\beta \vJ_{i+N,1})
	(\mu_{s,\vx_{i+N,1}}-1/2)(\mu_{s,\vx_{i+N,2}}-1/2)}.
	\end{align*}
Further, let
	\begin{align*}
\vt&=\vt(s)=\inf\cbc{u\in[0,1]:\int_0^u\vz_{u}\dd u\geq s\int_0^1\vz_{u}\dd u},&
\qquad\mbox{and}\qquad\mu^{*(N,M)}_{s,x}=\mu_{\vt,x}\in\meas.
\end{align*}	
Now, suppose that $\pi\in\Meas$ is  a distribution, and write $\MU^\pi\in\meas$ for a sample from $\pi$.
Then we let $\MEAS$ be the set of all $\pi\in\Meas$ such that the perturbed $\MU^{\pi*(N,M)}$ has distribution $\pi$ again for all $N,M\geq0$.

The definition of $\MEAS$, which is an adaptation of the one stated by Panchenko~\cite{Panchenko} in the case of models of \Erdos-\Renyi\ type,
mirrors a natural combinatorial invariance properties of the graph $\GG_{n,\omega}$ with the random cavities.
Indeed, because the numbers $\vX,\vY$ of deleted edges and vertices are Poisson with a large mean $\omega$, for any fixed $N,M$ 
the random graph $\GG_{n,\omega}$ with $\vX$ deleted vertices and $\vY$ deleted edges is close in total variation to the one with merely $\vX-N$ deleted vertices and $\vY-M$ deleted edges.
Furthermore, because adding or removing a small number of edges only affects the Boltzmann weights by a bounded factor, we should expect that the Bethe states of these two factor graphs remain the same.
But, of course, the relative probability masses of the Bethe states will be different.
Accordingly, the weights $\vz_s$ mirror the changes in the weights of the Bethe states upon re-insertion of $N$ vertices, each with $d$ incident edges, and another $M$ edges into $\GG_{n,\omega}$.
Once we take $\omega$ and $n$ to infinity, the closeness of the two random factor graphs in total variation translates into the statement that the distribution of the messages emitted by the cavities of $\GG_{n,\omega}$ belongs to $\MEAS$.

Finally, define a functional $\cB:\meas\to\RR$ by letting
	\begin{align*}
	\cB(\mu)&=\Erw\Big[
		\ln\bc{\sum_{\sigma\in\{\pm1\}}\int_0^1\prod_{i=1}^d
			1+2\sigma\tanh(\beta\vJ_{i,j})(\mu_{s,\vx_{i,j}}-1/2)\dd s}\\
		&\qquad-\frac d2\ln\bc{1+4\tanh(\beta\vJ_{1,1})\int_0^1(\mu_{s,\vx_{1,1}}-1/2)(\mu_{s,\vx_{1,2}}-1/2)
			\dd s}\Big]-\frac d2\ln 2.
	\end{align*}
We are ready to state the variational formula for the free energy.

\begin{theorem}\label{Thm_sgvariational}
For all $d\geq3$ and $\beta>0$ we have
	$\displaystyle\lim_{n\to\infty}\frac1n\Erw[\ln Z_\beta(\GG)]=\min_{\pi\in\MEAS}\Erw[\cB(\MU^\pi)].$
\end{theorem}

\Thm~\ref{Thm_sgZ} provides the combinatorial interpretation of the optimal $\pi$ for \Thm~\ref{Thm_sgvariational}:
it is the kernel representing the messages $(\mu_{\GG,\vc\to\vc'}[\nix|S_h])_{c\in\cC(\GG),h=1,\ldots,\ell}$ sent out by the cavities on the individual Bethe states.

\subsection{The hard-core model}\label{Sec_introhc}
As a second application we discuss the hard-core model on the random regular graph $\GG=\GG(n,d)$.
This is a probability distribution on the collection of independents sets of $\GG$ parametrized by $\lambda>0$, the {\em fugacity}.
Formally, encoding subsets of the vertex set by their indicator vectors, we define
	\begin{align*}
	\mu_{\GG}(\sigma)&=\frac{\lambda^{\sum_{i=1}^n\sigma_i}}{Z(\GG)}
		\prod_{1\leq i<j\leq n}1-\vecone\{v_i\in\partial v_j\}\sigma_i\sigma_j
		&(\sigma\in\{0,1\}^n),
	\end{align*}
with $Z(\GG)$ the partition function that turns $\mu_{\GG}$ into a probability measure.
Thus, $\mu_{\GG}(\sigma)=0$ unless the $1$-entries of $\sigma$ form an independent set in $\GG$, in which case
the weight of $\sigma$ is proportional to $\lambda$ taken to the power of the size of the independent set.

The hard-core model, of great prominence in statistical physics, is of eminent importance in combinatorics as well because it is closely related to the problem of finding the size of the largest independent set of the random regular graph.  For $d$ large, this problem was solved by Ding, Sly, and Sun~\cite{DSSind} using an intricate version of the second-moment method guided by insights from the 1-step replica symmetry breaking (1RSB) version of the cavity method.
But according to the physics predictions~\cite{BarbierHC}, the 1RSB method runs into an inherent obstacle for small $d$ as the model exhibits a continuous phase transition to a more complicated `full replica symmetry breaking' (full RSB) phase.
In \Cor~\ref{Cor_hc} below we will derive a formula for the largest independent set size that holds for all $d$ and that accommodates the full RSB scenario.

But let us first deal with the free energy of the hard-core model, in and of itself a well-known problem.
To derive a variational formula for the free energy, 
obtain $\meas_\lambda$ from the space of all measurable functions $[0,1]^2\to[0,\lambda/(1+\lambda)]$ by identifying any $\nu,\nu'$ with $\Cutm(\nu,\nu')=0$.
Then $\meas_\lambda$ is a compact.
In addition, we let $\Meas_\lambda$ be the space of probability measures on $\meas_\lambda$.
Similarly to the spin glass problem, the formula for the free energy comes as a variational problem on a subspace $\MEAS_\lambda$ of $\Meas_\lambda$.
This subspace is defined as follows.
Let $(\vx_{i,j})_{i,j\geq1}$ be a family of independent random variables, uniformly distributed on $[0,1]$, and let $N,M\geq0$ be integers.
Then for $\mu\in\meas_\lambda$ we define a random $\mu^{*(N,M)}\in\meas_\lambda$ as follows.
For $s\in[0,1]$ let
	\begin{align*}
	\vz_s&=\prod_{i=1}^N\bc{1+\lambda\prod_{j=1}^d1-\mu_{s,\vx_{i,j}}}
		\prod_{i=1}^M\bc{1-\mu_{s,\vx_{i+N,1}}\mu_{s,\vx_{i+N,2}}}\quad\mbox{and}\quad
	\vt=\vt(s)=
	\inf\cbc{u\in[0,1]:\int_0^u\vz_{\mu,u}^{N,M}\dd s\geq s\int_0^1\vz_{\mu,u}^{N,M}\dd u}
	\end{align*}
and set
	$$\mu^{*(N,M)}_{s,x}=\mu_{
		\vt,x}\in\meas.$$
Further, suppose that $\pi\in\Meas_\lambda$ is  a distribution, and write $\MU^\pi\in\meas_\lambda$ for an element chosen from $\pi$.
Then we let $\MEAS_\lambda$ be the set of all $\pi\in\Meas_\lambda$ such that $\MU^\pi$ and $\MU^{\pi*(N,M)}$ are identically distributed for all $N,M\geq0$.
Finally, let $\cB:\meas_\lambda\to\RR$ be the function defined by
	\begin{align*}
	\cB(\mu)&=\Erw\brk{\ln\bc{1+\lambda\int_0^1
		\prod_{j=1}^d1-\mu_{s,\vx_{1,j}}
		\dd s}
		+\frac d2\log\bc{1-\int_0^1\mu_{s,\vx_{1,1}}\mu_{s,\vx_{1,2}}\dd s}}.
	\end{align*}
The variational formula for the free energy reads as follows.

\begin{theorem}\label{Thm_hc}
For all $d\geq3$ and $\lambda>0$ we have
		\begin{align*}
\lim_{n\to\infty}\frac1n\Erw[\ln Z(\GG)]=\Phi_{d,\lambda},\qquad\mbox{with}\qquad
	\Phi_{d,\lambda}=\min_{\pi\in\MEAS_\lambda}\Erw[\cB(\MU^\pi)].
	\end{align*}
\end{theorem}

In the limit $\lambda\to\infty$ the distribution $\mu_{\GG,\lambda}$ concentrates on the maximum independent sets of the random graph.
As an application of \Thm~\ref{Thm_hc} we therefore obtain the following result on the size of the largest independent set, i.e., the independence number $\alpha(\GG)$ of the random graph.

\begin{corollary}\label{Cor_hc}
For all $d\geq3$ we have
	$\displaystyle\lim_{n\to\infty}\frac1n\Erw[\alpha(\GG)]=\lim_{\lambda\to\infty}\lambda\cdot(\Phi_{d,\lambda+1}-\Phi_{d,\lambda}).$
\end{corollary}

\noindent
The formula in \Cor~\ref{Cor_hc} may not be easy to evaluate; in particular, it may be difficult to obtain a numerical estimate for a given value of $d$.
Nonetheless, since the proofs show that the optimal $\pi$ in \Thm~\ref{Thm_hc} is closely related to the Belief Propagation fixed points on $\GG$, it should be possible to extract combinatorial information about the independent set problem on random graphs.
In any case, \Thm~\ref{Thm_hc} and \Cor~\ref{Cor_hc} put a lid on the complexity of the problem.

\subsection{Organization}
In \Sec~\ref{Sec_results} we present the main results of the paper, which cover a broad family of random factor graph models.
At the end of \Sec~\ref{Sec_results} we are in a position to discuss related work in detail.
\Sec s~\ref{Sec_prelim}--\ref{secAz} deliver the proofs of these general results.
Finally, in \Sec~\ref{Sec_app} we show how \Thm s~\ref{Thm_sg}--\ref{Thm_hc} and \Cor~\ref{Cor_hc} follow from the general results in \Sec~\ref{Sec_results}. 
In addition, we work through several more applications that have each received considerable attention in their own right, such as the Potts antiferromagnet.

\section{Random factor graphs}\label{Sec_results}
\noindent
In this section we present the main results of the paper, which cover a broad class of models called {\em random factor graphs}.
The class encompasses many well-studied examples of problems on random regular graphs or hypergraphs, including the spin glass model from the previous section.
Some other cases, such as the hard-core model or extremal cuts, can be dealt with by taking limits; we will come to that in \Sec~\ref{Sec_app}.

\subsection{Definitions}

To define random factor graph models, we consider a finite set $\Omega\neq\emptyset$ whose elements we call {\em spins}.
Moreover, for an integer $k\geq2$ we let $(\Psi,P)$ be a  probability space of  {\em weight functions} $\psi:\Omega^k\to(0,1)$.
We always denote by $\PSI$ an element of $\Psi$ chosen from the distribution $P$.
The space $\Psi$ may be finite or infinite.
In the latter case we assume that
	\begin{align}\label{eqP}
	\Erw[\exp(1/
			\min_{\sigma\in\Omega^k}\PSI(\sigma))]<\infty.
	\end{align}
Furthermore, we always assume that the distribution $P$ is invariant under permutations of the coordinates.
That is, for any $\psi\in\Psi$ and for any permutation $\kappa$ of $\{1,\ldots,k\}$ the function
$\psi^\kappa:\sigma\mapsto\psi(\sigma_{\kappa_1},\ldots,\sigma_{\kappa_k})$ belongs to $\Psi$ as well
and $\PSI^\kappa$ has the same distribution as $\PSI$.
Additionally, let $p$ be a probability distribution on $\Omega$ with $p(\omega)>0$ for all $\omega\in\Omega$.
Further, let $d\geq3 ,n>0$ be integers and set $m=\lfloor dn/k\rfloor $.
Let $V_n=\{v_1,\ldots,v_n\}$ be a set of {\em variable nodes} and let $F_m=\{a_1,\ldots,a_m\}$ be a set of {\em constraint nodes}.

\begin{definition}
Suppose that $k$ divides $dn$. 
The random factor graph $\G=\G(n,d,p,P)$ consists of
	\begin{itemize}
	\item a weight function $\psi_{a_i}\in\Psi$ drawn from the distribution $P$ independently for each $i=1,\ldots,m$ and
	\item an independent uniformly random bijection $\partial_{\G}:F_m\times\{1,\ldots,k\}\to V_n\times\{1,\ldots,d\}$.
	\end{itemize}
\end{definition}

The definition resembles the pairing model of random regular graphs~\cite{JLR}.
Accordingly, we use standard graph-theoretic terminology.
For instance, we call $x_i\in V_n$ and $a_j\in F_m$ adjacent if there exist $s\in[d]$ and $k\in[k]$ such that $\partial_{\G}(a_j,t)=(x_i,s)$.
We also use the symbol $\partial_{\G}(a_j,t)$ for the variable node $x_i$ such that 
$\partial_{\G}(a_j,t)=(x_i,s)$.
Further, we write $\partial_{\G}x_i$ for the set of all $a_j\in F_m$ that $x_i$ is adjacent to, and similarly for $a_j$.
We omit the index and just write $\partial x_i,\partial a_j$ etc.\ where the reference to the random graph is apparent.
In particular, $\G$ induces a bipartite graph on the variable and constraint nodes, and thereby the shortest path metric on $V_n\cup F_m$.
Hence, by extension of the above notation, we write $\partial^\ell_{\G}u$ for the set of all nodes at distance precisely $\ell$ from $u$
and $\nabla^\ell_{\G}u$ for set of all variable nodes at distance at most $\ell$ from $u$.

We let $\cS$ be the event that $\G$ is simple, i.e., that there do not occur multiple edges between any variable and constraint nodes.
Moreover, we denote by $\GG$ the conditional distribution of $\G$ given $\cS$.
Let us make a note of the following well known fact.

\begin{fact}[\cite{JLR}]\label{Fact_simple}
We have $\pr\brk{\G\in\cS}\sim\exp\brk{-(d-1)(k-1)/2-\vecone\{k=2\}(d-1)^2/4}$.
\end{fact}

The random factor graph induces a probability distribution on $\Omega^{V_n}$.
To define it, we introduce the shorthand $\psi_{a_i}(\sigma)=\psi_{a_i}(\sigma(\partial(a_i,1)),\ldots,\sigma(\partial(a_i,k)))$ for $i\in[m]$ and $\sigma\in\Omega^{V_n}$.
Thus, $\psi_{a_i}(\sigma)$ is the weight that constraint node $a_i$ gives to$\sigma$.
Further, we introduce the {\em total weight}
	\begin{align*}
	\psi_{\G}(\sigma)&=\prod_{i=1}^m\psi_{a_i}(\sigma)
		&(\sigma\in\Omega^{V_n}).
	\end{align*}
by multiplying up all the weight functions of the constraint nodes.
The total weights $\psi_{\G}(\sigma)$  give rise to the {\em partition function} and the {\em Boltzmann distribution}:
	\begin{align}
	\label{eqPartitionFunction}
	Z(\G)&=\sum_{\tau\in\Omega^{V_n}}\psi_{\G}(\tau)\prod_{i=1}^np(\sigma(x_i)),&
	\mu_{\G}(\sigma)&=\frac{\psi_{\G}(\sigma)}{Z(\G)}\prod_{i=1}^np(\sigma(x_i))&(\sigma\in\Omega^{V_n}).
	\end{align}
Since all the weight functions $\psi\in\Psi$ are strictly positive, the Boltzmann distribution is a well-defined probability measure on the phase space $\Omega^{V_n}$.

We set out to investigate the structure of the Boltzmann distribution $\mu_{\G}(\nix)$ and to compute the partition function $Z(\G)$ or, more specifically, its logarithm, which we call the {\em free energy}.
In \Sec~\ref{Sec_BetheStates} we will prove the main result of the paper, which provides that the Boltzmann distribution decomposes into a convex combination of relatively simple distributions called Bethe states.
But before we come to that, let us look at an example.

\begin{example}[the $k$-spin model]\label{Ex_sg}
Let $\Omega=\{\pm1\}$, let $k\geq2$ be an integer and let $\beta>0$ be a real parameter.
The $k$-spin model is a generalization of the spin glass model from the previous section, which corresponds to the special case $k=2$.
The weight functions of the $k$-spin model read
	\begin{align*}
	\psi_{\beta,J}(\sigma_1,\ldots,\sigma_k)&=\frac12\bc{1+\tanh(\beta J)\prod_{i=1}^k\sigma_i}\qquad\qquad\qquad(J\in\RR).
	\end{align*}
Thus, $\Psi=\{\psi_{\beta,J}:J\in\RR\}$, and the distribution $P$ on $\Psi$ is defined by choosing $J$ from the standard Gaussian distribution.
This distribution clearly satisfies~\eqref{eqP}.
Geometrically, this model lives on a generalized Bethe lattice where all variable nodes, representing the sites, have degree $d$, while all constraint nodes, representing the interactions, have degree $k$.
\end{example}

The hard-core model from \Sec~\ref{Sec_introhc} cannot be expressed as a factor graph model directly because of the requirement that all weight functions be strictly positive.
But it is possible to arrive at the hard-core model by taking suitable limits; see \Sec~\ref{Sec_app} for details.

\subsection{Bethe states}\label{Sec_BetheStates}
The Belief Propagation message-passing scheme provides the mainstay of the physicists' non-rigorous cavity method.
Our first main result vindicates its use by showing that the Boltzmann distribution of any random factor graph model can be described in terms of Belief Propagation fixed points.

To introduce Belief Propagation let $\msg(\G)$ be the {\em message space}, consisting of all families
	$$\nu=(\nu_{v\to a},\nu_{a\to v})_{v\in V_n, a\in F_n:v\in\partial_{\G}a}$$
of probability measures $\nu_{v\to a},\nu_{a\to v}$ on $\Omega$.
For adjacent $a,v$ we interpret $\nu_{v\to a}$ as a `message' from $v$ to $a$, and $\nu_{a\to v}$ as a message in the reverse direction.
We equip $\msg(\G)$ with the metric
	\begin{align*}
	\cD_1(\nu,\nu')&=\frac1n\sum_{v,a:v\in\partial_{\G}a}\TV{\nu_{v\to a}-\nu'_{v\to a}}+\TV{\nu_{a\to v}-\nu'_{a\to v}}.
	\end{align*}
Belief Propagation is the operator $\BP:\msg(\G)\to\msg(\G)$ that maps $\nu$ to $\hat\nu$ defined by
	\begin{align*}
	\hat\nu_{v\to a}(\sigma)&=\frac{p(\sigma)\prod_{b\in\partial v\setminus a}\nu_{b\to v}(\sigma)}
		{\sum_{\tau\in\Omega}p(\tau)\prod_{b\in\partial v\setminus a}\nu_{b\to v}(\tau)},&
	\hat\nu_{a\to v}(\sigma)&=\frac{\sum_{\tau\in\Omega^{\partial a}}\vecone\{\tau_v=\sigma\}\psi_a(\tau)\prod_{w\in\partial a\setminus v}\nu_{w\to a}(\tau_w)}
	{\sum_{\tau\in\Omega^{\partial a}}\psi_a(\tau)\prod_{w\in\partial a\setminus v}\nu_{w\to a}(\tau_w)}.
	\end{align*}
Further, a point $\nu\in\msg(\G)$ is an {\em $\eps$-Belief Propagation fixed point} if $\cD_1(\nu,\BP(\nu))<\eps.$

For a thorough discussion and motivation of Belief Propagation we refer to~\cite{MM}.
The punch line is that on acyclic factor graphs a Belief Propagation fixed point computation provably yields the marginals of the Boltzmann distribution as well as the free energy.
Since the random graph $\G$ contains only very few short cycles, one may therefore expect that Belief Propagation renders meaningful information on random factor graphs as well, provided that the Boltzmann distribution is free of long-range correlations.

Alas, in general long-range correlations do occur.
Nevertheless, we will prove that the Boltzmann distribution still decomposes into a convex combination of relatively few `Bethe states', characterized by Belief Propagation fixed points.
To be precise, suppose that $\emptyset \neq S\subset\Omega^{V_n}$ is an event.
Let $v$ be a variable node and let $a\in\partial_{\G} v$.
Then we define $\mu_{\G,v\to a}(\nix|S)$ as the conditional marginal of $v$ given $S$ under the Boltzmann distribution of the factor graph $\G-a$ obtained from $\G$ by removing the constraint node $a$.
In formulas, with $\scal\nix{\mu_{\G}(\nix|S)}$ denoting the expectation with respect to $\SIGMA$ drawn from $\mu_{\G}(\nix|S)$,
we have
	\begin{align*}
	\mu_{\G,v\to a}(\sigma\mid S)&=\frac{\scal{\vecone\{\SIGMA_v=\sigma\}/\psi_a(\SIGMA)}{\mu_{\G}(\nix|S)}}
		{\scal{1/\psi_a(\SIGMA)}{\mu_{\G}(\nix|S)}}&(\sigma\in\Omega).
	\end{align*}
Similarly, we let $\mu_{\G,a\to v}(\nix|S)$ be the conditional marginal of $v$ under the Boltzmann distribution of the factor graph obtained from $\G$ by removing all constraint nodes $b\in\partial_{\G} v\setminus a$ and disregarding the prior of $v$:
	\begin{align*}
	\mu_{\G,a\to v}(\sigma\mid S)&=\frac{\scal{\vecone\{\SIGMA_v=\sigma\}/(p(\sigma)\prod_{b\in\partial v\setminus a}\psi_b(\SIGMA))}{\mu_{\G}(\nix|S)}}
		{\scal{1/(p(\SIGMA_v)\prod_{b\in\partial v\setminus a}\psi_b(\SIGMA))}{\mu_{\G}(\nix|S)}}&(\sigma\in\Omega).
	\end{align*}
We refer to $\mu_{\G,v\to a}(\nix|S),\mu_{\G,a\to v}(\nix|S)$ as the {\em standard messages given $S$}.

\begin{definition}
Let $\eps>0$.
An event $S\subset\Omega^n$ is an {\em $\eps$-Bethe state} of $\G$ if the following two conditions hold.
\begin{description}
\item[BS1] the standard messages given $S$ are an $\eps$-Belief Propagation fixed point.
\item[BS2] if $\ell,\ell'\leq 1/\eps$ and if $\vI\subset V_n$, $\vJ\subset F_m$ are independent uniformly random sets of sizes $|\vI|=\ell$, $|\vJ|=\ell'$, then
	for every $\sigma\in\Omega^{V_n}$ we have
	\begin{align}\nonumber
	\Erw\bigg|&\scal{\vecone\{\forall v\in\vI\cup\partial\vJ\cup\partial^2\vI:
		\SIGMA_v=\sigma_v\}}{\mu_{\G}(\nix|S)}\\
		&-\prod_{v\in\vI}\frac{p(\sigma_v)\prod_{a\in\partial v}\psi_a(\sigma)\prod_{w\in\partial a\setminus v}\mu_{w\to a}(\sigma_w|S)}
			{\sum_{\chi\in\Omega}p(\chi)\prod_{a\in\partial v}\sum_{\tau\in\Omega^{\partial a}}\psi_a(\tau)
				\prod_{w\in\partial a\setminus v}\mu_{w\to a}(\tau_w|S)}
		\cdot\prod_{a\in\partial \vJ}\frac{\psi_a(\sigma)\prod_{w\in\partial a}\mu_{w\to a}(\sigma_w|S)}
			{\sum_{\tau\in\Omega^{\partial a}}\psi_a(\tau)\prod_{w\in\partial a}\mu_{w\to a}(\tau_w|S)}
		\bigg|<\eps.
							\label{eqBetheState}
	\end{align}
\end{description}
\end{definition}

Thus, on a Bethe state the standard messages form an approximate Belief Propagation fixed point.
Furthermore, locally around a bunch of randomly chosen variable and constraint nodes the Boltzmann distribution is characterized by the standard messages.
In particular, setting $\ell=0$ and $\ell'=1$ in {\bf BS2}, we see that the conditional joint distribution $\mu_{\G,\partial a}(\nix|S)$ of the variables around a typical random constraint node $a$ reads 
	\begin{align}\label{eqBP_1}
	\mu_{\G,\partial a}(\sigma\mid S)&=\frac{\psi_a(\sigma)\prod_{w\in\partial a}\mu_{w\to a}(\sigma_w|S)}
			{\sum_{\tau\in\Omega^{\partial a}}\psi_a(\tau)\prod_{w\in\partial a}\mu_{w\to a}(\tau_w|S)}+O(\eps)
			&(\sigma\in\Omega^{\partial a}).
	\end{align}
Additionally, setting $\ell=1$ and $\ell'=0$, we find that the local distribution around a typical variable node $v$, i.e., the distribution $\mu_{\G,v\cup\partial^2v}(\nix|S)$ induced  on the second neighborhood of $v$, reads
	\begin{align}\label{eqBP_2}
	\mu_{\G,v\cup\partial^2v}(\sigma\mid S)&=\frac{p(\sigma_v)\prod_{a\in\partial v}\psi_a(\sigma)\prod_{w\in\partial a}\mu_{w\to a}(\sigma_w|S)}
			{\sum_{\chi\in\Omega}p(\chi)\prod_{a\in\partial v}\sum_{\tau\in\Omega^{\partial a}}\psi_a(\tau)
				\prod_{w\in\partial a}\mu_{w\to a}(\tau_w|S)}+O(\eps)
			&(\sigma\in\Omega^{v\cup\partial^2v}).
	\end{align}
Thus, for most variable nodes $v$ the conditional Boltzmann marginal $\mu_{\G,v}(\nix|S)$ satisfies
	\begin{align}\label{eqBP_3}
	\mu_{\G,v}(\sigma\mid S)&=\frac{p(\sigma)\prod_{a\in\partial v}\mu_{a\to v}(\sigma|S)}
			{\sum_{\chi\in\Omega}p(\chi)\prod_{a\in\partial v}\mu_{a\to v}(\chi|S)}+O(\eps)
			&(\sigma\in\Omega).
	\end{align}	
Apart from the conditioning on $S$, the  formulas (\ref{eqBP_1})--(\ref{eqBP_3}) coincide with the ones known in the acyclic case~\cite{MM}.

In addition,  (\ref{eqBetheState}) implies that if we pick a few variable and/or constraint nodes randomly, then the joint distribution of their neighborhoods approximately factorizes.
Applied to $\ell=2$, $\ell'=0$, this means that once we condition on $S$, the joint distribution of two randomly chosen variable nodes is close to a product distribution:
	\begin{align}
	\label{eqRepSym}
	\frac1{n^2}\sum_{1\leq i<j\leq n}\Erw\TV{\mu_{\G,v_i,x_j}(\nix|S)-\mu_{\G,v_i}(\nix|S)\tensor\mu_{\G,v_j}(\nix|S)}&=O(\eps);
	\end{align}
in statistical physics jargon, the conditional distribution $\mu_{\G}(\nix|S)$ is \textit{replica symmetric}.

Confirming the picture sketched by the cavity method and vindicating the use of Belief Propagation for the study of the Boltzmann distribution,
the following theorem shows that \whp\ the Boltzmann distribution of a random factor graph decomposes into a relatively small number of Bethe states.

\begin{theorem}\label{Thm_BP}
For any function $L=L(n)\to\infty$ there exists $\eps=\eps(n)\to0$ such that the following is true.
There exists a decomposition $S_0=S_0(\G),S_1=S_1(\G),\ldots,S_\ell=S_\ell(\G)$, $\ell=\ell(\G)\leq L$, of  $\Omega^n$ into non-empty sets such that $\mu_{\GG}(S_0)\leq\eps$ such that with high probability
$S_1,\ldots,S_\ell\subset\Omega^n$ are $\eps$-Bethe states. 
The same statement holds with $\G$ replaced by $\GG$.
\end{theorem}

An important feature of \Thm~\ref{Thm_BP} is that the upper bound $L$ on the size of the Bethe state decomposition can be an arbitrarily slowly growing function of $n$.
Thus, the Gibbs measure can generally be decomposed into relatively few Bethe states, within which long-range correlations are negligible and where short-range correlations are characterized by Belief Propagation.

\subsection{The free energy}
\label{secFreeEnergy}

Apart from the structure of the Boltzmann distribution, a second key challenge is the computation of the free energy.
More specifically, arguably the single most important quantity associated with a random factor graph model is the {\em free energy density}
	\begin{align}\label{eqlimElog}
	\lim_{n\to\infty}\frac1n\Erw\brk{\log Z(\G)}.
	\end{align}
Of course, it comes as no surprise that computing (\ref{eqlimElog}) generally poses a formidable challenge.
In fact, even the existence of the limit remains an unresolved problem in several interesting cases.

The next theorem provides a formula for (and {\em en passant} establishes the existence of) the limit~\eqref{eqlimElog} 
in terms of the Bethe state decomposition from \Thm~\ref{Thm_BP} for a broad class of models.
We merely require a certain `convexity condition'.
This condition can be stated neatly in terms of a space that resembles the graphon space from combinatorics~\cite{Lovasz}.
Specifically, let $\states$ be the space of all measurable maps $[0,1]^2\to\cP(\Omega)$ modulo equality (Lebesgue-)almost everywhere.
We call these maps {\em strong kernels}.
For $(s,x)\in[0,1]$ and $\mu\in\states$ we let $\mu_{s,x}\in\cP(\Omega)$ denote the function value of $\mu$ at $(s,x)$.
Further, for $\mu,\mu'\in\states$ we define the {\em cut distance}
	\begin{align}\label{eqDefCutMetric}
	\Cutm(\mu,\mu')&=\inf_{\varphi,\varphi'}\,\sup_{\substack{S,X\subset[0,1]\\\omega\in\Omega}}
		\abs{\int_S\int_X\mu_{s,x}(\omega)-\mu'_{\varphi(s),\varphi'(x)}(\omega)\dd x\,\dd s},
	\end{align}
where the infimum is over all  measurable $\varphi,\varphi':[0,1]\to[0,1]$ that preserve the Lebesgue measure
and where the supremum runs over all measurable $S,X\subset[0,1]$.
Strictly speaking, $\Cutm(\nix,\nix)$ is a pre-metric (as possibly $\Cutm(\mu,\nu)=0$ even though $\mu\neq\nu$). 
We therefore let $\meas$ be the metric space where any two $\mu,\nu$ with $\Cutm(\mu,\nu)=0$ are identified.
Then $\meas$ is a compact Polish space \cite{Janson}.
Additionally, we write $\Meas$ for the space of all probability distributions on $\meas$.

Crucially, the convexity assumption that we require comes solely in terms of the distribution $P$ on the set $\Psi$ of weight functions.
Namely, let $\vx=(\vx_i)_{i\geq1}$ be a sequence of independent uniformly random points in $[0,1]$, chosen independently of $\PSI\in\Psi$.
Writing $\Erw\brk\nix$ for the expectation on $\vx,\PSI$, we make the following assumption.

\medskip
\begin{tabular}{|c|}\hline
\parbox[c]{14cm}{
For all $\mu,\mu'\in\states$ and for every integer $\ell\ge1$,
		\begin{align*}
		\Erw&\brk{\bc{1-\sum_{\sigma\in\Omega^k}\PSI(\sigma)\int_0^1\prod_{i=1}^k \mu_{s,\vx_i}(\sigma_i)\dd s}^\ell}
			+(k-1)\Erw\brk{\bc{1-\sum_{\sigma\in\Omega^k}\PSI(\sigma)\int_0^1
				\prod_{i=1}^k\mu'_{s,\vx_i}(\sigma_i)\dd s}^\ell}\\
	&\qquad\geq
			\sum_{h=1}^k\Erw\brk{\bc{1-\sum_{\sigma\in\Omega^k}\PSI(\sigma)
				\int_0^1\mu_{s,\vx_h}(\sigma_h)
					\prod_{i\in[k]\setminus\cbc h}\mu'_{s,\vx_i}(\sigma_i)\dd s}^\ell}.
		\end{align*}
}\\\hline
\end{tabular}
\hfill{\bf (POS)}

\medskip\noindent
We will see in \Sec~\ref{Sec_app} that {\bf POS} is easily verified for several interesting models, including the spin glass model from \Sec~\ref{Sec_intro}.

To obtain the formula for the free energy, we will represent the Bethe state decomposition of the random factor graph by a point in $\meas$.
Specifically, let $\vX,\vY$ be random variables with distribution $\Po(\omega)$ for an integer $\omega>0$, mutually independent and independent of $\G$.
Then with $S_1,\ldots,S_\ell$ the decomposition promised by \Thm~\ref{Thm_BP} we introduce for $i=1,\ldots,\ell$,
	\begin{align}\label{eqzGi}
	\check\vz_{\G,i}=\mu_{\G}(S_i)&\cdot
			\prod_{i=1}^{\vX}\bc{\sum_{\chi\in\Omega}p(\chi)\prod_{a\in\partial v_i}\sum_{\tau\in\Omega^{\partial a}}\vecone\{\tau_{v_i}=\chi\}\psi_a(\tau)\prod_{w\in\partial a\setminus v_i}\mu_{w\to a}(\tau_w|S_i)}^{-1}
			\\
&\cdot\prod_{i=1}^{\vY}\bc{\sum_{\tau\in\Omega^{\partial a_i}}\psi_{a_i}(\tau)\prod_{w\in\partial a_i}\mu_{w\to a_i}(\tau_w|S_i)}^{-1},\nonumber
	\end{align}
and we let $\check\vz_{\G}=\sum_{i=1}^\ell\check\vz_{\G,i}$.
It will emerge that combinatorially $\check\vz_{\G,i}/\check\vz_{\G}$ represents the probability mass of the Bethe state $S_i$ in the factor graph $\G'$ where we remove the first $\vY$ constraint nodes $a_1,\ldots,a_{\vY}$ as well as the first $\vX$ variable nodes $v_1,\ldots,v_{\vX}$ along with their adjacent constraint nodes.
While this removal operation has no discernible impact on the free energy (so long as $\omega=o(n)$), it enables us to set up a recurrence for computing this quantity.

The recurrence comes in terms of the messages sent out by those variable nodes that are left with degree $d-1$ after the removal operation.
We thus set up a kernel that captures these messages.
Specifically, let $v_{h_1},\ldots,v_{h_t}$ be the variable nodes of degree $d-1$ in the factor graph $\G'$ and let
$b_{1},\ldots,b_{t}$ be their $\G$-neighbors that got deleted.
Then we define the kernel $\check\mu_{\G,X,Y}:[0,1]^2\to\cP(\Omega)$ by letting
	\begin{align}\label{eqmuGXY}
	\check\mu_{\G,X,Y}:&(s,x)\mapsto\sum_{i=1}^t\sum_{j=1}^\ell
			\vecone\cbc{t-1\leq x<t,\,\sum_{h<j}\check\vz_{\G,h}<s \vz_{\G}\leq \sum_{h<j}\check\vz_{\G,h}}
			\mu_{\G,v_{h_i}\to b_i}(\nix|S_j).
	\end{align}
Recalling that $\G,\vX,\vY$ are random, we write $\check\pi_{n,\omega}\in\Meas$ for the distribution of $\check\mu_{\G,\vX,\vY}$.
Analogously, we write $\check\pi_{n,\omega,\cS}$ for the distribution of $\check\mu_{\GG,\vX,\vY}$ defined for the simple random factor graph.

Finally, we introduce a functional on the space $\Meas$ that 
encodes the recurrence for computing the free energy from the Bethe state decomposition.
Namely, let $(\vx_{i,j})_{i,j\geq1}$ be a family of random variables that are uniform on $[0,1]$,
let $(\vh_i)_{i\geq1}$ be a family of random variables that are uniform on $\{1,\ldots,k\}$,
let $(\PSI_i)_{i\geq1}$ be a sequence of samples from $P$, and let $\MU^{\pi}\in\meas$ be a sample from $\pi\in\Meas$, all mutually independent; then
	\begin{align}
	\label{eqBetheFunctional}
	\cB(\pi)&=\Erw\Big[
		\ln\int_0^1\sum_{\sigma\in\Omega}p(\sigma)\prod_{i=1}^{d}
		\sum_{\substack{\tau\in\Omega^k:\\\tau_{\vh_i}=\sigma}}\PSI_{i}(\TAU)\prod_{j\neq\vh_i}\MU^\pi_{s,\vx_{i,j}}(\tau_j)\dd s
	-d(1-k^{-1})\ln\int_0^1\sum_{\tau\in\Omega^k}\PSI_1(\tau)\prod_{j=1}^k\MU^\pi_{s,\vx_{1,j}}(\tau_j)\dd s
	\Big].
	\end{align}
We obtain the following expression for the free energy.

\begin{theorem}\label{thmBethePlus}
Assume that condition {\bf POS} is satisfied. Then
\begin{align*}
\lim_{n\to\infty}\frac1n\Erw\brk{\ln Z(\G)} &= \liminf_{\omega\to\infty}\,\liminf_{n \to \infty} \cB(\check \pi_{n,\omega}),
	&
\lim_{n\to\infty}\frac1n\Erw\brk{\ln Z(\GG)} &= \liminf_{\omega\to\infty}\,\liminf_{n \to \infty} \cB(\check \pi_{n,\omega,\cS}).
\end{align*}
\end{theorem}

\noindent
In particular, the limit on the left hand side exists, and it can be computed from the Bethe state decomposition.

\subsection{A variational formula}
\label{secSelfContain}
We proceed to state a variational formula for the free energy of the random factor graph models akin to the one from \Thm~\ref{Thm_sgvariational} for the spin glass model.
Namely, we express the limit (\ref{eqlimElog}) variationally as the infimum of $\cB(\pi)$ over $\pi$ chosen from a certain subspace $\MEAS\subset\Meas$.
The definition of $\MEAS$ is an adaptation to the Bethe lattice of the invariance property that Panchenko~\cite{Panchenko} put forward in the case of the \Erdos-\Renyi\ model.

To define the subspace $\MEAS$ let $\mu\in\states$, let $s\in[0,1]$ and let $N,M\geq0$ be integers.
We introduce the random variable
	\begin{align}\label{eqZNMmus}
	\vz(s)=&\prod_{i=1}^N\brk{\sum_{\sigma\in\Omega}p(\sigma)\prod_{j=1}^d\sum_{\tau\in\Omega^k}\vecone\{\tau_{\vh_i}=\sigma\}
		\PSI_{di+j}(\tau)\prod_{h\neq\vh_i}\mu_{s,\vx_{k(di+j)+h}}(\tau_h)}\\
		&\cdot\prod_{i=1}^M\brk{\sum_{\tau\in\Omega^k}\PSI_{dN+i}(\tau)\prod_{j=1}^k\mu_{s,\vx_{dk(N+1)+j}}(\tau_j)}.\nonumber
	\end{align}
Further, let
	\begin{align}\label{eqSelfCont}
	\vt&=\vt(s)=\inf\cbc{\theta\in[0,1]:\int_0^\theta \vz(u)\dd u\geq s\int_0^1 \vz(u)\dd u}
&\mbox{and}\qquad
		\mu^{\Join(N,M)}_{s,x}=\mu_{\vt,x}.
	\end{align}
Thus, for each $\mu\in\meas$ we obtain a random $\mu^{\Join(N,M)}\in\meas$.
Further, given $\pi\in\Meas$ we can apply this operation to a randomly chosen kernel $\MU^\pi\in\meas$,
thus obtaining a random kernel $\MU^{\pi\Join(N,M)}$.
We denote the distribution of $\MU^{\pi\Join(N,M)}$ by $\pi^{\Join(N,M)}$.
Now, let $\MEAS$ be the set of all densities $\pi\in\Meas$ such that $\pi^{\Join(N,M)}=\pi$ for all $N,M\geq0$.
Then we obtain the following self-contained formula for the free energy.

\begin{theorem}\label{Thm_freeEng}
Assume that {\bf POS} holds.
Then
	$$\lim_{n\to\infty}\frac1n\Erw\brk{\ln Z(\G)}=
	\lim_{n\to\infty}\frac1n\Erw\brk{\ln Z(\GG)}=\min_{\pi\in\MEAS}\cB(\pi).$$
\end{theorem}

Admittedly, the variational formula may not be easy to evaluate.
But \Thm~\ref{Thm_freeEng} places a lid on the complexity of the problem, and \Thm~\ref{thmBethePlus} provides an explicit combinatorial interpretation of the minimizer in terms of Belief Propagation fixed points and Bethe states.

\subsection{Discussion and related work}\label{Sec_related}
Over the past two decades an enormous amount of research, based on both rigorous and non-rigorous techniques, has been devoted to random factor graph models.
Much of this work has been sparked by the cavity method advanced in the original contribution of \Mezard\ and Parisi~\cite{MP1}.
A survey of this literature up until about 2008 can be found in~\cite{MM}.
More recently models of Bayesian inference problems such as the stochastic block model have received a great deal of attention as well; this literature is surveyed in~\cite{AbbeSurvey,Cris,LF}.

Rigorous work on random factor graphs and the cavity method can broadly be split into two categories.
First, contributions that investigate physics predictions on specific models.
Many of these contributions,  particularly the earlier ones, rely on `classical' techniques such as the second moment method, albeit frequently with physics-inspired twists.
Examples include work on the $k$-SAT threshold~\cite{nae,yuval,Catching,KostaSAT,DSS1}, which culminated in the proof of the $k$-SAT threshold conjecture for large $k$~\cite{DSS3}, the Potts model and the random graph coloring problem~\cite{AchNaor,Cond,ColReg,CDGS} or the hard-core model~\cite{DaniMoore,DSSind}. 
Some recent work is based on the powerful but technically demanding idea of `spatial coupling', which has led to important results in, e.g., coding theory~\cite{GMU} and random constraint satisfaction problems~\cite{DimitrisSpatial}.
A second line of work focused on the mathematical vindication of the cavity method in general, with applications to specific models of interest.
Examples include work on the role of spatial mixing~\cite{DemboBrazil,dembo}, the use of the interpolation method~\cite{bayati,PanchenkoTalagrand}, phase transitions in inference problems~\cite{Jean,CKPZ},
and contributions based on the asymptotic analysis of the Boltzmann distribution such as the influential work of Panchenko~\cite{Panchenko} as well as~\cite{Victor,Bethe}.
The present paper belongs to this second category.

In the following we discuss the main results and methods of the paper and how they compare to prior mathematical research.
Subsequently we compare the present work with the physics intuition and discuss directions for future research.

\subsubsection*{Mathematical work}
We regard \Thm~\ref{Thm_BP} as the main result of the paper.
The theorem confirms in great generality one of the key assumptions behind the cavity method and explains the success of Belief Propagation as a device for analyzing random regular factor graph models.
Indeed, the existence of a Bethe state decomposition has been conjectured explicitly, e.g., by \Mezard\ and Montanari~\cite[\Chap~19]{MM}; see also Dembo and Montanari~\cite{DemboBrazil}.

In a prior paper~\cite{Bethe} we constructed a Bethe state decomposition for random factor graph models of \Erdos-\Renyi\ type,
where the constraint nodes independently choose $k$-tuples of adjacent variable nodes.
While we will be able to use some of the general tools developed in that work, the main argument breaks in the case of the Bethe lattice due to its rigid geometry.
Indeed, the construction of the Bethe state decomposition hinges on coupling arguments involving, e.g., a coupling of a factor graph with $n$ variable and $m$ constraint nodes and another one with parameters $n'$ and $m'$ such that $n=n'+O(1)$, $m=m'+O(1)$.
Due to the Poisson degree distribution and the Stein-Chen property, such arguments are pretty straightforward in the \Erdos-\Renyi\ case.
One might say that the \Erdos-\Renyi\ graph resembles a gentle climbing wall  with footholds supplied by the irregularity of the Poisson degree distribution.
By contrast, the Bethe lattice with its regular degree makes for a smooth cliff.
As a consequence, the Bethe lattice requires new ideas, leading to a rather subtle but ultimately elegant argument.
The upshot is that this proof, which we present in \Sec~\ref{Sec_Bethe}, can be expected to generalize to other random graph models with given degrees.
Apart from the appeal of such lattice-like models from a physics perspective, these models play a vital role, e.g., in coding theory, where a suitably chosen degree sequence is apt to greatly boost performance~\cite{RichardsonUrbanke}.

Similarly, the variational formula for the free energy provided by \Thm~\ref{Thm_freeEng} is a generalization and adaptation of the formula established by Panchenko~\cite{Panchenko} for models of \Erdos-\Renyi\ type with spins $\Omega=\{\pm1\}$.
Panchenko's proof relies on two ingredients: an interpolation argument and a coupling argument.
So does ours.
But while the interpolation argument, an adaptation of the technique of Franz and Leone~\cite{FranzLeone}, goes through without too much trouble, the coupling argument does not.
Once more the rigidity of the Bethe lattice poses substantial challenges that require subtle new arguments.
A further, albeit relatively minor extension is that the present work applies to relatively general models with two or more spins, subject only to the condition {\bf POS}.
A further similarity between Panchenko's work and ours is the embedding of discrete Boltzmann distributions into a compact metric space, which enables us to pick convergent subsequences.
While Panchenko resorts to the Aldous-Hoover representation, here we use the cut metric and the associated kernel space, which is convenient to link the combinatorial representation of the measures in terms of messages directly with the free energy formula.
That the Aldous-Hoover representation is closely related to graph limits is, of course, a well known fact~\cite{DJ}.

Furthermore, Bayati, Gamarnik and Tetali~\cite{bayati} applied the interpolation method to factor graph models, including ones with regular degrees, to establish the {\em existence} of the limit $\lim_{n\to\infty}\frac1n\Erw[\ln Z(\G)]$ in certain cases via a super-additivity argument.
In the process they also used arguments based on `cavities', i.e., the removal of a small but linear number of vertices from the graph; a similar trick was used in~\cite{BP} as well.
But here, particularly in the construction of the Bethe state decomposition, we need to tread much more carefully.
In particular, while removal of a small linear number of vertices does not shift the free energy too much, here we can only afford the creation of a very small number of cavities in order to avoid a distortion of the Boltzmann distribution, an extremely volatile object.

\Thm~\ref{thmBethePlus}, which expresses the free energy density in terms of the Bethe state decomposition, is a synthesis of \Thm s~\ref{Thm_BP} and~\ref{Thm_freeEng}.
The proof shows that the free energy can be expressed in terms of a particular distribution on kernels $[0,1]^2\to\cP(\Omega)$, namely the one that encodes the Bethe state decomposition of the random factor graph or, more specifically, the associated Belief Propagation messages.
No corresponding result was previously known even in the conceptually simpler \Erdos-\Renyi\ case.

Apart from the interpolation method and coupling arguments, the proofs of \Thm s~\ref{Thm_BP}--\ref{Thm_freeEng} rely on some of the techniques that we developed in~\cite{Victor,CKPZ,Limits,Bethe}, particularly the cut metric and its ramifications.
The cut metric, which we apply to kernel representations of probability distributions, was originally developed in the context of the regularity method~\cite{FriezeKannan} and the theory of graph limits in combinatorics~\cite{Lovasz}.
Here we use the cut metric and certain assorted results, such as the `pinning lemma' from~\cite{CKPZ} (\Lem~\ref{Lemma_pinning} below) from~\cite{CKPZ} as tools, e.g., in the construction of the Bethe state decomposition.

While the present paper is concerned with diluted models where each node has a (fixed) bounded number of neighbors, there is also a substantial literature on fully connected models.
The prime example, of course, is the Sherrington-Kirkpatrick model.
The monographs of Panchenko~\cite{PanchenkoBook} and Talagrand~\cite{Tal1} provide an overview of this literature.
 In particular, the TAP equations, the (simplified) fixed point equations that correspond to the Belief Propagation equations in the fully connected case, have been established in several cases~\cite{Auffinger}.

In \Sec~\ref{Sec_app} we work out several application of the general results to specific models, such as the spin glass model from \Sec~\ref{Sec_intro}.
Pointers to related work on the specific problems can be found there.

\subsubsection*{The physics perspective}
The seminal work of \Mezard\ and Parisi~\cite{MP1} marks the starting point of a substantial body of physics work.
Highlights include the Survey Propagation algorithm and precise predictions on phase transitions, including satisfiability thresholds in combinatorial problems~\cite{pnas,MP2,MPZ}.

The results provided by \Thm~\ref{Thm_BP}--\ref{Thm_freeEng} are perfectly in line with the physics predictions.
But we should comment on a subtle point that is apt to cause confusion.
Namely, it has been pointed out that within the replica symmetric phase of certain models the support of the Boltzmann distribution may decompose into an exponentially large number of tiny `clusters'~\cite{MPZ,pnas}, a phenomenon called `dynamic replica symmetry breaking'.
Indeed, it has been conjectured that each of these tiny clusters induces a Bethe state~\cite{MM}; for the special case of the random graph coloring problem, this can be verified rigorously~\cite{Cond}.
At first glance this proliferation of Bethe states may appear to contradict \Thm~\ref{Thm_BP}, where the number of Bethe states is upper-bounded by an arbitrarily slowly growing function $L(n)$.
Yet the Bethe state decomposition is not unique, and despite the abundance of tiny clusters, $\mu_{\G}$ itself is replica symmetric (i.e., condition~\eqref{eqRepSym} holds for $S= \Omega^n$) throughout the dynamic RSB phase.
In effect, \Thm~\ref{Thm_BP} would render just a single Bethe state that comprises all of the tiny clusters.
By contrast, beyond the dynamic RSB phase, within the so-called condensed phase, \Thm~\ref{Thm_BP} would yield a non-trivial decomposition.
The existence of a condensed phase has been established rigorously in several examples~\cite{SoftCon,CKPZ}.

The variational formula for the free energy furnished by \Thm~\ref{Thm_freeEng} is in line with the physics work, which does, however, provide additional clues as to the structure of the minimizer of the functional $\cB(\nix)$.
Specifically, three different scenarios are expected to occur, depending on the model and the choice of its parameters.
First, the replica symmetric scenario with a single (or a  bounded number of) Bethe states. 
Second, the so-called `one-step replica symmetry breaking' scenario, where there are an unbounded number of `independent' Bethe states.
Third, the `full replica symmetry breaking'  scenario, where the Bethe states form a hierarchical structure; see~\cite{MM} for a detailed discussion.
Clearly, in order to better evaluate the variational formula it would be very valuable to establish this additional structural information  rigorously;  in the \Erdos-\Renyi\ case first attempts have been undertaken in~\cite{Panchenko2}.

\subsection{Organization}
In \Sec~\ref{Sec_prelim} we introduce the necessary pieces of notation and state some basic results that we will need.
Then in \Sec~\ref{Sec_cutm} we revisit the cut metric.
While much of what we need on this subject already appears in earlier papers, there are a few general preparations that we need to make and that we carry out in that section.
Subsequently \Sec~\ref{Sec_Bethe} deals with the proof of \Thm~\ref{Thm_BP}.
In \Sec s~\ref{secInterpolate} and~\ref{secAz} we then prove \Thm~\ref{Thm_freeEng} about the variational formula for the free energy.
\Sec~\ref{secAz} also contains the proof of \Thm~\ref{thmBethePlus}.
Finally, in \Sec~\ref{Sec_app} we work through a few applications, including the spin glass and hard-core models from \Sec~\ref{Sec_intro}.

\section{Preliminaries}\label{Sec_prelim}

\subsection{Basics}\label{Sec_basics}
For an integer $\ell\geq1$ we use the shorthand $[\ell]=\{1,\ldots,\ell\}$.
Furthermore, the symbols $O(\nix),\Omega(\nix),\ldots$ refer to the limit $n\to\infty$ by default.
To indicate asymptotics with respect to another variable $K$ tending to infinity, we write  $O_K(\nix),\Omega_K(\nix)$, etc.
Further, where set operations involve singletons, we usually omit braces.
For instance, if $x\in X$, then we just write $X\setminus x$ rather than $X\setminus\cbc x$.

For a finite set $\cX$ we let $\cP(\cX)$ be the set of all probability distributions on $\cX$, endowed with the total variation distance.
More generally, if $(\cX,\fA)$ is a measurable space, then $\cP(\cX)=\cP(\cX,\fA)$ denotes the set of all probability measures on this space.
Further, for probability measures $\pi,\pi'\in\cP(\cX)$ we let $\Gamma(\pi,\pi')$ be the set of all couplings of $\pi,\pi'$.
Thus, $\gamma\in\Gamma(\pi,\pi')$ is a probability distribution on $\cX\times\cX$ with marginals $\pi,\pi'$.

Suppose that $\cX$ is a finite set, that $n\geq1$ is an integer and that $\mu\in\cP(\cX^n)$.
Then we denote by $\SIGMA^\mu,\SIGMA^{1,\mu},\SIGMA^{2,\mu},\ldots$ a sequence of independent samples from $\mu$.
We omit the superscript $\mu$ where it is evident from the context.
Further, if $f:(\cX^n)^\ell\to\RR$ is a function, then we write
$\scal{f(\SIGMA^1,\ldots,\SIGMA^\ell)}\mu$ for the expectation of $f$ with respect to independent samples from $\mu$; thus,
	$$\scal{f(\SIGMA^1,\ldots,\SIGMA^\ell)}\mu=\sum_{\sigma^1,\ldots,\sigma^\ell\in\cX^n}
			f(\sigma^1,\ldots,\sigma^\ell)\prod_{i=1}^\ell\mu(\sigma^i).$$

Suppose that $\Omega,V\neq\emptyset$ are finite sets.
For a distribution $\mu\in\cP(\Omega^V)$ and an set $I\subset V$ we denote by $\mu_I$ the joint distribution of the coordinates $I$.
That is,
	\begin{align*}
	\mu_I(\sigma)&=\sum_{\tau\in\Omega^n}\vecone\{\forall i\in I:\tau_i=\sigma_i\}\mu(\tau)&(I\subset V,\sigma\in\Omega^I).
	\end{align*}
For $\sigma\in\Omega^I$ we use the shorthand $\mu(\sigma)=\mu_I(\sigma)$.
Moreover, if $I=\{i_1,\ldots,i_l\}$ we usually write $\mu_{i_1,\ldots,i_l}$ instead of $\mu_{\{i_1,\ldots,i_l\}}$.
Additionally, if $I\subset V$ and $\tau\in\Omega^V$, then we let $\tau_I=(\tau_i)_{i\in I}$ be the restriction of $\tau$ to $I$.

We keep the notation from \Sec~\ref{Sec_results};
in particular, $\Omega$ continues to denote a finite set of spins, $p$ is a probability distribution on $\Omega$, $\Psi$ is a measurable space of functions $\Omega^k\to(0,1)$, and $P$ is a probability distribution on $\Psi$.
In addition throughout  the paper we denote by
	$$\vx_i,\hat\vx_i,\vs_i,\vx_{i,j},\vx_{i,j}',\vx_{i,j}'',\hat\vx_{i,j}\qquad(i,j\geq1)$$
uniformly distributed random variables with values in $[0,1]$.
Additionally,
	$$\PSI,\PSI_i,\PSI_{i,j},\PSI_{i,j}',\PSI_{i,j}'',\hat\PSI_{i,j}\qquad(i,j\geq1)$$
denote elements of $\Psi$ drawn from the distribution $P$.
Further, 
	$$\vh_i,\vh_{i,j},\vh_i',\hat\vh_{i,j}\qquad(i,j\geq1)$$
are uniformly distributed random variables with values in $[k]$.
All of the above random variables are mutually independent as well as independently of any other sources of randomness.
These random variables yield random functions that will play an important role: for $i\geq1$ we let
	\begin{align*}
	\VARPHI_i&:\Omega^{dk}\to\RR,&\sigma\mapsto\sum_{\chi\in\Omega}p(\chi)\prod_{j=1}\vecone\{\sigma_{k(j-1)+\vh_{i,j}}=\chi\}\PSI_{i,j}(\sigma),\\
	\hat\VARPHI_i&:\Omega^{dk}\to\RR,&\sigma\mapsto\sum_{\chi\in\Omega}p(\chi)\prod_{j=1}\vecone\{\sigma_{k(j-1)+\hat\vh_{i,j}}=\chi\}\hat\PSI_{i,j}(\sigma).
	\end{align*}

\subsection{Factor graphs}

In \Sec~\ref{Sec_results} we already introduced the random factor graph model $\G(n,d,p,P)$.
To facilitate the proofs we need the following abstract definition.

\begin{definition}
Suppose that $(\cX,\fA)$ is a measurable space.
An {\em $\cX$-factor graph} $G=(V,F,(\partial a)_{a\in F},(\psi_a)_{a\in F},\pi)$ consists of
\begin{itemize}
\item a finite set $V$ of {\em variable nodes},
\item a finite set $F$ of {\em constraint nodes},
\item a set $\partial a\subset V$ for each $a\in F$,
\item a function $\psi_a:\cX^{\partial a}\to[0,\infty)$ for each $a\in F$ and
\item a probability measure $\prior$ on $\cX^V$, called the {\em prior}.
\end{itemize}
\end{definition}

A factor graph induces a bipartite graph on $V\cup F$, where $x\in V$ is adjacent to $a\in F$ iff $v\in\partial a$.
Accordingly, for a variable node $v$ we let $\partial v\subset F$ be the set of adjacent constraint nodes (i.e., $a\in\partial v$ iff $v\in\partial a$).
The bipartite graph defines a metric on $V\cup F$, the shortest path metric.
For a variable or constraint node $u$ we let $\partial^\ell u=\partial^\ell_Gu$ be the set of all nodes at distance precisely $\ell$ from $u$.
Moreover, $\nabla^\ell u=V\cap(u\cup\bigcup_{i\leq\ell}\partial^iu)$ denotes the set of all variable nodes at distance no more than $\ell$ from $u$.

Further, for an assignment $\sigma\in\Omega^V$ and $a\in F$ we use the notation $\psi_a(\sigma)=\psi_a(\sigma_{\partial a})$ and we define
	$$\psi_G(\sigma)=\prod_{a\in F}\psi_a(\sigma_{\partial a})\qquad\mbox{and}\qquad Z(G)=\int_{\cX^V}\psi_G(\sigma)\dd \prior(\sigma).$$
Providing that $Z(G)>0$, we introduce a probability measure $\mu_G$ on $\cX^V$, the {\em Boltzmann distribution}, by letting
	$$\dd\mu_G(\sigma)=\frac{\psi_G(\sigma)}{Z(G)}\dd\prior(\sigma).$$

Mostly the factor graphs that we deal with will have a finite space $\Omega$ and the prior $\prior$ will be the product measure $\prior=\bigotimes_{v\in V}p_v$.
In this case we introduce the standard messages given an event $S\subset\Omega^V$  as in \Sec~\ref{Sec_results}:
for a constraint node $a$ and $v\in\partial a$ we let
	\begin{align}\label{eqstdva}
	\mu_{G,v\to a}(\sigma\mid S)&=\frac{\scal{\vecone\{\SIGMA_v=\sigma\}/\psi_a(\SIGMA)}{\mu_{G}(\nix|S)}}
		{\scal{1/\psi_a(\SIGMA)}{\mu_{G}(\nix|S)}}&(\sigma\in\Omega),\\
	\mu_{G,a\to v}(\sigma\mid S)&=\frac{\scal{\vecone\{\SIGMA_v=\sigma\}/(p_v(\sigma)\prod_{b\in\partial v\setminus a}\psi_b(\SIGMA))}{\mu_{G}(\nix|S)}}
		{\scal{1/(p_v(\SIGMA_v)\prod_{b\in\partial v\setminus a}\psi_b(\SIGMA))}{\mu_{G}(\nix|S)}}&(\sigma\in\Omega).\label{eqstdav}
	\end{align}
In the case that $S=\Omega^V$ is the entire phase space, we omit the conditioning from the notation and just write
$\mu_{G,v\to a}$ and $\mu_{G,a\to v}$, respectively.

\subsection{The cut metric revisited}\label{Sec_cutm}

The cut metric, defined in (\ref{eqDefCutMetric}), plays a key role in the proofs of the main results.
In this section we summarize a few basic facts about the cut metric.
Although some have been proved in prior work, we will need to provide a few extensions and adaptations for our purposes.
In addition to the continuous version from (\ref{eqDefCutMetric}), we also need a discrete version of the cut metric, which we present in \Sec~\ref{Sec_cut_disc}.

\subsubsection{The continuous cut metric}
We remember that $\states$ denotes the space of all measurable maps $[0,1]^2\to\cP(\Omega)$, up to equality almost everywhere; we call such maps {\em strong kernels}.
The cut distance (\ref{eqDefCutMetric}) induces a pre-metric on this space~\cite{Janson}.
Moreover, on the space $\meas$ obtained by identifying points at cut distance $0$ the cut distance yields a metric.
The elements of $\meas$ are called {\em weak kernels}.
We drop the attribute and just speak of kernels where there is no danger of confusion.

\begin{proposition}[{\cite{Limits}}]\label{Prop_compact}
Endowed with the cut distance $\meas$ is a compact Polish space.
\end{proposition}

We continue to write $\Meas$ for the space of all probability measures on $\meas$.
This space is endowed with the weak topology.
Since $\meas$ is a compact Polish space, so is $\Meas$.
Hence, there is a natural metric on $\Meas$ that induces the weak topology, the $L_1$-Wasserstein metric.
We take license to denote this metric by $\Cutm(\nix,\nix)$ as well.
Thus, recalling that $\Gamma(\pi,\pi')$ is the set of all couplings of $\pi,\pi'\in\Meas$, we have
	\begin{align*}
\Cutm(\pi,\pi')&=\inf\cbc{\int_{\meas\times\meas}\Cutm(\mu,\mu')\,\dd\gamma(\mu,\mu'):
		\gamma\in\Gamma(\pi,\pi')}.
	\end{align*}
For $\pi\in\Meas$ we let $\MU^\pi\in\meas$ denote a sample.
We just write $\MU$ where $\pi$ is apparent.

By comparison to other metrics on the space of measurable functions $[0,1]^2\to\cP(\Omega)$ the cut metric is extremely weak; this is highlighted by the compactness of the space $\meas$ provided by \Prop~\ref{Prop_compact}.
Yet the cut metric is sufficiently strong to ensure that certain functions that will be of vital interest to us are continuous.
Indeed, suppose that $m,n>0$ are integers and that $f:\Omega^{m\times n}\to\RR$ is a function.
Then for $\mu\in\states$ we define the random variable
	\begin{align*}
	\scal f\mu&=\sum_{\sigma\in\Omega^{m\times n}}f(\sigma)
		\int_0^1\cdots\int_0^1\prod_{i=1}^m\prod_{j=1}^n\mu_{s_i,\vx_{j}}(\sigma_{i,j})\dd s_1\cdots\dd s_m.
	\end{align*}
Because we average out the $s_i$ and the $\vx_i$ are uniform, the random variables $\scal f\mu$ and $\scal f\nu$ are identically distributed if $\Cutm(\mu,\nu)=0$.
Thus, we may safely write $\scal f\mu$ for $\mu\in\meas$.

\begin{lemma}\label{Lemma_cntFct}
For any $f:\Omega^{m\times n}\to\RR$, $\ell\geq1$
 the map $\mu\in\meas\mapsto\Erw\brk{\scal f\mu^\ell}$ is continuous with respect to the cut metric.
\end{lemma}

\noindent
The proof of \Lem~\ref{Lemma_cntFct} can be found in the appendix.
For a probability distribution $\pi\in\Meas$ we let $\scal f\pi$ be the random variable $\scal f{\MU^{\pi}}$, with $\MU^{\pi}$ chosen independently of the $\vx_i$.
Since $\Meas$ carries the weak topology, \Lem~\ref{Lemma_cntFct} implies

\begin{corollary}\label{Cor_cntFct2}
For any $f:\Omega^{m\times n}\to\RR$, $\ell\geq1$ the map 
	$\pi\in\Meas\mapsto\Erw\brk{\scal f\pi^\ell}$ is continuous.
\end{corollary}

\noindent
We recall the functional $\cB(\nix)$ from \eqref{eqBetheFunctional}.

\begin{corollary}\label{Cor_Bcont}
The map	$\pi\in\Meas\mapsto\cB(\pi)$ is continuous.
\end{corollary}
\begin{proof}
Thanks to the tail bound (\ref{eqP}), we can approximate the logarithms in (\ref{eqBetheFunctional}) by polynomials.
Therefore, the assertion follows from \Cor~\ref{Cor_cntFct2}.
\end{proof}

The set $\MEAS$ of $\pi\in\Meas$ that are invariant under the $\Join(N,M)$-operation is a closed subset of $\Meas$.
To see this, and to interpret the $\Join(N,M)$-operation nicely in terms of operations that are continuous under the cut metric, we introduce the following general transformation.
Suppose that $f:\Omega^N\to(0,\infty)$ is a function and that $\mu\in\states$.
Then we define a random $f\Join\mu\in\states$ as follows.
Letting
	\begin{align*}
	\vz_s&=\vz_s^{f,\mu}=\sum_{\sigma\in\Omega^N}f(\sigma)\prod_{i=1}^N\mu_{s,\hat\vx_i}(\sigma_i),&
	\vz&=\vz^{f,\mu}=\int_0^1\vz_s^{f,\mu}\dd s,
	\end{align*}
we introduce
	\begin{align*}
	\vt&=\vt_s^{f,\mu}=\inf\cbc{u\in[0,1]:\int_0^u \vz_u\dd u\geq s\vz}.
	\end{align*}
Now, $f\Join\mu_{s,x}=\mu_{\vt,x}$.
We emphasize that $f\Join\mu\in\states$ is random, dependent on $\hat\vx_1,\ldots,\hat\vx_N$.
The kernel is characterized by the identity
	\begin{align*}
	\int_0^1\int_0^1 g_{s,x}\cdot f*\mu_{s,x}(\omega)\dd s\dd x&=
		\frac1{\vz}\sum_{\sigma\in\Omega^N}f(\sigma)\int_0^1\int_0^1g_{s,x}\cdot \mu_{s,x}(\omega)\vz_s\dd s\dd x
		\qquad\mbox{for all } g:[0,1]^2\to[0,1],\,\omega\in\Omega.
	\end{align*}
Further, since the $\hat\vx_i$ are uniform, we have $\Cutm(f\Join\mu,f\Join\nu)=0$ if $\Cutm(\mu,\nu)=0$.
Hence, the $\Join$-operation extends to weak kernels.
Furthermore, for a distribution $\pi$ we let $f\Join\pi$ be the distribution of $f\Join\MU^\pi$.

\begin{lemma}\label{Lemma_JoinCont}
For any function $f:\Omega^k\to(0,\infty)$ the map $\meas\to\Meas$, $\mu\mapsto f\Join\mu$ is continuous.
\end{lemma}

\noindent
The proof of \Lem~\ref{Lemma_JoinCont} can be found in the appendix.

The $\Join(N,M)$-operation is an application of the above $\Join$-operation to a particular random function $f$.
To define this random function, we need one more piece of notation.
Namely, for functions $f:\Omega^{M\times N}\to\RR$, $g:\Omega^{M\times L}\to\RR$ we define
	\begin{align*}
	f\oplus g&:\Omega^{M\times(N+L)}\to\RR,&
	\sigma&\mapsto f\bc{(\sigma_{i,j})_{i\in[M],j\in[N]}}\cdot g\bc{(\sigma_{i,j+N})_{i\in[M],j\in[L]}}.
	\end{align*}	
In words, we stick the first $N$ `columns' of $\sigma$ into $f$ and the last $M$ columns into $g$ and multiply the results.  Recalling the random functions $\hat\PSI_i$, $\hat\VARPHI_i$ from \Sec~\ref{Sec_basics}, we obtain the following.

\begin{lemma}\label{Lemma_JoinNM}
For any $\mu\in\meas$ the random $\mu^{\Join(N,M)}\in\meas$ is distributed as
	$\bc{\bigoplus_{i=1}^N\hat\VARPHI_i\oplus\bigoplus_{i=1}^M\hat\PSI_i}\Join\mu.$
\end{lemma}
\begin{proof}
This is immediate from the construction of $\mu^{\Join(N,M)}$.
\end{proof}

\begin{corollary}\label{Cor_JoinCont}
For any $N,M$ the map $\mu\in\meas\mapsto\mu^{\Join(N,M)}$ is continuous with respect to the cut metric.
\end{corollary}
\begin{proof}
Since $\Meas$ carries the weak topology, which is induced by the Wasserstein metric, 
the assertion follows from \Lem s~\ref{Lemma_JoinCont}--\ref{Lemma_JoinNM} and (\ref{eqP}).
\end{proof}

\noindent
As a further immediate consequence of \Lem~\ref{Lemma_JoinNM} we obtain

\begin{corollary}\label{Cor_JoinNM}
A distribution $\pi\in\Meas$ belongs to $\MEAS$ if and only if $\pi=\bc{\bigoplus_{i=1}^N\hat\VARPHI_i\oplus\bigoplus_{i=1}^M\hat\PSI_i}\Join\pi$ for all $N,M$.
\end{corollary}

\noindent
In particular, Corollaries~\ref{Cor_JoinCont} and~\ref{Cor_JoinNM} imply that the 
map $\pi\mapsto\pi^{*(N,M)}$ is continuous for all $N,M\geq0$.
Consequently,  $\MEAS$ is a closed subset of $\Meas$.

\subsubsection{The discrete version}\label{Sec_cut_disc}

Apart from the `continuous' installment of the cut metric, defined on kernels, we also need a discrete variant, defined on probability measures on discrete sets.
To be precise, with $\Omega\neq\emptyset$ our finite set of spins and $V$ another finite set of size $n\geq1$, 
we define a metric $\cutm(\nix,\nix)$ on $\cP(\Omega^V)$ as follows.
Recalling that $\Gamma(\mu,\nu)$ is the set of all couplings of probability measures $\mu,\nu$ on $\Omega^V$,
we let
	\begin{align}\label{eqCutm}
	\cutm(\mu,\nu)&=\frac1{n}\min_{\gamma\in\Gamma(\mu,\nu)}\max_{\substack{I\subset V\\
		B\subset\Omega^V\times\Omega^V\\\omega\in\Omega}}\abs{\sum_{i\in I}\sum_{(\sigma,\tau)\in B}\gamma(\sigma,\tau)
			(\vecone\{\sigma_i=\omega\}-\vecone\{\tau_i=\omega\})}\qquad\mbox{for }\mu,\nu\in\cP(\Omega^V).
	\end{align}

\begin{fact}[{\cite{Bethe}}]\label{Fact_metric}
$\cutm(\nix,\nix)$ is a metric on $\cP(\Omega^V)$.
\end{fact}

\noindent
We refer to $\cutm(\nix,\nix)$ as the {\em discrete cut metric}.

Suppose that $V$ is a finite set.
A measure $\mu\in\cP(\Omega^V)$ can be represented by a point $\dot\mu\in\states$.
Indeed, assume without loss that $V=[n]$ and that $\Omega=[q]$.
Then the set $\Omega^V$ can be ordered lexicographically as $\sigma^{(1)},\ldots,\sigma^{(q^n)}$.
We define $\dot\mu\in\states$ by letting
	\begin{align*}
	\dot\mu_{s,x}&=\sum_{i=1}^n\sum_{j=1}^{q^n}\vecone\{(i-1)/n\leq x<i/n\}
		\vecone\cbc{\sum_{h<j}\mu(\sigma^{(h)})\leq s<\sum_{h\leq j}\mu(\sigma^{(h)})}\delta_{\sigma^{(j)}_i}.
	\end{align*}
Comparing \eqref{eqCutm} with the definition~(\ref{eqDefCutMetric}) of the continuous cut metric, we see that
	\begin{align}\label{eqCutmcutm}
\Cutm(\dot\mu,\dot\nu)&\leq\cutm(\mu,\nu)&(\mu,\nu\in\Omega^V).
\end{align}

The discrete cut metric encodes a great deal of information about the discrete measures.
A particularly important case occurs when a measure $\mu\in\cP(\Omega^V)$ is close to a product measure.
To be precise, we say that $\mu$ is {\em $\eps$-extremal} if
	$\cutm(\mu,\bigotimes_{v\in V}\mu_v)<\eps$.
In words, $\mu$ is close to the product measure with the same marginals.
In addition, $\mu\in\cP(\Omega^V)$ is {\em $(\eps,\ell)$-symmetric} if 
	\begin{align}\label{eqepssym}
	\frac1{|V|^\ell}\sum_{v_1,\ldots,v_\ell\in V}\TV{\mu_{v_1,\ldots,v_\ell}-
		\mu_{v_1}\tensor\cdots\tensor\mu_{v_\ell}}&<\eps.
	\end{align}
Informally, if we choose $\ell$ coordinates randomly, then their joint distribution typically `nearly' factorizes.
The following statement shows that these concepts are essentially equivalent, up to a moderate loss in the parameters.

\begin{proposition}[\cite{Bethe}]\label{Cor_symmetry}
For any $\Omega$ of size $1<\abs\Omega<\infty$, any $0<\eps<1/2$ and  any $\ell\geq2$ there exists $n_0>0$ such that for all $n>n_0$ and all $\mu\in\cP(\Omega^V)$ the following two statements hold. 
\begin{enumerate}[(i)]
\item If $\mu$ is $(\eps/6)^3$-symmetric, then $\mu$ is $\eps$-extremal.
\item If  $\mu$ is $\eps^3/(128|\Omega|)^{4\ell}$-extremal, then $\mu$ is $(\eps,\ell)$-symmetric.
\end{enumerate}
\end{proposition}

It is an elementary observation that probability measures that are close in the discrete cut metric cannot have very different marginals.
Formally, we have the following.

\begin{lemma}\label{Lemma_tv1}
For any two probability measures $\mu,\nu$ on $\Omega^V$ we have
	$\sum_{v\in V}\tv{\mu_v-\nu_v}\leq2|\Omega|\cutm(\mu,\nu).$
\end{lemma}
\begin{proof}
There exists $\omega\in\Omega$ such that
	\begin{equation}\label{eqLemma_L1_1}
	\sum_{v\in V}(\mu_v(\omega)-\nu_v(\omega))\vee0\geq\frac1{2|\Omega|}\sum_{v\in V}\tv{\mu_v-\nu_v}.
	\end{equation}
Let $I=\cbc{i\in V:\mu_i(\omega)\geq\nu_i(\omega)}$ and $B=\Omega^V\times\Omega^V$.
Then for any coupling $\gamma\in\Gamma(\mu,\nu)$,
	\begin{align}\label{eqLemma_L1_2}
	\sum_{i\in I}\sum_{(\sigma,\tau)\in B}\gamma(\sigma,\tau)\bc{\vecone\{\sigma_i=\omega\}-\vecone\{\tau_i=\omega\}}
		&=\sum_{i\in V}(\mu_i(\omega)-\nu_i(\omega))\vee0.
	\end{align}
Combining (\ref{eqLemma_L1_1}) and (\ref{eqLemma_L1_2}) completes the proof.
\end{proof}

\noindent
The converse bound, that close marginals imply closeness in the cut metric, holds for extremal measures.

\begin{lemma}\label{Lemma_tv2}
For any two $\eps$-extremal $\mu,\nu\in\cP(\Omega^V)$ we have
	$\Cutm(\mu,\nu)\leq2\eps+\sum_{v\in V}\tv{\mu_v-\nu_v}.$
\end{lemma}
\begin{proof}
Assume without loss that $V=[n]$ and let $\bar\mu=\bigotimes_{i=1}^n\mu_i$ and $\bar\nu=\bigotimes_{i=1}^n\nu_i$.
Since $\mu,\nu$ are $\eps$-extremal, we have
	\begin{align}\label{eqLemma_tv2_1}
	\cutm(\mu,\bar\mu)&<\eps,&\cutm(\nu,\bar\nu)&<\eps.
	\end{align}
Let $\gamma_i\in\cP(\Omega\times\Omega)$ be an optimal coupling of $\mu_i,\nu_i$, i.e., $\tv{\mu_i-\nu_i}=\sum_{\sigma\neq\tau}\gamma_i(\sigma,\tau)$.
Then $\gamma=\bigotimes_{i=1}^n\gamma_i$ is a coupling of $\bar\mu,\bar\nu$.
Further, for any $I\subset[n],B\subset\Omega^n\times\Omega^n,\omega\in\Omega$ we have
	\begin{align*}
	\abs{\sum_{i\in I}\sum_{(\sigma,\tau)\in B}\gamma(\sigma,\tau)\bc{\vecone\{\sigma_i=\omega\}-\vecone\{\tau_i=\omega\}}}
		&\leq\sum_{i\in I}\sum_{(\sigma,\tau)\in B}\gamma(\sigma,\tau)\vecone\{\sigma_i\neq\tau_i\}\leq
			=\sum_{i=1}^n\tv{\mu_i-\nu_i}.
	\end{align*}
Hence, $\cutm(\bar\mu,\bar\nu)\leq\sum_{i=1}^n\tv{\mu_i-\nu_i}$, and thus   the assertion follows from (\ref{eqLemma_tv2_1}) and the triangle inequality.
\end{proof}

\noindent
We also make a note of the following enhanced triangle inequality.

\begin{lemma}[{\cite{Bethe}}]\label{Lemma_triangle}
Suppose that $\mu^{(1)},\nu^{(1)},\ldots,\mu^{(\ell)},\nu^{(\ell)}$ are probability measures on $\Omega^V$ and that $u_1,\ldots,u_\ell\geq0$ are numbers such that $\sum_{i=1}^\ell u_i=1$.
Then
	\begin{align*}
	\cutm\bc{\sum_{i=1}^\ell u_i\mu^{(i)},\sum_{i=1}^\ell u_i\nu^{(i)}}&\leq\sum_{i=1}^\ell u_i\cutm(\mu^{(i)},\nu^{(i)}).
	\end{align*}
\end{lemma}

Finally, we come to an important fact, intimately related to the \Szemeredi\ regularity lemma from combinatorics. Namely, any probability distribution $\mu\in\cP(\Omega^V)$ is close in the cut metric to a mixture of a `small' number of product measures.
To state this results precisely, suppose that $I\subset V$ and that $\sigma\in\Omega^I$.
Let
	$$S^{I,\sigma}=\cbc{\tau\in\Omega^V:\tau_I=\sigma}$$
be the  sub-cube of $\Omega^V$ where the entries of the coordinates in $I$ coincide with $\sigma$.
Further, assuming that $\mu\in\cP(\Omega^V)$ and $\mu(S^{I,\sigma})>0$, we let
	\begin{align}
	\label{eqCondMu}
\mu^{I,\sigma}=\mu[\nix|S^{I,\sigma}]
\end{align}
be the corresponding conditional distribution of $\mu$.
(If $\mu(S^{I,\sigma})=0$, then we agree that $\mu^{I,\sigma}$ is the uniform distribution on $S^{I,\sigma}$.)
The following key lemma shows that $\mu^{I,\sigma}$ is likely $\eps$-symmetric for suitably random $I,\sigma$.

\begin{lemma}[{\cite[\Lem~3.5]{CKPZ}}]\label{Lemma_pinning}
For any set $\Omega$ of size $1<|\Omega|<\infty$ and any $\eps>0$ there exist $n_0>0$ and a random variable $0<\THETA\leq2\eps^{-3}\ln|\Omega|$ such that for all $n>n_0$ and all $\mu\in\cP(\Omega^V)$ the following holds.
Let $\vI\subset V$ be a uniformly random subset of size $\THETA$ and choose $\SIGMA\in\Omega^{\vI}$ from $\mu_{\vI}$.
	Then  $\pr\brk{\mu^{\vI,\SIGMA}\mbox{ is $\eps$-symmetric}}>1-\eps.$
\end{lemma}

We can apply \Lem~\ref{Lemma_pinning} multiple times to obtain a decomposition of the set $\Omega^V$ into sub-cubes $S_1,\ldots,S_\ell$ such that $\mu[\nix|S_i]$ is $\eps$-symmetric.
To obtain these sub-cubes we just choose the set $\vI$ randomly as in \Lem~\ref{Lemma_pinning} and let $\sigma$ range over all $|\Omega|^{\THETA}$ possible assignments of $\vI$.
We then obtain the following version of the regularity lemma.

\begin{corollary}[\cite{Victor}]\label{Lemma_rl}
For any finite set $\Omega\neq\emptyset$ and any $\eps>0$ there exist $L,n_0$ such that for all $n>n_0$ the following is true.
For any $\mu\in\cP(\Omega^V)$ there exists a partition of $\Omega^V$ into pairwise disjoint sets $S_0,\ldots,S_\ell$, $\ell\leq L$, such that
 $\mu(S_0)<\eps$ and such that $\mu(\nix|S_i)$ is $\eps$-symmetric for each $1\leq i\leq\ell$.
\end{corollary}

\subsubsection{Contiguity}\label{Sec_contig}

Suppose that $\Omega\neq\emptyset$ is a finite set and let $c\geq1$.
A probability distribution $\nu$ on $\Omega^n$ is {\em $c$-contiguous} with respect to another probability distribution $\mu$ if 
	\begin{align*}
	\nu(\sigma)&\leq c\mu(\sigma)&\mbox{for all }\sigma\in\Omega^n.
	\end{align*}
Moreover, $\mu,\nu$ are {\em mutually $c$-contiguous} if each is $c$-contiguous with respect to the other.

\begin{lemma}\label{Lemma_extremal_contig}
For any $c\geq1,\delta>0$ there exists $\eps>0$ such that for all large enough $n$ the following is true.
Assume that $\mu\in\cP(\Omega^n)$ is $\eps$-extremal and that $\nu$ is $c$-contiguous with respect to $\mu$.
Then $\nu$ is $\delta$-extremal and $\cutm(\mu,\nu)<\delta$.
\end{lemma}
\begin{proof}
Choose $0<\eps\ll\eta\ll\zeta\ll\xi\ll\delta$ and assume that $n>n_0(\eps)$ is sufficiently large and that $\mu$ is $\eps$-extremal.  Applying \Cor~\ref{Lemma_rl} to the measure $\nu$, we obtain a partition $S_0,S_1,\ldots,S_\ell$ of the
cube $\Omega^n$ into pairwise disjoint sets such that $\nu(S_0)<\eta$ and such that $\nu(\nix|S_i)$ is $\eta$-symmetric for every $i=1,\ldots,\ell$.
Moreover, $\ell$ is bounded by a number $L(\eta,\Omega)>0$ that depends on $\eta$ and $|\Omega|$ only.

Suppose that for every $1\leq i\leq\ell$ with $\nu(S_i)\geq\eta/\ell$ we have
	\begin{equation}\label{eqLemma_extremal_contig1}
	\sum_{j=1}^n\tv{\nu_j(\nix|S_i)-\mu_j}\leq\zeta n.
	\end{equation}
Then \Lem~\ref{Lemma_tv2} yields $\cutm(\nu(\nix|S_i),\mu)\leq2\eta+\zeta$.
Hence, \Lem~\ref{Lemma_triangle} shows that
	\begin{equation}\label{eqLemma_extremal_contig2}
	\cutm(\nu,\mu)\leq4\eta+\zeta<\delta.
	\end{equation}
Further, (\ref{eqLemma_extremal_contig1}) implies that  $\sum_{j=1}^n\tv{\nu_j-\mu_j}\leq\zeta+2\eta$.
Hence, letting $\bar\mu=\bigotimes_{i=1}^n\mu_i$, $\bar\nu=\bigotimes_{i=1}^n\nu_i$ and applying \Lem~\ref{Lemma_tv2} a second time, we obtain $\cutm(\bar\mu,\bar\nu)\leq\zeta+2\eta$.
Thus, invoking the $\eps$-extremality of $\mu$ and (\ref{eqLemma_extremal_contig2}), we conclude that
	\begin{equation}\label{eqLemma_extremal_contig3}
	\cutm(\nu,\bar\nu)\leq\cutm(\nu,\mu)+\cutm(\mu,\bar\mu)+\cutm(\bar\mu,\bar\nu)\leq(4\eta+\zeta)+\eps+(\zeta+2\eta)<\delta.
	\end{equation}
In summary, if (\ref{eqLemma_extremal_contig1}) is satisfied, then (\ref{eqLemma_extremal_contig2}) and (\ref{eqLemma_extremal_contig3}) yield $\cutm(\nu,\mu)<\delta$ and $\cutm(\nu,\bar\nu)<\delta$, as claimed.

Thus, we are left to establish (\ref{eqLemma_extremal_contig1}).
Assume for contradiction that there is $1\leq i\leq\ell$ with $\nu(S_i)\geq\eta/\ell$ and $\sum_{j=1}^n\tv{\nu_j(\nix|S_i)-\mu_j}>\zeta n$.
Then there exist $J\subset[n]$ and $\omega\in\Omega$ such that $\sum_{j\in J}{\nu_j(\omega|S_i)-\mu_j(\omega)} >\zeta n/\bc{2|\Omega|}$.  In other words, the random variable $X(\sigma)=\sum_{j\in J}\vecone\{\sigma_j=\omega\}$ satisfies
	\begin{align}\label{eqLemma_extremal_contig1a}
	\scal{X}{\nu(\nix|S_i)}-\scal{X}{\mu}>\zeta n/\bc{2|\Omega|}.
	\end{align}
Due to the $\eta$-symmetry of $\nu(\nix|S_i)$ and the $\eps$-symmetry of $\mu$, the second moments work out as
	\begin{align}\label{eqLemma_extremal_contig1b}
	\scal{X(X-1)}{\nu(\nix|S_i)}&=\sum_{j,j'\in J:j\neq j'}\scal{\vecone\{\SIGMA_j=\SIGMA_{j'}=\omega\}}{\nu(\nix|S_i)}
		\leq\eta n+\scal{X}{\nu(\nix|S_i)}^2,&
	\scal{X(X-1)}{\mu}&\leq\eps n+\scal{X}{\mu}^2.
	\end{align}
Combining (\ref{eqLemma_extremal_contig1a}) and (\ref{eqLemma_extremal_contig1b}) with Chebyshev's inequality
and keeping in mind that $\eps\ll\eta\ll\zeta$, we conclude that the event
	$B=\cbc{X(\SIGMA)\geq\scal{X}\mu+\zeta/(4|\Omega|)}$
satisfies 
	\begin{align}\label{eqLemma_extremal_contig1c}
	\nu(B|S_i)&\geq3/4,&\mu(B)&\leq\eps^{1/4}.
	\end{align}
However, if $\nu$ is $c$-contiguous with respect to $\mu$, then \eqref{eqLemma_extremal_contig1c} yields
	\begin{align*}
	c\eps^{1/4}\geq c\mu(B)&\geq \nu(B)=\nu(B|S_i)\nu(S_i)\geq3\eta/(4\ell)\geq 3\eta/(4L(\eta,\Omega)),
	\end{align*}
which contradicts the choice of the parameters $\eps,\eta$.
\end{proof}

\begin{corollary}\label{Lemma_extremal_cond}
For any $\delta>0$ there exists $\eps>0$ such that the following is true.
Suppose that $\mu$ is $\eps$-extremal and that $S\subset\Omega^n$ is an event such that $\mu(S)\geq\delta$.
Then $\mu(\nix|S)$ is $\delta$-extremal and
	$\cutm(\mu(\nix|S),\mu)<\delta$.
\end{corollary}
\begin{proof}
Since $\mu(\sigma|S)\leq\mu(\sigma)/\mu(S)$ for every $\sigma$, the conditional distribution
$\mu(\nix|S)$ is $1/\eps$-contiguous with respect to $\mu$.
Thus, the assertion follows from \Lem~\ref{Lemma_extremal_contig} immediately.
\end{proof}

\section{Bethe state decompositions}\label{Sec_Bethe}

\subsection{The construction}\label{Sec_Proof_BP}
In this section we prove \Thm~\ref{Thm_BP}.
Specifically, we aim to show that the Boltzmann distribution $\mu_{\G}$ is well approximated by a collection of no more than $L$ Belief Propagation fixed points.
For a given variable node $v$ the corresponding fixed point equations involve the messages sent by the constraint $a\in\partial v$, which in turn are determined by the messages sent out by the variables $w$ at distance precisely two from $v$.
Thus, to express a single application of the Belief Propagation operator we require information about the variable nodes at distance two from $v$.
Therefore, in addition to the Boltzmann distribution $\mu_{\G}$ we will consider an enhanced measure $\hat\mu_{\G}$ that captures the joint distribution of the second neighborhoods

To be precise, let $G=(V,F,(\partial a)_{a\in F},(\psi_a)_{a\in F},p^{\tensor n})$ be a factor graph.
Then its Boltzmann distribution $\mu_G$ `lives' on the space $\Omega_G=\Omega^V$.
In addition, recalling that $\nabla^2_Gv\cap V$ consists of all variable nodes at distance at most two from $v$, consider the space
	$$\hat\Omega_G=\prod_{v\in V}\Omega^{\nabla^2_Gv}$$
of second neighborhood assignments, whose
 elements we denote as $\tau=(\tau(v,w))_{v\in V,w\in\nabla^2_Gv}$.
The factor graph $G$ induces an embedding
	\begin{align*}
	\Omega_G&\to\hat\Omega_G,&\sigma&\mapsto \hat\sigma=(\hat\sigma(v,w))_{v\in V,w\in\nabla^2_Gv},
		\qquad\mbox{where }\hat\sigma(x,y)=\sigma(y).
	\end{align*}
Thus,  $\mu_G$ induces a probability distribution $\hat\mu_G$ on $\hat\Omega_G$.
For a variable $v$ we denote by $\hat\mu_{G,v}\in\cP(\Omega^{\nabla^2_Gv})$ the marginal distribution of $\hat\mu_G$ on the $v$-factor of $\hat\Omega_G$.

The enhanced measure $\hat\mu_{\G}$ will play a vital role in the construction of the Bethe state decomposition.
Indeed, by comparison to the \Erdos-\Renyi\ case, the rigid geometry of the random regular graph causes significant difficulties.
More precisely, while the Belief Propagation messages are defined in terms of removing one or a few constraints, such operations clearly destroy regularity.
Hence, we need to create a bit of wiggling room.
To this end, we  remove some variable nodes along with their adjacent constraint nodes, thereby leaving a few variable nodes with degree $d-1$ rather than $d$.
We refer to these variables as `cavities'.
Clearly, this operation loses some information and would therefore by itself not suffice to prove \Thm~\ref{Thm_BP}.
However, what saves the day is that the enhanced measure $\hat\mu_{\G}$ contains the extra information needed to stitch the graph back up without losing track of the Bethe decomposition.

Unsurprisingly,	the construction is subtle and involves several steps.
It requires a number of carefully chosen parameters.
Specifically,  given a slowly diverging monotonically increasing positive integer sequence $L=L(n)\to\infty$ as in \Thm~\ref{Thm_BP}, we choose  a sequence $0<\xi=\xi(L)=o(1)$ that tends to zero monotonically sufficiently slowly, 
a further sequence $\omega=\omega(\xi)\to\infty$ that tends to infinity monotonically sufficiently slowly, as well as sequences
$0<\thet=\thet(\omega)=o(1)$, $0<\zeta=\zeta(\thet)=o(1)$,  $0<\beta=\beta(\thet)=o(1)$, 
$0<\alpha=\alpha(\beta)=o(1)$, $0<\eta=\eta(\alpha)=o(1)$ and
 $0<\eps=\eps(\eta)=o(1)$ that tend monotonically to zero slowly enough.
In summary, the pecking order reads
	\begin{align}\label{eqfunctions}
	1\ll1/\eps&\ll1/\eta\ll1/\alpha\ll1/\beta\ll1/\zeta\ll 1/\thet\ll \omega\ll 1/\xi\ll L\ll\ln\ln n,
	\end{align}
and we always assume tacitly that $n>n_0$ is sufficiently large.

We are ready to begin the construction.
Let $\G_*$ be the random factor graph obtained from $\G$ as follows.
Let $\THETA_*$ be a copy of the random variable $\THETA_\xi$ promised by \Lem~\ref{Lemma_pinning}; $\THETA_*$ is independent of $\G_*$.
Further, let $\vU_*$ be a random set of $\THETA_*$ variable nodes of $\G$
and draw $\SIGMA_*$ from $\mu_{\G}$ independently of $\THETA_*$ and $\vU_*$.
Now, obtain $\G_*$ from $\G$ by  changing the prior  distribution to
	\begin{align}\label{eqNewPrior}
	p_{\G_*}(\sigma)=\prod_{u\in\vU_*\cup\partial^2\vU_*}\vecone\{\sigma_u=\SIGMA_{*\,u}\}\prod_{v\not\in\vU_*\cup\partial^2\vU_*}p(\sigma_v).
	\end{align}
Additionally, let $\OMEGA$ be a random variable with distribution $\Po(\omega)\wedge2\omega$, independent of everything else, and let $\vW=\{v_{n-\OMEGA+1},\ldots,v_n\}$.
Finally, obtain $\G_*'$ from $\G_*$ by removing the variable nodes in $\vW$ along with their adjacent constraint nodes.

Thus, in $\G_*'$ we pin the spins of the variable nodes in $\vU_*$ and their neighbors to the values observed under $\SIGMA_*$, which is drawn from $\mu_{\G}$.  Additionally, we create cavities by removing the last $\OMEGA$ variable nodes along with their adjacent constraints.
The following lemma shows that the removal of the variable nodes in $\vW$ does not shift the marginals of the enhanced Boltzmann distribution much.

\begin{lemma}\label{Lemma_shift}
With probability at least $1-\omega^{-10}$ over the choice of $\THETA_*$, $\SIGMA_*$ and $\G$ the following statements are true.
\begin{enumerate}[(i)]
\item both $\mu_{\G_*}$ and $\hat\mu_{\G_*}$ are $\xi^{1/4}$-extremal.
\item we have
	$\sum_{v\in V_n\setminus(\vW\cup\partial^2\vW)}\tv{\hat\mu_{\G_*,v}-\hat\mu_{\G_*',v}}<\thet n.$
\end{enumerate}
\end{lemma}
\begin{proof}
By construction, $\hat\mu_{\G_*}$ is identical to the measure obtained through the pinning procedure of \Lem~\ref{Lemma_pinning} applied to the $\vU_*$-components of the space $\hat\Omega_{\G}$.
Hence, \Lem~\ref{Lemma_pinning} and \Prop~\ref{Cor_symmetry} imply that $\hat\mu_{\G_*}$ is $\xi^{1/4}$-extremal with probability at least $1-\xi^{1/4}$.
Since $\mu_{\G_*}$ is a projection of $\hat\mu_{\G_*}$, we obtain (i).

Further, let $\cV=V_n\setminus(\vW\cup\partial^2\vW)$.
If $\hat\mu_{\G_*}$ is $\xi^{1/4}$-extremal, then by the definition of the cut metric the distribution  $\hat\mu_{\G_*,\cV}$ induced on the neighborhoods of $\cV$ is $2\xi^{1/4}$-extremal, because $|\cV|\geq n/2$.
Additionally, there is $C=C(\omega)$ such that $\hat\mu_{\G_*,\cV}$ is $C$-contiguous with respect to $\hat\mu_{\G_*',\cV}$ with probability at least $1-\omega^{-11}$.
This follows from (\ref{eqP}), because $\G_*'$ is obtained from $\G_*$ by removing no more than $d\omega$ constraint nodes.
Therefore, (ii) follows from (i) and \Lem~\ref{Lemma_extremal_contig}, provided that $\xi,\omega,\thet$ are chosen appropriately in accordance with~\eqref{eqfunctions}.
\end{proof}

The following proposition, which establishes the Belief Propagation equations on $\G_*'$, constitutes the main technical step of the proof.

\begin{proposition}\label{Lemma_fixp}
With probability at least $1-\alpha^{9}$, $\G'_*$ enjoys the following properties.
\begin{enumerate}[(i)]
\item the standard messages $(\mu_{\G'_*,v\to a},\mu_{\G'_*,a\to v})_{v\in V(\G_*'),a\in\partial v}$ form an $\alpha^{9d}$-Belief Propagation fixed point.
\item we have 
	\begin{align*}
	\sum_{v\in V(\G'_*)}
		\sum_{\sigma\in\Omega^{\nabla^2v}}\abs{\mu_{\G_*'}(\sigma)
			-\frac{p(\sigma_v)\prod_{a\in\partial v}\psi_a(\sigma)\prod_{w\in\partial a}\mu_{\G_*',w\to a}(\sigma_w)}
			{\sum_{\chi\in\Omega}p(\chi)\prod_{a\in\partial v}\sum_{\tau\in\Omega^{\partial a}}\psi_a(\tau)
				\prod_{w\in\partial a}\mu_{\G_*',w\to a}(\tau_w)}
			}&<\alpha^{9d}n.
	\end{align*}
\end{enumerate}
\end{proposition}

\noindent
Before we prove \Prop~\ref{Lemma_fixp} in \Sec~\ref{Sec_Lemma_fixp}, let us indicate how the theorem follows.
As a final preparation we need the following basic fact.

\begin{lemma}\label{Lemma_localRemoval}
For any factor graph $G$, for any variable node $v$, any $S\subset\partial v$ and any $\sigma\in\Omega$ we have
	$$\mu_{G-S,v}(\sigma)=
		\frac{\scal{\vecone\{\SIGMA_v=\sigma\}/\prod_{a\in S}\psi_a(\SIGMA)}{\mu_{G,v\cup\partial^2v}}}{\scal{1/\prod_{a\in S}\psi_a(\SIGMA)}{\mu_{G,v\cup\partial^2v}}}.$$
\end{lemma}
\begin{proof}
The partition function works out to be
	\begin{align*}
	Z(G-S)&=\sum_{\sigma\in\Omega^{V(G)}}\prod_{a\in F(G)\setminus S}\psi_a(\sigma_{\partial a})
		=\sum_{\sigma\in\Omega^{V(G)}}\frac{\prod_{a\in F(G)}\psi_a(\sigma_{\partial a})}{\prod_{a\in S}\psi_a(\sigma_{\partial a})}=Z(G)\scal{1/\prod_{a\in S}\psi_a}{\mu_G}.
	\end{align*}
Hence, for any $\tau\in\Omega^{V(G)}$,
	\begin{align*}
	\mu_{G-S}(\tau)&=\frac1{Z(G-S)}\prod_{a\in F(G)\setminus S}\psi_a(\tau_{\partial a})
		=\frac{Z(G)}{Z(G-S)}\frac1{Z(G)}\frac{\prod_{a\in F(G)}\psi_a(\tau_{\partial a})}{\prod_{a\in S}\psi_a(\tau_{\partial a})}
		=\frac{\mu_G(\tau)}{\scal{1/\prod_{a\in S}\psi_a}{\mu_G}\prod_{a\in S}\psi_a(\tau_{\partial a})},
	\end{align*}
and the average in the denominator involves variables in $v\cup\partial^2v$ only.
\end{proof}

\begin{proof}[Proof of \Thm~\ref{Thm_BP}]
For any assignment $\chi\in\cX=\prod_{v\in\vU_*}\Omega^{\nabla^2v}$ of the variables in $\vU_*$ and their neighborhoods let
	$$S(\chi)=\{\sigma\in\hat\Omega_{\G}:\forall v\in\vU_*,\,w\in\nabla^2v:\sigma(v,w)=\chi(v,w)\}.$$
Then $(S(\chi))_{\chi\in\cX}$ is a decomposition of $\hat\Omega_{\G}$ into no more than $q^{(1+d(k-1))\THETA_*}$ sub-cubes, corresponding to the neighborhood assignments of the first $\THETA_*$ variable nodes.
As \Lem~\ref{Lemma_pinning} shows, 
by choosing the functions from (\ref{eqfunctions}) appropriately we can guarantee that $q^{(1+d(k-1))\THETA_*}\leq L$.
We are going to show that the decomposition $(S(\chi))_{\chi}$ meets the requirements of the theorem \whp\

For $\chi\in\cX$ let $\G_*[\chi]$ 
be the random factor graph $\G_*$ given that $\SIGMA_{*}(w)=\chi(v,w)$ for all $v\in\vU_*$ and all $w\in\nabla^2v$.
Also let $\G_*'[\chi]$ be the factor graph obtained from $\G_*[\chi]$ by removing the variables in $\vW$ along with their adjacent constraint nodes.
Further, let $\cE_\chi$ be the event that the following four conditions are satisfied.
\begin{description}
\item[E1] Both $\mu_{\G_*[\chi]}$ and $\hat\mu_{\G_*[\chi]}$ are $\xi^{1/8}$-extremal.
\item[E2] We have
	\begin{align}\label{eqiiib}
	\sum_{v\in V_n\setminus(\vW\cup\partial_2\vW)}\TV{\hat\mu_{\G_*[\chi],v}-\hat\mu_{\G_*'[\chi],v}}<\thet n.
	\end{align}
\item[E3] On $\G_*'[\chi]$ the standard messages form an $\alpha^{9d}$-BP fixed point and
	\begin{align}
	\sum_{v\not\in \vW\cup\partial^2\vW}
		\sum_{\sigma\in\Omega^{\nabla^2v\cap V_n}}\abs{\mu_{\G_*'[\chi]}(\sigma)
			-\frac{p(\sigma_v)\prod_{a\in\partial v}\psi_a(\sigma)\prod_{w\in\partial a}\mu_{\G_*'[\chi],w\to a}(\sigma_w)}
			{\sum_{\chi\in\Omega}p(\chi)\prod_{a\in\partial v}\sum_{\tau\in\Omega^{\partial a}}\psi_a(\tau)
				\prod_{w\in\partial a}\mu_{\G_*'[\chi],w\to a}(\tau_w)}
			}&<\alpha^{9d}n.\label{eqiiia}
	\end{align}	
\item[E4] There are no more than $\alpha^{10d}n$ constraint nodes $a$ in $\G$ such that $\min_{\sigma\in\Omega^k}\psi_a(\sigma)\leq\alpha^{1/4}$, nor are  there more than $\eps^{20}n$ constraint nodes $a$ such that $\min_{\sigma\in\Omega^k}\psi_a(\sigma)\leq\eps$.
\end{description}
Then (\ref{eqP}), \Lem~\ref{Lemma_shift} and \Prop~\ref{Lemma_fixp} yield $\Erw\brk{\sum_{\chi}\mu_{\G}(\chi)(1-\vecone\cE_\chi)}\leq\alpha^8$.
Thus, Markov's inequality shows
	\begin{align}\label{eqiiic}
	\pr\brk{\sum_{\chi}\mu_{\G}(\chi)(1-\vecone\cE_\chi)\geq\alpha^4}\leq\alpha^4.
	\end{align}

Thus, we are left to argue that $S(\chi)$ is an $\eps$-Bethe state of $\G$ if the event $\cE_\chi$ occurs.
As a first step, we are going to show that the standard messages of $\G_*'[\chi]$, $\G_*[\chi]$ are close:
given $\cE_\chi$, we claim
	\begin{align}\label{eqiiid}
	\sum_{\sigma\in\Omega}\sum_{v\not\in \vW\cup\partial_2\vW}\sum_{a\in\partial v}
		\abs{\mu_{\G_*[\chi],v\to a}(\sigma)-\mu_{\G_*'[\chi],v\to a}(\sigma)}+
		\abs{\mu_{\G_*[\chi],a\to v}(\sigma)-\mu_{\G_*'[\chi],a\to v}(\sigma)}&<\alpha^{8d}n.
	\end{align}
To see this, recall that $\mu_{\G_*[\chi],v\to a}$ is the marginal of $v$ in the factor graph $\G_*[\chi]-a$.
Hence, \Lem~\ref{Lemma_localRemoval} shows that
	\begin{align}\label{eqiiie}
	\mu_{\G_*[\chi],v\to a}(\sigma)&=
		\frac{\scal{\vecone\{\SIGMA_v=\sigma\}/\psi_a(\SIGMA)}{\mu_{\G_*[\chi],\nabla^2 v}}}{\scal{1/\psi_a(\SIGMA)}{\mu_{\G_*[\chi],\nabla^2 v}}},\\
	\label{eqiiif}
	\mu_{\G_*[\chi],a\to v}(\sigma)&=
		\frac{\scal{\vecone\{\SIGMA_v=\sigma\}/\prod_{b\in\partial v\setminus a}\psi_b(\SIGMA)}{\mu_{\G_*[\chi],\nabla^2 v}}}{\scal{1/\prod_{b\in\partial v\setminus a}\psi_b(\SIGMA)}{\mu_{\G_*[\chi],\nabla^2 v}}}.
		\end{align}
Providing  $\thet\ll\alpha^d$, we obtain (\ref{eqiiid}) from (\ref{eqiiib}), (\ref{eqiiie}), (\ref{eqiiif}) and {\bf E4}.
Further, 
the estimate (\ref{eqiiid}) and the fact that the standard messages of $\G_*'[\chi]$ are an $\alpha^{9d}$-BP fixed point imply that the standard messages of $\G_*[\chi]$ are an $\eta$-BP fixed point.
Thus, we have established {\bf BS1}.

In order to prove {\bf BS2}, we estimate the derivatives of a term like in \eqref{eqiiia} as follows:
	\begin{align*}
	\abs{\frac{\partial}{\partial \nu_w(\sigma)}
	\frac{p(\sigma_v)\prod_{a\in\partial v}\psi_a(\sigma)\prod_{w\in\partial a}\nu_{w}(\sigma(w))}
			{\sum_{\chi\in\Omega}p(\chi)\prod_{a\in\partial v}\sum_{\tau\in\Omega^{\partial a}}\psi_a(\tau)
				\prod_{w\in\partial a}\nu_{w}(\tau(w))}}
				&\leq\frac{1}{\min_{a\in\partial v,\tau\in\Omega^{\partial a}}\psi_a(\tau)^{2d}}.
	\end{align*}
Hence, \eqref{eqiiib}, (\ref{eqiiia}), \eqref{eqiiid} and {\bf E4} and the bound $|\vW\cup\partial^2\vW|=O(\ln n)$ yield
	\begin{align}\label{eqiiig}
	\sum_{v\in V_n}
		\sum_{\sigma\in\Omega^{\nabla^2v}}\abs{\mu_{\G_*[\chi]}(\sigma)
			-\frac{p(\sigma_v)\prod_{a\in\partial v}\psi_a(\sigma)\prod_{w\in\partial a}\mu_{\G_*[\chi],w\to a}(\sigma(w))}
			{\sum_{\chi\in\Omega}p(\chi)\prod_{a\in\partial v}\sum_{\tau\in\Omega^{\partial a}}\psi_a(\tau)
				\prod_{w\in\partial a}\mu_{\G_*[\chi],w\to a}(\tau(w))}
			}&<\alpha^{2d}n.
	\end{align}	

Additionally, we claim that
	\begin{align}\label{eqiiih}
	\sum_{b\in F_m}
		\sum_{\sigma\in\Omega^{\partial b}}\abs{\mu_{\G_*[\chi]}(\sigma)
			-\frac{\psi_b(\sigma)\prod_{w\in\partial b}\mu_{\G_*[\chi],w\to b}(\sigma(w))}
			{\sum_{\tau\in\Omega^{\partial b}}\psi_b(\tau)\prod_{w\in\partial b}\mu_{\G_*[\chi],w\to b}(\tau(w))}
			}&<\eps^9 n.
	\end{align}	
To see this, suppose that $b$ satisfies $\min_{\sigma\in\Omega^k}\psi_b(\sigma)\geq\eps$, that
$b\in\partial v$ for a variable node $v$ such that
	\begin{align}\label{eqiiii}
		\sum_{\sigma\in\Omega^{\nabla^2v}}\abs{\mu_{\G_*[\chi]}(\sigma)
			-\frac{p(\sigma_v)\prod_{a\in\partial v}\psi_a(\sigma)\prod_{w\in\partial a\setminus v}\mu_{\G_*[\chi],w\to a}(\sigma_w)}
			{\sum_{\kappa\in\Omega}p(\kappa)\prod_{a\in\partial v}\sum_{\tau\in\Omega^{\partial a}}\psi_a(\tau)\vecone\{\tau_v=\kappa\}
				\prod_{w\in\partial a\setminus v}\mu_{\G_*[\chi],w\to a}(\tau_w)}
			}&<\alpha^{d}
	\end{align}	
and that
	\begin{align}\label{eqiiij}
	\sum_{\sigma\in\Omega}\abs{\mu_{\G_*[\chi],v\to b}(\sigma)-\frac{p(\sigma)\prod_{a\in\partial v\setminus b}\mu_{\G_*[\chi],a\to v}(\sigma)}
		{\sum_{\kappa\in\Omega}p(\kappa)\prod_{a\in\partial v\setminus b}\mu_{\G_*[\chi],a\to v}(\kappa)}}&<\eta^{1/4},\\
	\sum_{a\in\partial v}\sum_{\sigma\in\Omega}\abs{\mu_{\G_*[\chi],a\to v}(\sigma)-
			\frac{\sum_{\tau\in\Omega^{\partial a}}\vecone\{\tau_v=\sigma\}\psi_a(\tau)\prod_{w\in\partial a\setminus v}\mu_{\G_*[\chi],w\to a}(\tau_w)}
				{\sum_{\tau\in\Omega^{\partial a}}\psi_a(\tau)\prod_{w\in\partial a\setminus v}\mu_{\G_*[\chi],w\to a}(\tau_w)}}&<\eta^{1/4}.\label{eqiiik}
	\end{align}	
All but $\eps^{10}n$ constraint nodes $b$ enjoy these properties, due to {\bf E4}, \eqref{eqiiig} and because the standard messages of $\G_*[\chi]$ form an $\eta$-BP fixed point.
For any such $b$ and any $\sigma\in\Omega^{\partial b}$ we obtain
	\begin{align*}
		&\mu_{\G_*[\chi],\partial b}(\sigma)=\sum_{\tau\in\Omega^{\nabla^2v}}\vecone\{\tau_{\partial b}=\sigma\}\mu_{\G_*[\chi],\nabla^2v}(\tau)\\
			&\quad\stacksign{\eqref{eqiiii}}=\sum_{\tau\in\Omega^{\nabla^2v}}
		\frac{\vecone\{\tau_{\partial b}=\sigma\}p(\sigma_v)\prod_{a\in\partial v}\psi_a(\tau)\prod_{w\in\partial a\setminus v}\mu_{\G_*[\chi],w\to a}(\tau_w)}
			{\sum_{\kappa\in\Omega}\prod_{a\in\partial v}\sum_{\tau'\in\Omega^{\partial a}}\psi_a(\tau')\vecone\{\tau'_v=\kappa\}
				\prod_{w\in\partial a\setminus v}\mu_{\G_*[\chi],w\to a}(\tau_w')}+O(\alpha^d)\\
	&\quad\stacksign{\eqref{eqiiik}}=\frac{p(\sigma_v)\psi_b(\sigma)\prod_{w\in\partial b\setminus v}\mu_{\G_*[\chi],w\to b}(\sigma_w)
			\prod_{a\in\partial v\setminus b}\mu_{\G_*[\chi],a\to v}(\sigma_v)\sum_{\tau\in\Omega^{\partial a}}\psi_a(\tau)\prod_{w\in\partial a\setminus v}\mu_{\G_*[\chi],w\to a}(\tau_w)}
		{\sum_{\kappa\in\Omega^{\partial b}}p(\kappa_v)\psi_b(\kappa)
				\prod_{w\in\partial b\setminus v}\mu_{\G_*[\chi],w\to b}(\kappa_w)
			\prod_{a\in\partial v\setminus b}\mu_{\G_*[\chi],a\to v}(\kappa)\sum_{\tau\in\Omega^{\partial a}}\psi_a(\tau)\prod_{w\in\partial a\setminus v}\mu_{\G_*[\chi],w\to a}(\tau_w)}\\
	&\qquad\qquad\qquad+O(\alpha^d)\\
	&\quad=\frac{p(\sigma_v)\psi_b(\sigma)\prod_{w\in\partial b\setminus v}\mu_{\G_*[\chi],w\to b}(\sigma_w)
			\prod_{a\in\partial v\setminus b}\mu_{\G_*[\chi],a\to v}(\sigma_v)}
		{\sum_{\kappa\in\Omega^{\partial b}}p(\kappa_v)\psi_b(\kappa)
				\prod_{w\in\partial b\setminus v}\mu_{\G_*[\chi],w\to b}(\kappa_w)
			\prod_{a\in\partial v\setminus b}\mu_{\G_*[\chi],a\to v}(\kappa)}+O(\alpha^d)\\
	&\quad\stacksign{\eqref{eqiiij}}=
		\frac{\psi_b(\sigma)\prod_{w\in\partial b}\mu_{\G_*[\chi],w\to b}(\sigma_w)}
				{\sum_{\kappa\in\Omega^{\partial b}}\psi_b(\kappa)\prod_{w\in\partial b}\mu_{\G_*[\chi],w\to b}(\kappa_w)}+O(\alpha^d),
	\end{align*}	
whence (\ref{eqiiih}) follows by averaging on $b$.

Finally, {\bf BS2} follows from (\ref{eqiiig}), (\ref{eqiiih}), the  $\xi^{1/8}$-extremality of $\hat\mu_{\G_*[\chi]}$.
Indeed, let $\vI,\vJ$ be random sets of at most $1/\eps$ variable/constraint nodes.
For each $b\in\vJ$ pick a variable node $v_b\in\partial b$.
Because $\hat\mu_{\G_*[\chi]}$ is $\xi^{1/8}$-extremal,  \Prop~\ref{Cor_symmetry} yields
	\begin{align}\label{eqBS2_1}
	\Erw\TV{\hat\mu_{\G_*[\chi],\vI\cup\{v_b:b\in\vJ\}}-\bigotimes_{v\in\vI\cup\{v_b:b\in\vJ\}}\hat\mu_{\G_*[\chi],v}}&<\alpha^4.
	\end{align}
Furthermore, (\ref{eqiiig}) and (\ref{eqiiih}) imply that with probability at least $1-\eps^2$ over the choice of $\vI,\vJ$ we have
	\begin{align*}
	\forall b\in\vJ&:
		\sum_{\sigma\in\Omega^{\partial b}}\abs{\mu_{\G_*[\chi]}(\sigma)
			-\frac{\psi_b(\sigma)\prod_{w\in\partial b}\mu_{\G_*[\chi],w\to b}(\sigma(w))}
			{\sum_{\tau\in\Omega^{\partial b}}\psi_b(\tau)\prod_{w\in\partial b}\mu_{\G_*[\chi],w\to b}(\tau(w))}
			}<\eps^3,\\
	\forall v\in\vI&:
		\sum_{\sigma\in\Omega^{\nabla^2v}}\abs{\mu_{\G_*[\chi]}(\sigma)
			-\frac{p(\sigma_v)\prod_{a\in\partial v}\psi_a(\sigma)\prod_{w\in\partial a}\mu_{\G_*[\chi],w\to a}(\sigma(w))}
			{\sum_{\chi\in\Omega}p(\chi)\prod_{a\in\partial v}\sum_{\tau\in\Omega^{\partial a}}\psi_a(\tau)
				\prod_{w\in\partial a}\mu_{\G_*[\chi],w\to a}(\tau(w))}
			}<\eps^3.
	\end{align*}	
If these estimates hold, then  for any configuration $\sigma\in\Omega^{\vI\cup\partial\vJ}$ we obtain
	\begin{align}							\label{eqBS2_2}
	\bigg|&\bigotimes_{v\in\vI\cup\{v_b:b\in\vJ\}}\hat\mu_{\G_*[\chi],v}(\sigma)\\
		&-\prod_{v\in\vI}\frac{p(\sigma_v)\prod_{a\in\partial v}\psi_a(\sigma)\prod_{w\in\partial a\setminus v}\mu_{\G_*[\chi],w\to a}(\sigma_w)}
			{\sum_{\chi\in\Omega}p(\chi)\prod_{a\in\partial v}\sum_{\tau\in\Omega^{\partial a}}\psi_a(\tau)
				\prod_{w\in\partial a\setminus v}\mu_{\G_*[\chi],w\to a}(\tau_w)}
		\cdot\prod_{a\in\vJ}\frac{\psi_a(\sigma)\prod_{w\in\partial a}\mu_{\G_*[\chi],w\to a}(\sigma_w)}
			{\sum_{\tau\in\Omega^{\partial a}}\psi_a(\tau)\prod_{w\in\partial a}\mu_{\G_*[\chi],w\to a}(\tau_w)}
		\bigg|<\frac\eps2.\nonumber
	\end{align}
Thus, {\bf BS2} follows from \eqref{eqBS2_1} and \eqref{eqBS2_2}.

Finally, to obtain the Bethe state decomposition of the simple factor graph $\GG$, we merely recall that $\pr\brk{\G\in\cS}=\Omega(1)$ by Fact~\ref{Fact_simple}.
Hence, claim about $\GG$ follows immediately form the statement for $\G$ and Bayes' rule.
\end{proof}

\subsection{Proof of \Prop~\ref{Lemma_fixp}}\label{Sec_Lemma_fixp}
By construction, the random factor graph $\G_*'$ comprises a pairing of variable clones $(v_i,h)\in V_n\times[d]$ and constraint clones $(a_j,h)\in F_m\times[d]$.
But since we obtained $\G_*'$ from $\G_*$ by removing some variable nodes $\vW$ along with their adjacent constraint nodes, not all of the variable clones $(v_i,h)$ with $i\leq n-\OMEGA$ are paired.
We call variables with at least one unpaired clone {\em cavities}.
Let $\cC$ be the set of all cavities.

The basic idea behind the proof is as follows.
We will add a new variable node $v^+$ along with new adjacent constraint nodes $b_1,\ldots,b_d$ to $\G_*'$.
Apart from $v^+$, these new constraint nodes are adjacent to some of the cavities.
The fresh randomness afforded by this construction will facilitate the study of the standard messages from $v^+$ to the $b_i$ as well as the reverse messages.
Then we will argue that $v^+$ is essentially indistinguishable from a randomly chosen variable node of $\G_*'$,
thereby extending the analysis to almost all the messages of $\G_*'$.

Formally, since $\OMEGA$ is a Poisson variable with mean $\omega$ truncated at $2\omega$, \whp\ we have $\omega/2\leq|\cC|\leq2d(k-1)\omega$.
Given that $|\cC|\geq d(k-1)$, obtain $\G_*^-$ from $\G_*'$ by re-inserting one variable node $v_+=v_{n-\OMEGA+1}$ along with $d$ new constraint nodes $b_1,\ldots,b_d$.
For each of these constraint nodes a random clone $(b_i,\vh_i)$, $\vh_i\in[k]$, is paired with a random clone $v_+$.
In addition, the $b_i$ are paired randomly to $k-1$ cavities.
The weight functions $\psi_{b_i}$ are chosen independently from $P$.
The following lemma shows that the distributions of $\G_*'$ and $\G_*^-$ are reasonably close.

\begin{lemma}\label{Claim_fixp_contig1}
For any event $\cE$ we have $\pr\brk{\G_*'\in\cE}\leq \alpha^{-1}\pr\brk{\G_*^-\in\cE}+O(\alpha^{100}).$
\end{lemma}
\begin{proof}
We need to get a grip on the conditional distribution of the second neighborhood of $v_+$ in $\G$ given $\G_*'$.
This is non-trivial because of the revised prior of $\G_*$ introduced by the pinning operation~\eqref{eqNewPrior}; for the assignment $\SIGMA_*$ is correlated with the neighborhood of $v_+$ in $\G$.
To begin, let $\cA$ be the event that no constraint node of $\G$ is connected by two edges with the variable nodes $v_{n-\OMEGA+1},\ldots,v_n$ and that all cavities have degree precisely $d-1$.
Then  (\ref{eqfunctions}) guarantees that
	\begin{equation}\label{eqClaim_fixp_contig1_1}
	\pr\brk{\G\in\cA}=1-O(\omega^2/n)=1-O(n^{-1/2}).
	\end{equation}
Further, given $\cA$ the total number of cavities of $\G_*'$ is equal to $d(k-1)\OMEGA$, and thus
	\begin{equation}\label{eqClaim_fixp_contig1_2}
	\pr\brk{\abs\cC\geq\omega/2\mid\cA}=1-O(\omega^{-1}).
	\end{equation}
Let $\cA'$ be the event that $\cA$ occurs, that $\abs\cC\geq\omega/2$ and that the weight functions of all constraints adjacent to $v_+$ take a minimum value of at least $\alpha^{-1/(2dk)}$.

We condition on the event $\cA'$, which occurs with probability $1+O(\alpha^{100})$ due to (\ref{eqP}), \eqref{eqClaim_fixp_contig1_1} and \eqref{eqClaim_fixp_contig1_2} show.
Let $\cN_1,\cN_2$ be two possible outcomes of the depth-two neighborhoods of $v_+$ in $\G$ given $\G_*'$.
Thus, $\cN_1,\cN_2$ specify the weight functions of the $d$ constraints adjacent to $v_+$, the pairing of the clones of $v_+$ to those of these constraint nodes,
and the pairing of these $d$ constraint nodes and the cavities $\cC$.
In addition, let $\G'$ be the random factor graph obtained from $\G_*'$ by restoring the prior to $p^{\tensor n}$.
Then we can set up a coupling $(\vec\Gamma_1,\vec\Gamma_2)$ of $\G$ given $\G',\cN_1,\cA'$ and of $\G$ given $\G',\cN_2,\cA'$ such that under $\Gamma$ the two random factor graphs differ in no more than $2dk$ edges:
the coupling simply switches the pairings occurring in $\cN_1$ but not in $\cN_2$, and vice versa.
In effect, on $\cA'$ the Boltzmann distributions $\mu_{\vec\Gamma_1}, \mu_{\vec\Gamma_2}$ are mutually $\alpha^{-1}$-contiguous.
Consequently, since the priors are amended according to samples from these respective Boltzmann distribution, 
we conclude that for any two outcomes $\cN_1,\cN_2$ of the second neighborhood of $v^+$ and for any possible outcome $g$ of $\G_*'$,
	\begin{align}\label{eqClaim_fixp_contig1_3}
	\pr\brk{\G_*'=g\mid\cN_2,\G',\cA'}\leq \alpha^{-1}\,\pr\brk{\G_*'=g\mid\cN_1,\G',\cA'}.
	\end{align}
Combining (\ref{eqClaim_fixp_contig1_1})--(\ref{eqClaim_fixp_contig1_3}), we conclude that for any possible $g$ and for any  $2\omega/3\leq\omega_0\leq3\omega/2$,
	\begin{align}\label{eqClaim_fixp_contig1_3a}
	\pr\brk{\G_*'=g\mid \OMEGA=\omega_0}\leq \alpha^{-1}\,\pr\brk{\G_*^-=g\mid \OMEGA=\omega_0+1}+O(\alpha^{100}).
	\end{align}
Finally, the assertion follows from (\ref{eqClaim_fixp_contig1_3a}) because $\dTV(\OMEGA,\OMEGA+1)=O(\omega^{-1/2})=O(\alpha^{200})$
and $\pr\brk{2\omega/3\leq\OMEGA\leq3\omega/2}=1-\exp(-\Omega(\omega))=1-O(\alpha^{200})$.
\end{proof}

\Lem~\ref{Claim_fixp_contig1} shows that studying the messages received by and emanating from $v_+$ is about as good as studying the messages of a random variable node of $\G_*'$.
The randomness involved in the attachment process will help, but is not yet quite sufficient to actually verify the Belief Propagation equations.
Namely, we also need to make sure that the Boltzmann distribution of the cavities is extremal in order to argue that typically the joint distribution of the variables where the new constraints $b_1,\ldots,b_d$ are anchored factorizes.
Unfortunately, we do not know a priori that extremality holds.
Indeed, while going from $\G$ to $\G_*'$ renders the Boltzmann distribution $\xi^{1/4}$-extremal (by \Lem~\ref{Lemma_shift}), the cavities are far too few in number to conclude that their joint distribution is extremal.

Hence, we will apply a second round of pinning.
But this time we will pin the cavities directly.
To be precise, recalling the random variable $\THETA_{+}=\THETA_{\zeta}$ from \Lem~\ref{Lemma_pinning}, let  $\cC_+\subset\cC$ be a random subset of size $\THETA_{+}\wedge|\cC|$.
Further, draw a sample $\SIGMA_{+}$ from $\mu_{\G_*'}$.
The choice of $\cC_+,\SIGMA_{+}$ is independent of the choice of the constraints $b_1,\ldots,b_d$, and $\SIGMA_+$ is independent of $\cC_+$.
Now, obtain $\G_*''$ from $\G_*'$ by changing the prior to
	\begin{align}\label{eqChangePriors}
	p_{\G_*''}(\sigma)&\propto p_{\G_*'}(\sigma)\prod_{y\in\cC_+}\vecone\{\sigma=\SIGMA_{+}(y)\}.
	\end{align}
Thus, we pin the cavities $y\in\cC_+$ to the spins observed under $\SIGMA_{+}$, which are independent of $b_1,\ldots,b_d$.

\begin{lemma}\label{Claim_fixp_contig}
The joint distribution $\mu_{\G_*'',\cC}$ of the cavities is $\zeta$-symmetric with probability at least $1-\zeta$.
\end{lemma}
\begin{proof}
Since $|\cC|\geq\omega/2$ with probability $1-\exp(-\Omega(\omega))$, 
the assertion follows immediately from \Lem~\ref{Lemma_pinning} and the construction of $\G_*''$.
\end{proof}

\noindent
Additionally, obtain $\G^+_*$ from $\G^-_*$ by changing the prior as per (\ref{eqChangePriors}) as well, i.e., 
	\begin{align}\label{eqpeg}
	p_{\G^+_*}(\sigma)\propto  p_{\G_*^-}(\sigma)\prod_{y\in\cC_+}\vecone\{\sigma=\SIGMA_{+}(y)\}.
	\end{align}
We are ready to verify the Belief Propagation equations for $v_+$ on $\G_*^+$.

\begin{lemma}\label{Claim_fixp1}
With probability $1-O(\alpha^{90})$ the random factor graph $\G_*^+$ has the following properties:
	\begin{align}
	\abs{\mu_{\G_*^+,b_i\to v_+}(\sigma)-\frac{\sum_{\tau\in\Omega^{\partial b_i}}\vecone\{\tau(v_+)=\sigma\}\psi_{b_i}(\tau)\prod_{w\in\partial b_i\setminus v_+}\mu_{\G_*^+,w\to b_i}(\tau(y))}{\sum_{\tau\in\Omega^{\partial b_i}}\psi_{b_i}(\tau)\prod_{w\in\partial b_i\setminus v_+}\mu_{\G_*^+,w\to b_i}(\tau(w))}}&\leq\alpha^{70d}&\forall i\in[d],\sigma\in\Omega,\label{eqClaim_fixp1_1}\\
	\abs{\mu_{\G_*^+,v_+\to b_i}(\sigma)-
		\frac{p(\sigma)\prod_{j\neq i}\sum_{\tau\in\Omega^{\partial b_j}}\vecone\{\tau_{v_+}=\sigma\}\psi_{b_j}(\tau)\prod_{w\in\partial b_j\setminus v_+}\mu_{\G_*^+,w\to b_j}(\tau_w)}
		{\sum_{\chi\in\Omega}p(\chi)\prod_{j\neq i}\sum_{\tau\in\Omega^{\partial b_j}}\psi_{b_j}(\tau)\vecone\{\tau_{v_+}=\chi\}\prod_{w\in\partial b_j\setminus v_+}\mu_{\G_*^+,w\to b_j}(\tau_w)}}
			&\leq\alpha^{70d}&\forall  i\in[d],\sigma\in\Omega,\label{eqClaim_fixp1_2}\\
	\abs{\mu_{\G_*^+}(\sigma)
			-\frac{\prod_{i=1}^d\psi_{b_i}(\sigma)\prod_{w\in\partial b_i\setminus v_+}\mu_{\G_*^+,w\to b_i}(\sigma(w))}
			{\sum_{\chi\in\Omega}\prod_{i=1}^d\sum_{\tau\in\Omega^{\partial b_i}}\vecone\{\tau_{v_+}=\chi\}\psi_{b_i}(\tau)
				\prod_{w\in\partial b_i\setminus v_+}\mu_{\G_*^+,w\to b_i}(\tau(w))}
			}&<\alpha^{70d}&\forall \sigma\in\Omega^{\nabla^2v_+}.\label{eqClaim_fixp1_3}
	\end{align}
\end{lemma}
\begin{proof}
\Lem~\ref{Claim_fixp_contig} shows that $\mu_{\G_*'',\cC}$ is $\zeta$-symmetric with probability at least $1-\zeta$.
Suppose it is.
Then \Prop~\ref{Cor_symmetry} shows that $\mu_{\G_*'',\cC}$ is $(\beta,d(k-1))$-symmetric.
We may also assume that $|\cC|\geq\omega/2$, an event that occurs with probability at least $1-\exp(-\Omega(\omega))$ by the construction of $\G_*'$.
Additionally, due to (\ref{eqP}) we may assume that 
	\begin{align}\label{eqP_1}
	\min_{\sigma\in\Omega^k}\psi_{b_i}(\sigma)\geq\alpha\qquad\mbox{ for all $i\in[d]$.}
	\end{align}

According to (\ref{eqstdav}),  the standard message $\mu_{\G_*^+,b_1\to v_+}$ is defined as the marginal of $v_+$ in the factor graph obtained from $G_*^+$ by removing $b_2,\ldots,b_d$ and replacing the prior of $v_+$ by the uniform distribution.
By construction, this factor graph is obtained from $\G_*''$ by adding the variable node $v_+$ and constraint node $b_1$ and replacing the prior of $v_+$ by the uniform distribution.
Therefore, 
	\begin{align}\label{eqClaim_fixp1_4}
	\mu_{\G_*^+,b_1\to v_+}(\sigma)&=\frac{\sum_{\tau\in\Omega^{\partial b_1}}\vecone\{\tau_{v_+}=\sigma\}\psi_{b_1}(\tau)
		\scal{\vecone\{\forall w\in\partial b_1\setminus v_+:\SIGMA_w=\tau_w\}}{\mu_{\G_*'',\cC}}}
		{\sum_{\tau\in\Omega^{\partial b_1}}\psi_{b_1}(\tau)
		\scal{\vecone\{\forall w\in\partial b_1\setminus v_+:\SIGMA_w=\tau_w\}}{\mu_{\G_*'',\cC}}}\qquad(\sigma\in\Omega).
	\end{align}
Further, the neighbors $\partial b_1\setminus v^+$ are chosen uniformly from $\cC$ (without replacement).
Because $|\cC|\geq\omega/2$ and   $\mu_{\G_*'',\cC}$ is $(\beta,d(k-1))$-symmetric, we conclude that
	\begin{align}\label{eqClaim_fixp1_5}
	\pr\brk{\TV{\mu_{\G_*'',\partial b_1\setminus v^+}-\bigotimes_{w\in\partial b_1\setminus v^+}\mu_{\G_*'',w}}\leq\beta^{1/3}}&\geq1-\beta^{1/3}.
	\end{align}
Combining (\ref{eqP_1}), (\ref{eqClaim_fixp1_4}) and (\ref{eqClaim_fixp1_5}), we obtain the estimate
	\begin{align}\label{eqClaim_fixp1_6}
	\Erw\abs{\mu_{\G_*^+,b_1\to v^+}(\sigma)-
	\frac{\sum_{\tau\in\Omega^{\partial b_1}}\vecone\{\tau_{v^+}=\sigma\}\psi_{b_1}(\tau)\prod_{w\in\partial b_1\setminus x^+}\mu_{\G_*'',w}(\tau_w)}
		{\sum_{\tau\in\Omega^{\partial b_1}}\psi_{b_1}(\tau)\prod_{w\in\partial b_1\setminus v^+}\mu_{\G_*'',w}(\tau_w)}
	}&\leq\beta^{1/4}.
	\end{align}	
Moreover, the factor graph $\G^+_*-b_1$ is obtained from $\G_*''$ by adding $v_+$ and $b_2,\ldots,b_d$.
Hence, (\ref{eqP_1}) implies that $\mu_{\G^+_*-b_1,\cC}$ is $(2/\alpha)^{dk}$-contiguous with respect to $\mu_{\G_*'',\cC}$.
Since $\mu_{\G_*'',\cC}$ is $\zeta$-symmetric, \Prop~\ref{Cor_symmetry} and \Lem~\ref{Lemma_extremal_contig} yield
	$\cutm(\mu_{\G_*^+-b_1,\cC},\mu_{\G_*'',\cC})<\beta.$
Because the neighborhood $\partial b_1$ is random, \Lem~\ref{Lemma_tv1} therefore yields
	\begin{align}\label{eqClaim_fixp1_7}
	\Erw\sum_{w\in\partial b_1\setminus v_+}\TV{\mu_{\G_*'',w}-\mu_{\G_*^+,w\to b_1}}&\leq O(\beta).
	\end{align}
Combining \eqref{eqP_1}, (\ref{eqClaim_fixp1_6}) and (\ref{eqClaim_fixp1_7}), we obtain (\ref{eqClaim_fixp1_1}).

The proofs of (\ref{eqClaim_fixp1_2}) and (\ref{eqClaim_fixp1_3}) are similar.
Indeed, $\mu_{\G_*^+,v_+\to b_1}$ is the marginal of $v_+$ in $\G_*^+-b_1$, which is obtained from $\G_*''$ by adding $b_2,\ldots,b_d$.
Hence,
	\begin{align*}
	\mu_{\G_*^+,v_+\to b_1}(\sigma)&=\frac{\sum_{\tau\in\Omega^{\{v_+\}\cup\partial^2v_+}}p(\sigma)\vecone\{\tau_{v_+}=\sigma\}
		\scal{\vecone\{\forall w\in\partial^2v_+:\SIGMA_w=\tau_w\}}{\mu_{\G_*'',\cC}}
		\prod_{i=2}^d\psi_{b_i}(\tau)}
		{\sum_{\tau\in\Omega^{\{v_+\}\cup\partial^2v_+}}p(\tau_{v_+})\scal{\vecone\{\forall w\in\partial^2v_+:\SIGMA_w=\tau_w\}}{\mu_{\G_*'',\cC}}
		\prod_{i=2}^d\psi_{b_i}(\tau)}.
	\end{align*}
Invoking the $(\beta,d(k-1))$-symmetry of $\mu_{\G_*''}$ and (\ref{eqP_1}), we obtain
	\begin{align*}
	\Erw\abs{\mu_{\G_*^+,v_+\to b_1}(\sigma)-\frac{\sum_{\tau\in\Omega^{\{v_+\}\cup\partial^2v_+}}p(\sigma)\vecone\{\tau_{v_+}=\sigma\}
		\prod_{i=2}^d\psi_{b_i}(\tau)\prod_{w\in\partial^2v_+}\mu_{\G_*'',w}(\tau_w)}
		{\sum_{\tau\in\Omega^{\{v_+\}\cup\partial^2v_+}}p(\tau_{v_+})
		\prod_{i=2}^d\psi_{b_i}(\tau)\prod_{w\in\partial^2v_+}\mu_{\G_*'',w}(\tau_w)}}&\leq\beta^{1/4}.
	\end{align*}
Moreover, reordering the sums and products, we simplify the last expression and find
	\begin{align}\label{eqClaim_fixp1_10}
	\Erw\abs{\mu_{\G_*^+,v_+\to b_1}(\sigma)-\frac{p(\sigma)
		\prod_{i=2}^d\sum_{\tau\in\Omega^{\partial b_i}}\vecone\{\tau_{v_+}=\sigma\}\psi_{b_i}(\tau)\prod_{w\in\partial b_i\setminus v_+}\mu_{\G_*'',w}(\tau_w)}
		{\sum_{\chi\in\Omega}p(\chi)\prod_{i=2}^d\sum_{\tau\in\Omega^{\partial b_i}}\vecone\{\tau_{v_+}=\chi\}\psi_{b_i}(\tau)\prod_{w\in\partial b_i\setminus v_+}\mu_{\G_*'',w}(\tau_w)}}&\leq\beta^{1/4}.
	\end{align}
Further, (\ref{eqP_1}) ensures that for each $i\in[d]$ the distribution $\mu_{\G_*^+-b_i,\cC}$ is $(2/\alpha)^{dk}$-contiguous with respect to $\G_*''$.
Consequently, since the neighbors of $b_i$ are chosen randomly from the set $\cC$ of cavities, \Prop~\ref{Cor_symmetry} and \Lem~\ref{Lemma_extremal_contig} yield
	\begin{align}\label{eqClaim_fixp1_11}
	\sum_{i=2}^d\Erw\brk{\sum_{w\in\partial b_i\setminus v_+}\TV{\mu_{\G_*'',w}-\mu_{\G_*^+,w\to b_i}}}&\leq O(\beta).
	\end{align}
Combining (\ref{eqClaim_fixp1_10}) and (\ref{eqClaim_fixp1_11}), we obtain (\ref{eqClaim_fixp1_2}).

Moving on to \eqref{eqClaim_fixp1_3}, we consider $\sigma\in\Omega^{\{v^+\}\cup\partial^2v^+}$.
Since $\G_*^+$ is obtained from $\G_*''$ by adding $v_+$ along with $b_1,\ldots,b_d$, we have the exact formula
	\begin{align*}
	\mu_{\G_*^+}(\sigma)&=
		\frac{p(\sigma)\scal{\vecone\{\forall w\in\partial^2v^+:\SIGMA_w=\sigma_w}{\mu_{\G_*''}}\prod_{i=1}^d\psi_{b_i}(\sigma)}
			{\sum_{\chi\in\Omega}p(\chi)
				\scal{\prod_{i=1}^d\sum_{\tau\in\Omega^{\partial b_i}}\vecone\{\tau_{v^+}=\chi\}\psi_{b_i}(\tau)
				\prod_{w\in\partial b_i\setminus v^+}\vecone\{\SIGMA_w=\tau_w\}}{\mu_{\G_*''}}}
				\qquad(\sigma\in\Omega^{\nabla^2v_+}).
	\end{align*}
Since $\partial^2v^+$ is a random set of cavities, the $(\beta,d(k-1))$-symmetry of $\mu_{\G_*''}$ and (\ref{eqP_1}) ensure that
	\begin{align}\label{eqClaim_fixp1_20}
	\Erw\abs{\mu_{\G_*^+}(\sigma)-
		\frac{p(\sigma)\prod_{w\in\partial^2v^+}\mu_{\G_*'',w}(\sigma_w)\prod_{i=1}^d\psi_{b_i}(\sigma)}
			{\sum_{\chi\in\Omega}p(\chi)\prod_{i=1}^d\sum_{\tau\in\Omega^{\partial b_i}}\vecone\{\tau_{v^+}=\chi\}\psi_{b_i}(\tau)
				\prod_{w\in\partial b_i\setminus v^+}\mu_{\G_*'',w}(\tau_w)}}&\leq\beta^{1/4}.
	\end{align}
Finally, to complete the proof we combine (\ref{eqClaim_fixp1_11}) and (\ref{eqClaim_fixp1_20}).
\end{proof}

We set up the random factor graph $\G_*^+$ so as to facilitate the verification of the BP equations.
But in a sense the model is a bit `out of line' because the prior is pinned according to a configuration $\SIGMA_+$ drawn from $\mu_{\G_*'}$ rather than $\mu_{\G_*^-}$; see~(\ref{eqpeg}).
Thus, with $\THETA_{+}=\THETA_{\zeta}$ and  $\cC_+\subset\cC$ as before, draw  
 $\SIGMA_{++}$ from $\mu_{\G_*^-}$ and let $\G_*^{++}$ be the random factor graph obtained from $\G_*^-$ by changing the prior to
	$$p_{\G_*^{++}}(\sigma)\propto p_{\G_*'}(\sigma)\prod_{w\in\cC_+}\vecone\{\sigma=\SIGMA_{++}(w)\}.$$
Hence, the pinning $\SIGMA_{++}$ takes $v_+,b_1,\ldots,b_d$ into account.

\begin{corollary}\label{Cor_Claim_fixp1}
With probability $1-O(\alpha^{80})$  the bounds (\ref{eqClaim_fixp1_1})--(\ref{eqClaim_fixp1_3}) hold with $\G_*^+$ replaced by $\G_*^{++}$.
\end{corollary}
\begin{proof}
The only difference between $\G_*^{++}$ and $\G_*^+$ lies in the choice of the configuration to which the variable nodes in $\cC_+$ get pinned.
But since $\G_*^+$ is obtained from $\G_*''$ by the mere addition of $d$ constraint nodes $b_1,\ldots,b_d$, (\ref{eqP}) shows that
$\SIGMA_{++}$ is $\alpha^{-1}$-contiguous with respect to the distribution of  $\SIGMA_{+}$ with probability $1-O(\alpha^{80}).$
Thus, the assertion follows from \Lem~\ref{Claim_fixp1}.
\end{proof}

We are finally ready to go back to the random factor graph $\G_*''$.
Indeed, basically the only difference between $\G_*^{++}$ and $\G_*''$ is that the former has one more variable node, along with $d$ adjacent constraint nodes.
But since the number of variable nodes of $\G_*''$ is random, this difference should hardly be noticeable.
Also $\G_*''$ is invariant under permutations of its variable nodes.
Thus, whatever we can prove for the last variable node $v_+$ of $\G_*^{++}$  carries over to a random variable node of $\G_*''$.
The following corollary makes this precise.

\begin{corollary}\label{Claim_fixp2}
With probability $1-O(\alpha^{70})$ we have
	\begin{align}
	\sum_{v\in V_n}\sum_{b\in\partial v}\sum_{\sigma\in\Omega}
	\abs{\mu_{\G_*'',b\to v}(\sigma)-\frac{\sum_{\tau\in\Omega^{\partial b}}\vecone\{\tau_v=\sigma\}\psi_{b}(\tau)\prod_{w\in\partial b\setminus v}\mu_{\G_*'',w\to b}(\tau_w)}{\sum_{\tau\in\Omega^{\partial b}}\psi_{b}(\tau)\prod_{w\in\partial b\setminus v}\mu_{\G_*'',w\to b}(\tau_w)}}&\leq n\alpha^{60d},\label{eqClaim_fixp2_1}\\
	\sum_{v\in V_n}\sum_{b\in\partial v}\sum_{\sigma\in\Omega}\abs{\mu_{\G_*'',v\to b}(\sigma)-
		\frac{p(\sigma)\prod_{a\in\partial v\setminus b}\sum_{\tau\in\Omega^{\partial a}}\vecone\{\tau_{v}=\sigma\}\psi_{a}(\tau)\prod_{w\in\partial a\setminus v}\mu_{\G_*'',w\to a}(\tau_w)}
		{\sum_{\chi\in\Omega}p(\chi)\prod_{a\in\partial v\setminus b}\sum_{\tau\in\Omega^{\partial a}}\vecone\{\tau_{v}=\chi\}\psi_{a}(\tau)\prod_{w\in\partial a\setminus v}\mu_{\G_*'',w\to a}(\tau_w)}}
			&\leq n\alpha^{60d}\label{eqClaim_fixp2_2}\\
	\sum_{v\in V_n}\sum_{\sigma\in\Omega^{x\cup\partial^2x}}	\abs{\mu_{\G_*''}(\sigma)
			-\frac{p(\sigma)\prod_{a\in\partial v}\psi_{a}(\sigma)\prod_{w\in\partial a\setminus v}\mu_{\G_*'',w\to a}(\sigma_w)}
			{\sum_{\chi\in\Omega}p(\chi)\prod_{a\in\partial v}\sum_{\tau\in\Omega^{\partial a}}\vecone\{\tau_{v}=\chi\}\psi_{a}(\tau)
				\prod_{w\in\partial a\setminus v}\mu_{\G_*'',w\to a}(\tau_w)}
			}&<n\alpha^{60d}.\label{eqClaim_fixp2_3}
	\end{align}
\end{corollary}
\begin{proof}
Consider the event $\cE$ that in $\G_*''$, for the variable node $v_{n-\OMEGA}$ with the largest index the estimate
	\begin{align}\label{eqClaim_fixp2_10}
	\sum_{b\in\partial v_{n-\OMEGA}}\sum_{\sigma\in\Omega}
	\abs{\mu_{\G_*'',b\to v_{n-\OMEGA}}(\sigma)-\frac{\sum_{\tau\in\Omega^{\partial b}}\vecone\{\tau(v_{n-\OMEGA})=\sigma\}\psi_{b}(\tau)\prod_{w\in\partial b\setminus v_{n-\OMEGA}}\mu_{\G_*'',w\to b}(\tau(w))}{\sum_{\tau\in\Omega^{\partial b}}\psi_{b}(\tau)\prod_{w\in\partial b\setminus v_{n-\OMEGA}}\mu_{\G_*'',w\to b}(\tau_w)}}&\leq\alpha^{69}
	\end{align}
holds.
Since $\G_*^{++}$ is obtained from $\G_*^+$ by the same process that produces $\G_*''$ from $\G_*'$, 
\Lem~\ref{Claim_fixp_contig1} and \Cor~\ref{Cor_Claim_fixp1} show that $\pr\brk{\G_*''\in\cE}=1-O(\alpha^{70})$.
But since the distribution of $\G_*''$ is invariant under permutations of the $n-O(\omega)$ variable nodes of degree $d$, we can replace
$v_{n-\OMEGA}$ in (\ref{eqClaim_fixp2_10}) by a random variable node of degree $d$.
Thus, we obtain (\ref{eqClaim_fixp2_1}).
The two bounds (\ref{eqClaim_fixp2_2}) and (\ref{eqClaim_fixp2_3}) follow analogously.
\end{proof}

To complete the proof of \Prop~\ref{Lemma_fixp}, we finally need to get from $\G_*''$ back to $\G_*'$.
Thus, we need to undo the additional pinning of the cavities that was required to verify the BP equations \eqref{eqClaim_fixp2_1}--\eqref{eqClaim_fixp2_3}.
The elegant insight that makes this possible is that \eqref{eqClaim_fixp2_3}--\eqref{eqClaim_fixp2_3} really just describe a property of the joint distribution of the second neighborhoods of the variable nodes $v_i$, $i\neq n-\OMEGA$.
Indeed, by \Lem~\ref{Lemma_localRemoval} the standard messages, defined via the removal of a few constraints adjacent to a single variable node $v$, can be expressed easily in terms of the joint distribution of the second neighborhood of $v$.
Furthermore, \Lem~\ref{Lemma_shift} implies that the enhanced measure $\hat\mu_{\G_*'}$ describing the second neighborhood distributions is $\xi^{1/4}$-extremal, with $\xi$ is near the top of the pecking order (\ref{eqfunctions}).
In effect, $\hat\mu_{\G_*'}$ is impervious to the additional pinning required to go from  $\G_*'$ to $\G_*''$.
Let us formalize this argument to finish the proof of \Prop~\ref{Lemma_fixp}.

\begin{proof}[Proof of \Prop~\ref{Lemma_fixp}]
By \Lem~\ref{Lemma_shift} the measure $\hat\mu_{\G^*}$ is $\xi^{1/4}$-extremal with probability $1-\omega^{-1}$.
Consequently, since $\G_*'$ is obtained by deleting $O(\omega)$ constraints, 
\Lem~\ref{Lemma_extremal_contig} and (\ref{eqP}) ensure that $\hat\mu_{\G_*'}$ is $\thet$-extremal with probability $1-\alpha^{10}$.
Furthermore, $\hat\mu_{\G_*''}$ is nothing but the conditional distribution $\hat\mu_{\G_*'}$ given the event $S_+$ that the spins of the $\THETA_\zeta$ cavities $\cC_+$ coincide with the ones of the reference configuration $\SIGMA_+$.
Since $\SIGMA_+$ is drawn from $\hat\mu_{\G_*'}$, with probability at least $1-\zeta$ we have 
	$$\hat\mu_{\G_*'}(S_+)\geq \zeta q^{-\THETA_\zeta}.$$
If so, and if $\hat\mu_{\G_*'}$ is $\thet$-extremal, then (\ref{eqfunctions}) and \Cor~\ref{Lemma_extremal_cond} imply that 
$\cutm(\hat\mu_{\G_*''},\hat\mu_{\G_*'})\leq\beta$.
In summary,
	\begin{align*}
	\pr\brk{\cutm(\hat\mu_{\G_*''},\hat\mu_{\G_*'})\leq\beta}&\geq1-2\zeta.
	\end{align*}
In addition (\ref{eqP}) ensures that with probability at least $1-\alpha^{10}$, 
	\begin{align}\label{eqTailBound}
	\abs{\cbc{a\in F(\G_*'):\min_{\sigma\in\Omega^k}\psi_a(\sigma)\leq\alpha}}&\leq n\alpha^{1000d}.
	\end{align}
Further, by \Cor~\ref{Claim_fixp2} the bounds (\ref{eqClaim_fixp2_1})--(\ref{eqClaim_fixp2_3}) hold with probability $1-O(\alpha^{70})$.

Thus, we are left to prove statements (i) and (ii) under the assumption that $\cutm(\hat\mu_{\G_*''},\hat\mu_{\G_*'})\leq\beta$
and that (\ref{eqClaim_fixp2_1})--(\ref{eqClaim_fixp2_3}) and~\eqref{eqTailBound} hold.
Applying \Lem~\ref{Lemma_tv1}, we obtain 
	\begin{align}\label{eqLemma_fixp_1}
	\sum_{v\in V}\TV{\hat\mu_{\G_*'',v}-\hat\mu_{\G_*',v}}&\leq O(\beta).
	\end{align}
Further, \Lem~\ref{Lemma_localRemoval} shows that the messages $\mu_{\G_*'',v\to a}$, $\mu_{\G_*'',a\to v}$ and  $\mu_{\G_*',v\to a}$, $\mu_{\G_*',a\to v}$  can be expressed in terms of the marginal distributions $\hat\mu_{\G_*'',v}$ and $\hat\mu_{\G_*',v}$ of the depth-two neighborhood.
Indeed, according to (\ref{eqstdva})--(\ref{eqstdav}), for any $a\in\partial v$ and $\sigma\in\Omega$,
	\begin{align*}
	\mu_{\G_*'',v\to a}(\sigma)&=\frac{\scal{\vecone\{\SIGMA_v=\sigma\}/\psi_a(\SIGMA)}{\hat\mu_{\G_*'',v}}}{\scal{1/\psi_a(\SIGMA)}{\hat\mu_{\G_*'',v}}},&
	\mu_{\G_*'',a\to v}(\sigma)&=\frac{\scal{\vecone\{\SIGMA_v=\sigma\}/(p(\sigma)\prod_{b\in\partial v\setminus a}\psi_b(\SIGMA))}{\hat\mu_{\G_*'',v}}}{\scal{1/(p(\SIGMA_v)\prod_{b\in\partial v\setminus a}\psi_b(\SIGMA))}{\hat\mu_{\G_*'',v}}},
	\end{align*}
and analogously for $\G_*'$.
Hence, the total variation bound (\ref{eqLemma_fixp_1}) and (\ref{eqTailBound}) imply that
	\begin{align}\label{eqLemma_fixp_2}
	\sum_{v\in V}\sum_{a\in\partial v}\TV{\mu_{\G_*'',v\to a}-\mu_{\G_*',v\to a}}+\TV{\mu_{\G_*'',a\to v}-\mu_{\G_*',a\to v}}
		&=O(n\alpha^{900d}).
	\end{align}
Combining \eqref{eqClaim_fixp2_3} and \eqref{eqTailBound}--(\ref{eqLemma_fixp_2}), we obtain assertion (ii).
Further, (\ref{eqClaim_fixp2_1}), (\ref{eqTailBound}) and (\ref{eqLemma_fixp_2}) readily yield
	\begin{align}
	\sum_{v\in V}\sum_{b\in\partial v}\sum_{\sigma\in\Omega}
	\abs{\mu_{\G_*',b\to v}(\sigma)-\frac{\sum_{\tau\in\Omega^{\partial b}}\vecone\{\tau(v)=\sigma\}\psi_{b}(\tau)\prod_{w\in\partial b\setminus v}\mu_{\G_*',w\to b}(\tau(w))}{\sum_{\tau\in\Omega^{\partial b}}\psi_{b}(\tau)\prod_{w\in\partial b\setminus v}\mu_{\G_*',w\to b}(\tau(w))}}&\leq O(n\alpha^{60d}).\label{eqLemma_fixp_3}
	\end{align}
Moreover, combining (\ref{eqClaim_fixp2_2}), (\ref{eqTailBound}), (\ref{eqLemma_fixp_1}) and (\ref{eqLemma_fixp_3}), we obtain 
	\begin{align}
	\sum_{v\in V}\sum_{b\in\partial v}\sum_{\sigma\in\Omega}
	\abs{\mu_{\G_*',v\to b}(\sigma)-\frac{p(\sigma)\prod_{a\in\partial v\setminus b}\mu_{\G_*',a\to v}(\sigma)}{\sum_{\chi\in\Omega}p(\chi)\prod_{a\in\partial v\setminus b}\mu_{\G_*',a\to v}(\chi)}}
&\leq O(n\alpha^{50d}).\label{eqLemma_fixp_4}
	\end{align}
Finally, (\ref{eqLemma_fixp_3}) and (\ref{eqLemma_fixp_4}) show that the standard messages are a $O(\alpha^{40d})$-BP fixed point.
\end{proof}

\section{The free energy: upper bound}\label{secInterpolate}

\subsection{Outline}
In this section we derive the following upper bound on the free energy.

\begin{proposition}\label{Prop_regint}
Assume that {\bf POS} is satisfied.
Then 
	\begin{align*}
\limsup_{n\to\infty}\frac1n\Erw\brk{\ln Z(\G)}&\leq\inf_{\pi\in\MEAS}\cB(\pi),&
	\limsup_{n\to\infty}\frac1n\Erw\brk{\ln Z(\GG)}&\leq\inf_{\pi\in\MEAS}\cB(\pi).
\end{align*}
\end{proposition}

The proof of \Prop~\ref{Prop_regint} consists of two parts.
First, we will prove that {\em any} $\mu\in\states$ yields an upper bound on $\Erw\brk{\ln Z(\G)}$.
Specifically, recalling the notation from \Sec~\ref{Sec_basics}, let
	\begin{align*}
	\cB'(\mu)&=\Erw\ln\scal{\bigoplus_{i=1}^n\VARPHI_i}\mu,&
	\cB''(\mu)&=\Erw\ln\scal{\bigoplus_{1\leq i\leq (k-1)dn/k}\PSI_{1,i}}\mu.
	\end{align*}
Then we have the following generic upper bound, which may be of interest in its own right.

\begin{proposition}\label{Prop_regint1}
Assume that {\bf POS} is satisfied.
Then  $\Erw\brk{\ln Z(\G)}\leq o(n)+\cB'(\mu)-\cB''(\mu)$ for any $\mu\in\states$.
\end{proposition}

\noindent
The proof of \Prop~\ref{Prop_regint1}, based on the interpolation method, is relatively standard, although the fact that we deal with regular graphs requires a bit of care.
The details are carried out in \Sec~\ref{Sec_regint1}.
This is the only place where condition {\bf POS} is required.

The second step toward the proof of \Prop~\ref{Prop_regint} is to show that for $\mu$ drawn from $\pi\in\MEAS$ the
upper bound from \Prop~\ref{Prop_regint1} boils down to the expression $\cB(\pi)$.

\begin{proposition}\label{Prop_regint2}
For any $\pi\in\MEAS$ we have $\cB(\pi)=\Erw[\cB'(\MU^\pi)-\cB''(\MU^\pi)].$
\end{proposition}

\noindent
We prove \Prop~\ref{Prop_regint2} in \Sec~\ref{Sec_regint2}.

\begin{proof}[Proof of \Prop~\ref{Prop_regint}]
The first assertion is immediate from \Prop s~\ref{Prop_regint1} and~\ref{Prop_regint2}.
To obtain the second assertion, we apply Azuma's inequality and~\eqref{eqP} to see that
	$n^{-0.51}\abs{\ln Z(\G)-\Erw\ln Z(\G)}$ converges to zero in probability.
Hence, Fact~\ref{Fact_simple} and Bayes' rule show that $\Erw\ln Z(\GG)=\Erw\ln Z(\G)+o(n)$, and thus the second assertion follows from the first.
\end{proof}

\subsection{Proof of \Prop~\ref{Prop_regint1}}\label{Sec_regint1}
We construct a family of random factor graph models parametrized by $t\in[0,1]$.
The free energy of the model at $t=1$ will be easy to compute, and we will see that it is (nearly) equal to $\cB'(\mu)-\cB''(\mu)$.
The model with $t=0$ essentially coincides with $\G$.
Furthermore, we will show that the derivative of the free energy is non-negative for all $t$, thus
 obtaining the desired upper bound on $\Erw\ln Z(\G)$.

To construct this interpolating family, fix $\mu\in\states$ and a small $\eps>0$.
For $t\in[0,1]$ let
	\begin{align*}
	\vm_t&=\Po((1-t)\exp(-\eps)dn/k),\\
	\vm_t'&=\Po(t\exp(-\eps)dn),\\
	\vm_t''&=\Po((1-t)(k-1)\exp(-\eps)dn/k),
	\end{align*}
all three mutually independent and independent of everything else.
Given $k\vm_t+\vm_t'\leq dn$, we define the random factor graph $\G_t$ as follows.
	\begin{description}
	\item[INT1] the set of variable nodes is $\cV=\cbc s\cup V_n$, and the set of spins is $\cX=\Omega\cup[0,1]$.
	\item[INT2] the set of constraint nodes is 
			$$\cF_{t}=\cbc{a_1,\ldots,a_{\vm_t},a_1',\ldots,a_{\vm_t'}',a_1'',\ldots,a_{\vm_t''}''}.$$
	\item[INT3] each constraint node $a_i$ independently chooses a weight function $\psi_{a_i}$ from $P$, and
		the $a_i$ are joined to the variable nodes $v_1,\ldots,v_n$ by a random pairing of $V_n\times[d]$ and $\{a_1,\ldots,a_{\vm_t}\}\times[k]$.
	\item[INT4] each of the constraint nodes $a_i'$, $i\in[\vm_t']$, is adjacent to the variable node $s$ and one further variable node from $v_1,\ldots,v_n$;
		the links between the $a_i'$ and the $v_j$ are constructed by choosing a random pairing between the $\vh_i'$-clone of each $a_i'$ and the clones in $V_n\times[d]$
		that are not paired to a constraint node $a_h$.
		The weight function associated with $a_i'$ reads
		\begin{align*}
		\psi_{a_i'}(s,\sigma)&=\sum_{\tau\in\Omega^k}\vecone\{\tau_{\vh_i'}=\sigma\}\PSI_i'(\tau)\prod_{h\neq\vh_i'}\mu_{s,\vx_{i,h}'}(\tau_h).
		\end{align*}
	\item[INT5] the constraint nodes $a_i''$, $i\in[\vm_t'']$, are unary, adjacent to $s$ only.
		Their weight functions read
		\begin{align*}
		\psi_{a_i''}(s)&=\sum_{\tau\in\Omega^k}\PSI_i''(\tau)\prod_{h=1}^k\mu_{s,\vx_{i,h}''}(\tau_h).
		\end{align*}
	\item[INT6] the prior $\prior$ is 		a product measure  
		$$\dd\prior(\sigma)=\vecone\{\sigma_s\in[0,1],\,\forall 1\leq i\leq n:\sigma_{v_i}\in\Omega\}
				\prod_{i=1}^np\bc{\sigma_{x_i}} \dd \sigma_s;$$
		thus, for each $v_i\in V_n$ a spin from $\Omega$ is chosen independently from $p$, and $\sigma_s$ is uniform on $[0,1]$.
	\end{description}
Thus, the total weight, partition function and Boltzmann distribution of $\G_t$ read
	\begin{align}\nonumber
	\psi_{\G_t}(\sigma)&=\prod_{i=1}^{\vm_t}\psi_{a_i}(\sigma)\prod_{i=1}^{\vm_t'}\psi_{a_i'}(\sigma)\prod_{i=1}^{\vm_t''}\psi_{a_i''}(\sigma),
		&&(\sigma_s\in[0,1],\ \sigma_{x_i}\in\Omega),\\
	Z(\G_t)&=\sum_{\sigma_{x_1},\ldots,\sigma_{x_n}\in\Omega}\int_0^1\psi_{\G_t}(\sigma)\dd\sigma_s \prod_{i=1}^n p(\sigma_{x_i}),&
	\dd\mu_{\G_t}(\sigma)&=\frac{\psi_{\G_t}(\sigma)}{Z(\G_t)}\dd\prior(\sigma).
		\label{eqBoltzmannInt}
	\end{align}
The following lemma establishes the monotonicity of the free energy in $t$; its proof is the only place where we use condition {\bf POS}.

\begin{lemma}\label{Lemma_intderiv}
Suppose that {\bf POS} is satisfied.
Then uniformly for all $t\in(0,1)$ 
we have
	$$\frac1n\frac{\partial}{\partial t}\Erw\brk{\ln Z(\G_t)}\geq o(1).$$
\end{lemma}
\begin{proof}
We recall the derivative of the Poisson density: for any $\lambda>0$, $\ell\geq1$,
	\begin{align}\label{eqLemma_intderiv1}
\frac{\partial}{\partial\lambda}\pr\brk{\Po(\lambda)=\ell}=
\frac{\partial}{\partial\lambda}\frac{\lambda^\ell}{\ell!}\exp(-\lambda)&=\frac{\lambda^{\ell-1}}{(\ell-1)!}\exp(-\lambda)-
	\frac{\lambda^\ell}{\ell!}\exp(-\lambda)=\pr\brk{\Po(\lambda)=\ell-1}-\pr\brk{\Po(\lambda)=\ell}.
\end{align}
The variable $t$ affects the distribution of $\G_t$ by way of the variables $\vm_t,\vm_t',\vm_t''$.
Specifically, let
	\begin{align*}
	\lambda_t&=(1-t)\exp(-\eps)dn/k,&\lambda_t'&=t\exp(-\eps)dn,&\lambda_t''&=(1-t)(k-1)\exp(-\eps)dn/k.
	\end{align*}
Recall that $\vm_t,\vm_t'$ are conditional Poisson variables $\Po(\lambda_t)$ and $\Po(\lambda_t')$, respectively, given that $k\vm_t+\vm_t'\leq dn$.
Since $\eps>0$ is independent of $n$, \eqref{eqLemma_intderiv1} shows that  for any two integers $m_t,m_t'\geq1$,
	\begin{align}\nonumber
	\frac1n\frac{\partial}{\partial t}
	\pr\brk{\vm_t=m_t,\  \vm_t'=m_t'}&=\exp(-\Omega(n))+	\frac1n\frac{\partial}{\partial t}\pr\brk{\Po(\lambda_t)=m_t}\pr\brk{\Po(\lambda_t')=m_t'}\\
		&=\exp(-\Omega(n))+
			\frac1n\frac{\partial}{\partial t}\pr\brk{\Po(\lambda_t)=m_t}\pr\brk{\Po(\lambda_t')=m_t'}\nonumber\\
		&=\exp(-\Omega(n))+\bc{\pr\brk{\Po(\lambda_t)=m_t}-\pr\brk{\Po(\lambda_t)=m_t-1}}\pr\brk{\Po(\lambda_t')=m_t'}\exp(-\eps)d/k\nonumber\\
		&\qquad +\bc{\pr\brk{\Po(\lambda_t')=m_t'-1}-\pr\brk{\Po(\lambda_t')=m_t'}}\cdot\pr\brk{\Po(\lambda_t)=m_t}\exp(-\eps)d.
			\label{eqLemma_intderiv2}
	\end{align}
Further, given the event $k\vm_t+\vm_t''\leq dn-k$ let $\G_t'$ be the random factor graph obtained from $\G_t$ by adding one more constraint node $a_{\vm_t+1}$ as per {\bf INT3}.
Similarly, given $k\vm_t+\vm_t''\leq dn-1$ obtain $\G_t''$ from $\G_t$ by adding $a_{\vm_t'+1}$ according to {\bf INT4}.
Additionally, obtain $\G_t'''$ from $\G_t$ by adding a unary $a_{\vm_t''+1}$ as described in {\bf INT5}.
Since $\vm_t''$ is independent of $\vm_t',\vm_t''$, (\ref{eqLemma_intderiv1}) and (\ref{eqLemma_intderiv2}) yield
	\begin{align}\nonumber
	\frac1n\frac\partial{\partial t}\Erw[\ln Z(\G_t)]&=\exp\bc{-\Omega(n)}+
		\sum_{\substack{m_t,m_t',m_t''\geq1\\km_t+m_t''\leq dn-k}}
			\Erw\brk{\ln Z(\G_t)\,\Bigg|\,
				\begin{pmatrix}\vm_t\\\vm_t'\\\vm_t''\end{pmatrix}
				=\begin{pmatrix}m_t\\m_t'\\m_t''\end{pmatrix}}
			\cdot\frac1n\frac{\partial}{\partial t}\pr\brk{\begin{pmatrix}\vm_t\\\vm_t'\\\vm_t''\end{pmatrix}
				=\begin{pmatrix}m_t\\m_t'\\m_t''\end{pmatrix}}\\
		&=\exp\bc{-\Omega(n)}-\exp\bc{-\eps}\frac dk\bigg[\Erw\ln\frac{Z(\G_t')}{Z(\G_t)}
			-k\Erw\ln\frac{Z(\G_t'')}{ Z(\G_t)}+(k-1)\Erw\ln\frac{Z(\G_t''')}{Z(\G_t)}\bigg].
			\label{eqLemma_intderiv3}
	\end{align}
Hence, it suffices to prove that for all $0<t<1$,
	\begin{align}			\label{eqLemma_intderiv4}
	\Erw\brk{\ln\frac{Z(\G_t')}{Z(\G_t)}}-k\Erw\brk{\ln\frac{Z(\G_t'')}{Z(\G_t)}}+(k-1)\Erw\brk{\ln\frac{Z(\G_t''')}{Z(\G_t)}}&\leq0.
	\end{align}

By the definition of the Boltzmann distribution (\ref{eqBoltzmannInt}),
	\begin{align*}
	\frac{Z(\G_t')}{Z(\G_t)}&=\scal{\psi_{a_{\vm_t+1}}}{\mu_{\G_t}},&
	\frac{Z(\G_t'')}{Z(\G_t)}&=\scal{\psi_{a_{\vm_t'+1}'}}{\mu_{\G_t}},&
	\frac{Z(\G_t''')}{Z(\G_t)}&=\scal{\psi_{a_{\vm_t''+1}''}}{\mu_{\G_t}}.
	\end{align*}
Hence,
	\begin{align}		\label{eqLemma_intderiv5}
	\ln \frac{Z(\G_t')}{Z(\G_t)}&=\ln \scal{\psi_{a_{\vm_t+1}}}{\mu_{\G_t}}
		=-\sum_{\ell\geq1}\frac1\ell\scal{1-\psi_{a_{\vm_t+1}}}{\mu_{\G_t}}^\ell.
	\end{align}
Further, in terms of the kernel representation $\dot\mu_{\G_t}$ of the Boltzmann distribution 
we obtain
	\begin{align}\label{eqLemma_intderiv6}
	\Erw\brk{\scal{1-\psi_{a_{\vm_t+1}}}{\mu_{\G_t}}^\ell}&=
		\Erw\brk{\bc{1-\sum_{\sigma\in\Omega^k}\PSI(\sigma)\int_0^1\prod_{i=1}^k\dot\mu_{\G_t,z,\vx_i}(\sigma_i)\dd z}^\ell}.
	\end{align}
Combining (\ref{eqLemma_intderiv5}) and (\ref{eqLemma_intderiv6}) yields
	\begin{align}\label{eqLemma_intderiv7a}
	\Erw\brk{\ln\frac{Z(\G_t')}{Z(\G_t)}}&=-\Erw\brk{\sum_{\ell\geq1}\frac1\ell
		\bc{1-\sum_{\sigma\in\Omega^k}\PSI(\sigma)\int_0^1\prod_{i=1}^k\dot\mu_{\G_t,z,\vx_i}(\sigma_i)\dd z}^\ell}.
	\end{align}
Due to (\ref{eqP}) and Fubini's theorem, we can exchange the sum and the expectation in (\ref{eqLemma_intderiv7a});
indeed, (\ref{eqP}) yields
	\begin{align*}
	\sum_{\ell\geq1}\Erw\abs{\bc{1-\sum_{\sigma\in\Omega^k}\PSI(\sigma)\int_0^1\prod_{i=1}^k\tilde\mu_{\G_t,z,\vx_i}(\sigma_i)\dd z}^\ell}
		&\leq\sum_{\ell\geq1}\Erw\brk{\max_{\sigma\in\Omega^k}|1-\PSI(\sigma)|^\ell}<\infty.
	\end{align*}
Thus, (\ref{eqLemma_intderiv7a}) becomes
	\begin{align}\label{eqLemma_intderiv7}
	\Erw\brk{\ln\frac{Z(\G_t')}{Z(\G_t)}}&=-\sum_{\ell\geq1}\frac1\ell
		\Erw\brk{\bc{1-\sum_{\sigma\in\Omega^k}\PSI(\sigma)\int_0^1\prod_{i=1}^k\tilde\mu_{\G_t,z,\vx_i}(\sigma_i)\dd z}^\ell}.
	\end{align}
Following similar steps, we obtain expansions for the other two terms from (\ref{eqLemma_intderiv4}) as well:
	\begin{align}\label{eqLemma_intderiv8}
	\Erw\brk{\ln\frac{Z(\G_t'')}{Z(\G_t)}}&=-\frac1k\sum_{h=1}^k\sum_{\ell\geq1}\frac1\ell
		\Erw\brk{\bc{1-\sum_{\sigma\in\Omega^k}\PSI(\sigma)\int_0^1\tilde\mu_{\G_t,z,\vx_h}(\sigma_h)
			\prod_{i\neq h}\mu_{z,\vx_i}(\sigma_i)\dd z}^\ell},\\
	\Erw\brk{\ln\frac{Z(\G_t''')}{Z(\G_t)}}&=-\sum_{\ell\geq1}\frac1\ell
		\Erw\brk{\bc{1-\sum_{\sigma\in\Omega^k}\PSI(\sigma)\int_0^1\prod_{i=1}^k\mu_{z,\vx_i}(\sigma_i)\dd z}^\ell}.
			\label{eqLemma_intderiv9}
	\end{align}
Finally, the assertion follows from {\bf POS} and (\ref{eqLemma_intderiv4}), (\ref{eqLemma_intderiv7}), (\ref{eqLemma_intderiv8}) and~(\ref{eqLemma_intderiv9}).
\end{proof}

\begin{proof}[Proof of \Prop~\ref{Prop_regint1}]
Integrating $t$ from $0$ to $1$ and applying \Lem~\ref{Lemma_intderiv}, we obtain for any $\eps>0$,
	\begin{equation}\label{eqProp_regint1_1}
	\Erw[\ln Z(\G_0)]\leq\Erw[\ln Z(\G_1)]+o(n).
	\end{equation}
Letting
	$$Y=\ln\int_0^1\prod_{i=1}^{\vm_0''}\sum_{\sigma\in\Omega^k}\PSI_{i}''(\sigma)
			\prod_{h=1}^k\mu_{z,\vx_{i,h}''}(\sigma_h)\dd z,$$
we claim that for a certain number $c=c(P)>0$,
	\begin{equation}\label{eqProp_regint1_2a}
	\Erw\ln Z(\G)+\Erw[Y]\leq \Erw\ln Z(\G_1) +\eps c n.
	\end{equation}
Indeed, at $t=0$ the variable node $s$ is adjacent to the constraint nodes $a_i''$, $i\in[\vm_t'']$, only.
Hence, $\G_0$ decomposes into connected components, one of which comprises $s$ and the $a_i''$.
Let $\G_0''$ be this component, and let $\G_0'$ be the remainder of $\G_0$.
Then by construction we have $\Erw\ln Z(\G_0'')=\Erw[Y]$.
Thus, (\ref{eqProp_regint1_1}) yields
	\begin{align}\label{eqProp_regint1_1_a}
	\Erw[\ln Z(\G_0')]+\Erw[Y]&=\Erw[\ln Z(\G_0)]\leq \Erw[\ln Z(\G_1)]+o(n).
	\end{align}
Furthermore, $\G_0'$ consists of the variable nodes $v_1,\ldots,v_n$ and the constraint nodes $a_1,\ldots,a_{\vm_1}$,
where $\vm_1$ is a Poisson variable $\Po(\exp(-\eps)dn/k)$ conditioned on taking a value of at most $dn/k$.
Thus, we can construct a random factor graph with the same distribution as $\G$ from $\G_0'$ by simply adding
$dn/k-\vm_1$ further random $k$-ary constraint nodes as per {\bf INT3}.
Since all weight functions $\psi\in\Psi$ take values in $(0,2)$, we obtain $c=c(P)>0$ such that
	\begin{align}\label{eqProp_regint1_1_b}
	\Erw\ln Z(\G)\leq\Erw\ln Z(\G_0')+\eps cn.
	\end{align}
Combining (\ref{eqProp_regint1_1_a}) and (\ref{eqProp_regint1_1_b}), we obtain (\ref{eqProp_regint1_2a}).

We further claim that there is a constant $c'=c'(P)>0$ such that
	\begin{equation}\label{eqProp_regint1_2}
	\frac1n\Erw[Y]\leq \eps c'+o(1)+\cB''(\mu).
	\end{equation}
Indeed, $\cB''(\mu)=\Erw[Y\mid\vm_0''=(k-1)dn/k]$.
In other words, we can think of $\cB''(\mu)$ as the free energy of $\G_0''$ given that $\vm_0''=(k-1)dn/k$.
Thus, obtain $\G_0'''$ from $\G_0''$ by adding $(k-1)dn/k-\vm_0''$ more constraint nodes according to {\bf INT5}, or by removing some random constraint nodes if $\vm_0''>(k-1)dn/k$.
Then $\cB''(\mu)=\Erw\ln Z(\G_0''')$.
Since $\vm_0''$ is a Poisson variable with mean $\exp(-\eps)(k-1)dn/k$, with probability $1-\exp(-\Omega(n))$ we do not need to add or remove more than $2\eps(k-1)dn$ constraint nodes.
The tail bound~(\ref{eqP}) therefore implies together with the Chernoff bound that (\ref{eqProp_regint1_2}) is satisfied for a certain $c'=c'(P)$.

By similar arguments, for a certain $c''=c''(P)$ we have
	\begin{equation}\label{eqProp_regint1_3}
	\frac1n\Erw[\ln Z(\G_1)]\leq \cB'(\mu)+ \eps c''+o(1).
	\end{equation}
Indeed, $\cB'(\mu)$ is nothing but the conditional expectation of $\ln Z(\G_1)$ given that $\vm_1'=dn$.
Hence, if we pad $\G_1$ by adding the missing $dn-\vm_1'$ constraint nodes $a_i'$ according to {\bf INT4}, then
the total number of constraints added does not exceed $2\eps dn$ with probability $1-\exp(-\Omega(n))$.
Hence, (\ref{eqProp_regint1_3}) follows from \eqref{eqP} and the Chernoff bound.

Finally, combining (\ref{eqProp_regint1_1})--(\ref{eqProp_regint1_3}), we conclude that
	\begin{align*}
	\frac1n\Erw\ln Z(\G)\leq \cB'(\mu)-\cB''(\mu)+\eps c'''+o(1)
	\end{align*}
for a certain $c'''=c'''(P)>0$.
Since this is true for any fixed $\eps>0$, the assertion follows.
\end{proof}

\subsection{Proof of \Prop~\ref{Prop_regint2}}\label{Sec_regint2}
Following Panchenko~\cite{Panchenko}, who worked with factor graphs of \Erdos-\Renyi\ type, we are going to use the invariance property of $\pi\in\MEAS$ under the $*(N,M)$-operation to simplify $\cB',\cB''$ separately.

\begin{lemma}\label{Lemma_Prop_regint2_B''}
Suppose that $\pi\in\MEAS$.
Then
	\begin{align}\label{eqLemma_Prop_regint2_B''}
	\Erw[\cB''(\MU^\pi)]&=
		\frac{d(k-1)n}{k}
		\Erw\brk{\ln\scal{\PSI_1}{\pi}}.
	\end{align}
\end{lemma}
\begin{proof}
Let $\phi=\Erw\brk{\ln\scal{\PSI_1}{\pi}}$ for brevity.
We claim that for any integer $m\geq0$,
	\begin{align}\label{eqLemma_Prop_regint2_B''_1}
	\Erw\brk{\ln\frac{\scal{\bigoplus_{i=1}^{m+1}\PSI_i}{\pi}}{\scal{\bigoplus_{i=1}^{m}\PSI_i}{\pi}}}&=\phi.
	\end{align}
Then (\ref{eqLemma_Prop_regint2_B''}) follows by summing (\ref{eqLemma_Prop_regint2_B''_1}) on $0\leq m<d(k-1)n/k$.

Thus, we are left to prove (\ref{eqLemma_Prop_regint2_B''_1}).
Since $\pi\in\MEAS$, Corollaries~\ref{Cor_cntFct2} and~\ref{Cor_JoinNM} imply that for any integer $\ell\geq1$,
	\begin{align}\nonumber
	\Erw\brk{\bc{\frac{\scal{\bigoplus_{i=1}^{m+1}\PSI_i}{\pi}}{\scal{\bigoplus_{i=1}^{m}\PSI_i}{\pi}}}^\ell}
		&=\Erw\brk{\scal{\frac{\bigoplus_{i=1}^m\PSI_i}{\scal{\bigoplus_{j=1}^m\PSI_j}{\pi}}\oplus\PSI_{m+1}}{\pi}^\ell}\\
		&=\Erw\brk{\scal{\PSI_{m+1}}{\bigoplus_{i=1}^m\PSI_i\Join\pi}^\ell}=\Erw\brk{\scal{\PSI_{m+1}}{\pi}^\ell}=\Erw\brk{\scal{\PSI_{1}}{\pi}^\ell}.\nonumber
	\end{align}
Consequently, for all $\ell\geq1$ we have
	\begin{equation}\label{eqLemma_Prop_regint2_B''_4a}
	\Erw\brk{\bc{1-\frac{\scal{\bigoplus_{i=1}^{m+1}\PSI_i}{\pi}}{\scal{\bigoplus_{i=1}^{m}\PSI_i}{\pi}}}^\ell}=\Erw\brk{\bc{1-\scal{\PSI_{1}}{\pi}}^\ell}.
	\end{equation}
Further, because the continuous function $z\in[-1,1]\mapsto|z|$ is a uniform limit of polynomials, (\ref{eqLemma_Prop_regint2_B''_4a}) yields
	\begin{align*}
	\Erw\abs{\bc{1-\frac{\scal{\bigoplus_{i=1}^{m+1}\PSI_i}{\pi}}{\scal{\bigoplus_{i=1}^{m}\PSI_i}{\pi}}}^\ell}=\Erw\abs{\bc{1-\scal{\PSI_{1}}{\pi}}^\ell}
	\end{align*}
Therefore, invoking (\ref{eqP}), we obtain
	\begin{align*}
	\sum_{\ell\geq1}\frac1\ell\Erw\abs{\bc{1-\frac{\scal{\bigoplus_{i=1}^{m+1}\PSI_i}{\pi}}{\scal{\bigoplus_{i=1}^{m}\PSI_i}{\pi}}}^\ell}
		&=\sum_{\ell\geq1}\frac1\ell\Erw\abs{\bc{1-\scal{\PSI_{1}}{\pi}}^\ell}
		\leq\sum_{\ell\geq1}\Erw\brk{\max_{\sigma\in\Omega^k}|1-\PSI_1(\sigma)|^\ell}<\infty.
	\end{align*}
Hence, by (\ref{eqLemma_Prop_regint2_B''_4a}) and Fubini's theorem, 
	\begin{align*}
	\Erw\brk{\ln\frac{\scal{\bigoplus_{i=1}^{m+1}\PSI_i}{\pi}}{\scal{\bigoplus_{i=1}^{m}\PSI_i}{\pi}}}&=
		-\sum_{\ell\geq1}\frac1\ell\Erw\brk{\bc{1-\frac{\scal{\bigoplus_{i=1}^{m+1}\PSI_i}{\pi}}{\scal{\bigoplus_{i=1}^{m}\PSI_i}{\pi}}}^\ell}
		=-\sum_{\ell\geq1}\frac1\ell\Erw\brk{\bc{1-\scal{\PSI_{1}}{\pi}}^\ell}=\phi,
	\end{align*}
which is \eqref{eqLemma_Prop_regint2_B''_1}.
\end{proof}

\begin{lemma}\label{Lemma_Prop_regint2_B'}
Suppose that $\pi\in\MEAS$.
Then $\Erw[\cB'(\MU^\pi)]=\Erw\log\scal{\VARPHI_1}{\pi}.$
\end{lemma}
\begin{proof}
We use a similar argument as in the proof of \Lem~\ref{Lemma_Prop_regint2_B''}.
This time we set $\phi=\Erw\log\scal{\VARPHI_1}{\pi}$.
It suffices to show that for every $n\geq0$,
	\begin{align}\label{eqLemma_Prop_regint2_B'_1}
	\Erw\brk{\ln\frac{\scal{\bigoplus_{i=1}^{n+1}\VARPHI_i}\pi}{\scal{\bigoplus_{i=1}^{n}\VARPHI_i}\pi}}&=\phi.
	\end{align}
As in the proof of \Lem~\ref{Lemma_Prop_regint2_B''}, we  use that $\pi\in\MEAS$ and apply Corollaries~\ref{Cor_cntFct2} and~\ref{Cor_JoinNM} to obtain
for any $\ell\geq1$,
	\begin{align}
	\Erw\brk{\bc{\frac{\scal{\bigoplus_{i=1}^{n+1}\VARPHI_i}\pi}{\scal{\bigoplus_{i=1}^{n}\VARPHI_i}\pi}}^\ell}&=
		\Erw\brk{\bc{\scal{\frac{\bigoplus_{i=1}^{n}\VARPHI_i}{\scal{\bigoplus_{j=1}^{n}\VARPHI_j}{\pi}}\oplus\VARPHI_{n+1}}\pi}^\ell}
		&=
		\Erw\brk{\scal{\VARPHI_{n+1}}{\bigoplus_{i=1}^{n}\VARPHI_i\Join\pi}^\ell}=
		\Erw\brk{\scal{\VARPHI_1}\pi^\ell}.
		\label{eqLemma_Prop_regint2_B'_4}
	\end{align}
Hence, for any $\ell\geq1$,
	\begin{align}		\label{eqLemma_Prop_regint2_B'_4a}
	\Erw\brk{\bc{1-\frac{\scal{\bigoplus_{i=1}^{n+1}\VARPHI_i}\pi}{\scal{\bigoplus_{i=1}^{n}\VARPHI_i}\pi}}^\ell}
		&=\Erw\brk{\bc{1-\scal{\VARPHI_1}\pi}^\ell}.
	\end{align}
Further,  approximating the absolute value by polynomials, we obtain from (\ref{eqLemma_Prop_regint2_B'_4}) that 
	\begin{align*}
	\Erw\abs{\bc{1-\frac{\scal{\bigoplus_{i=1}^{n+1}\VARPHI_i}\pi}{\scal{\bigoplus_{i=1}^{n}\VARPHI_i}\pi}}^\ell}
		&=\Erw\abs{\bc{1-\scal{\VARPHI_1}\pi}^\ell}.
	\end{align*}
Thus, (\ref{eqLemma_Prop_regint2_B'_1}) follows from (\ref{eqLemma_Prop_regint2_B'_4a}) and Fubini's theorem.
\end{proof}

\noindent
Finally, \Prop~\ref{Prop_regint2} is immediate from \Lem s~\ref{Lemma_Prop_regint2_B''} and~\ref{Lemma_Prop_regint2_B'}.

\section{The free energy: lower bound}
\label{secAz}

\subsection{Outline}
In this section we prove the following lower bound on the free energy that matches the upper bound from \Prop~\ref{Prop_regint}.
The lower bound does not require the assumption {\bf POS}.

\begin{proposition}\label{MainAizLemReg}
We have 
	\begin{align*}
\liminf_{n \to \infty} \frac{1}{n} \Erw\brk{\log Z (\G)}&\ge \inf_{\pi\in\MEAS}\cB(\pi),&	
	\liminf_{n \to \infty} \frac{1}{n} \Erw\brk{\log Z (\GG)}\ge \inf_{\pi\in\MEAS}\cB(\pi).
\end{align*}
\end{proposition}

\noindent
\Thm~\ref{Thm_freeEng} follows  immediately from \Prop s~\ref{Prop_regint} and~\ref{MainAizLemReg}.

The proof of \Prop~\ref{MainAizLemReg} is based on a kind of coupling argument that is colloquially referred to as the `Aizenman--Sims--Starr' scheme.
This technique has been applied with great success to random factor graphs of \Erdos-\Renyi\ type, where the degree distribution is approximately Poisson~\cite{BP,Bethe,Panchenko}.
The basic idea is to couple a random factor graph with $n$ variable nodes with a random factor graph with $n+1$ variable nodes and to calculate the difference of their free energies very precisely.
This coupling is very easy to set up in the \Erdos-\Renyi\ case due to the Stein-Chen property of the Poisson distribution.

However, in the case of random regular graphs matters are more intricate.
Due to the rigid local structure there is no obvious way of coupling random regular factor graphs with $n$ and $n+1$ variable nodes.
As in \Sec~\ref{Sec_Bethe}, we therefore resort to the idea of creating a bit of wiggling room by carving out a few cavities, in such a way that the free energy does not change significantly.
But the details of the construction are delicate.

Let $n,\omega$ be integers and let $\vX,\vY$ be two independent Poisson variables with mean $\omega$.
The protagonist of the proof is the random factor graph $\G_{n,\omega}$ defined as follows.
Let
	$$N_{n,\omega}={k\vee}(n-\vX)\qquad\mbox{and let}\qquad
		\Delta_{n,\omega}\disteq\Be(d N_{n,\omega}/k-\lfloor d N_{n,\omega}/k\rfloor)$$
be independent of $Y$.
Further, set
	$$M_{n,\omega}={d\vee}\bc{\lfloor d N_{n,\omega}/k\rfloor\wedge\bc{\lfloor d N_{n,\omega}/k\rfloor+\Delta_{n,\omega}-d\vX-\vY}}.$$
Then $\G_{n,\omega}$ has $N_{n,\omega}$ variable nodes $v_i$, $i\in[N_{n,\omega}]$, and
$M_{n,\omega}$ constraint nodes $a_i$, $i\in[M_{n,\omega}]$.
The weight functions $\psi_{a_i}$ are chosen independently from $P$.
Furthermore, the variable and constraint nodes are linked through a random (one-to-one) pairing
	$$F_{M_{n,\omega}}\times[k]\to V_{N_{n,\omega}}\times[d].$$
Since $kM_{n,\omega}\leq dN_{n,\omega}$  by construction, such a pairing exists, but some variable clones may go unpaired.
We are going to harness these unpaired `cavities' to set up a coupling of $\G_{n,\omega}$ and $\G_{n+1,\omega}$.

To this end, consider a further random factor graph $\hat\G_{n,\omega}$ with $N_{n,\omega}$ variable nodes and $\hat M_{n,\omega}=M_{n+1,\omega}-d$ constraint nodes.
The weight functions are chosen independently from $P$, and the connections between the constraint and variable nodes are induced by a random pairing 
	$$F_{\hat M_{n,\omega}}\times[k]\to V_{N_{n,\omega}}\times[d].$$
Rather than coupling  $\G_{n,\omega}$ and $\G_{n+1,\omega}$ directly, we will couple $\G_{n,\omega}$ and $\hat\G_{n,\omega}$ as well as $\G_{n+1,\omega}$ and $\hat\G_{n,\omega}$.

This construction leads to an approximate formula for the free energy of $\G_{n,\omega}$
that comes in terms of the kernel representation of the Boltzmann distribution of $\hat\G_{n,\omega}$.
To be precise, let $\hat\cC$ be the set of variables $v_i\in V_{N_{n,\omega}}$ with at least one unpaired clone in $\hat\G_{n,\omega}$.
Consider the random kernel
	$$\hat\rho_{n,\omega}=\dot\mu_{\hat\G_{n,\omega},\hat\cC}\in\meas$$
 representing  the joint distribution of the cavities $\hat\cC$.
Further, let $\hat\pi_{n,\omega}\in\Meas$ be the distribution of $\hat\rho_{n,\omega}$.
To deal with the conditioning on the event $\cS$, we also introduce versions $\GG_{n,\omega}$, $\hat\GG_{n,\omega}$ of the above random factor graphs conditional on $\cS$.
Let
	$\tilde\rho_{n,\omega}=\dot\mu_{\hat\GG_{n,\omega},\hat\cC}\in\meas$
be the kernel representation of the corresponding Boltzmann distribution, and let $\tilde\pi_{n,\omega}\in\Meas$
be the law of $\tilde\rho_{n,\omega}$.
In \Sec~\ref{Sec_Prop_Aizenman} we will derive the following formula.

\begin{proposition}\label{Prop_Aizenman}
For any $\eps>0$ there exists $\omega>0$ such that
	\begin{align*}
	\liminf_{n\to\infty}\, \Erw\brk{\ln\frac{Z(\G_{n+1,\omega})}{Z(\G_{n,\omega})}}-
		\Erw\brk{\ln\scal{\VARPHI_1}{\hat\rho_{n,\omega}}
		-\ln\scal{\bigoplus_{\hat M_{n,\omega}<i\leq M_{n,\omega}}\PSI_{1,i}}{\hat\rho_{n,\omega}}}&\geq-\eps,\\
	\liminf_{n\to\infty}\, \Erw\brk{\ln\frac{Z(\GG_{n+1,\omega})}{Z(\GG_{n,\omega})}}-
		\Erw\brk{\ln\scal{\VARPHI_1}{\tilde\rho_{n,\omega}}
		-\ln\scal{\bigoplus_{\hat M_{n,\omega}<i\leq M_{n,\omega}}\PSI_{1,i}}{\tilde\rho_{n,\omega}}}&\geq-\eps.
	\end{align*}
\end{proposition}

There are still two gaps to fill toward the proof of \Prop~\ref{MainAizLemReg}.
First, the estimate
of the free energy provided by \Prop~\ref{Prop_Aizenman} does not quite match the functional $\cB(\hat\pi_{n,\omega})$.
Second, the distribution $\hat\pi_{n,\omega}\in\Meas$ does not generally belong to the subspace $\MEAS$.
The following proposition deals with the second issue, which holds the key to resolving the first.
Recall that the topology of $\Meas$ is induced by the Wasserstein metric $\Cutm(\nix,\nix)$.
We introduce a relaxed version of $\MEAS$ by letting
	$$\MEAS_{\eps,N,M}=\cbc{\pi\in\Meas:\Cutm(\pi,\pi^{\Join(u,w)})\leq\eps\mbox{ for all $u\leq N$ and $w\leq M$}}.$$
Since (\ref{eqP}) and \Lem~\ref{Lemma_JoinCont} show that the map $\pi\mapsto\pi^{\Join(u,w)}$ is continuous, 
$\MEAS_{\eps,N,M}$ is a closed subspace of the compact Polish space $\Meas$.

\begin{proposition}\label{Prop_apxInv}
For any $\eps,L>0$ there is $\omega_0>0$ such that for every $\omega>\omega_0$ for large enough $n$ we have
	$$\hat\pi_{n,\omega},\tilde\pi_{n,\omega}\in \MEAS_{\eps,L,L}.$$
\end{proposition}

The proof of \Prop~\ref{Prop_apxInv} can be found in \Sec~\ref{Sec_apxInv}.
Finally, in \Sec~\ref{Sec_MainAizLemReg} we derive \Prop~\ref{MainAizLemReg} from \Prop s~\ref{Prop_Aizenman} and~\ref{Prop_apxInv}.

\subsection{Proof of \Prop~\ref{Prop_Aizenman}}\label{Sec_Prop_Aizenman}
We assume throughout that $\omega>\omega_0$ for a big  enough $\omega_0=\omega_0(d,P)$ and that $n$  sufficiently large.

Obtain the random factor graph $\G'_{n,\omega}$ from $\hat\G_{n,\omega}$ by adding $M_n-\hat M_{n,\omega}$ new random constraint nodes
$a_{i}$, $\hat M_{n,\omega}<i\leq M_n$, whose weight functions are drawn from $P$ independently and that are linked with the variable nodes via a random pairing with the cavities $\hat\cC$ of $\hat\G_{n,\omega}$.

Further, if $k\hat M_{n,\omega}\leq dN_{n,\omega}-d(k-1)$, then obtain $\G''_{n,\omega}$ from $\hat\G_{n,\omega}$ by adding one new variable node $\hat v=v_{N_{n,\omega}+1}$ along with $d$ random constraint nodes $\hat a_1,\ldots,\hat a_d$ adjacent to $\hat v$ whose weight functions are drawn independently from $P$.
To be precise,  the clones of $\hat v$ are paired each with a uniformly random clone $\hat\vh_i$ of $\hat a_i$ for $i=1,\ldots,d$, and the remaining $d(k-1)$ clones of the $\hat a_i$ are paired with randomly chosen cavities of $\hat\G_{n,\omega}$.
If $k\hat M_{n,\omega}> dN_{n,\omega}-d(k-1)$, then obtain $\G''_{n,\omega}$ from $\hat\G_{n,\omega}$ by just adding a new isolated variable node $\hat v$.

Obtain $\GG_{n,\omega}',\GG_{n,\omega}''$ analogously from $\hat\GG_{n,\omega}$ while conditioning on the event that the outcome is simple.
If it is impossible to add the required number of constraint nodes in such a way that the resulting factor graph is simple, then do not add any.

\begin{lemma}\label{Lemma_G'G''cpl}
For $\omega>0$ we have
	\begin{align}\nonumber
	\abs{\Erw\ln Z(\G_{n,\omega})-\Erw\ln Z(\G'_{n,\omega})}&=o(1),&
	\abs{\Erw\ln Z(\G_{n+1,\omega})-\Erw\ln Z(\G''_{n,\omega})}&= o_\omega(1),\\
	\abs{\Erw\ln Z(\GG_{n,\omega})-\Erw\ln Z(\GG'_{n,\omega})}&=o_\omega(1),&
	\abs{\Erw\ln Z(\GG_{n+1,\omega})-\Erw\ln Z(\GG''_{n,\omega})}&= o_\omega(1).
		\label{eqLemma_G'G''cpl1}
	\end{align}
\end{lemma}
\begin{proof}
Since $d\geq3$ and $k\geq2$,
	\begin{align*}
	\hat M_{n,\omega}-\lfloor dN_{n,\omega}/k\rfloor
		&\leq dN_{n,\omega}/k+1+d/k-d-\lfloor dN_{n,\omega}/k\rfloor
		\leq2-d(1-1/k)\leq1/2.
	\end{align*}
Because the left-hand side is an integer, we conclude that $\hat M_{n,\omega}\leq\lfloor dN_{n,\omega}/k\rfloor$.
Similarly,
	\begin{align*}
	\hat M_{n,\omega}-\bc{\lfloor d N_{n,\omega}/k\rfloor+\Delta_{n,\omega}-d\vX-\vY}
	&\leq 2-d(1-1/k)\leq1/2.
	\end{align*}
Thus, $M_{n,\omega}\geq\hat M_{n,\omega}$.
Hence, $\G_{n,\omega}$ and $\G'$ are identically distributed.

Moving on to the second claim, we consider the event $\cA$ that the last variable node is adjacent to precisely $d$ distinct constraint nodes.
Then 
	\begin{equation}\pr\brk{\G_{n+1,\omega}\in\cA}=1-O(\omega/n),\label{eqLemma_G'G''1}\end{equation}
while $\G_{n,\omega}''\in\cA$ with certainty.
Given $\cA$ and given that $\vX+\vY\leq\sqrt n$, say, the subgraph $\tilde\G_{n+1,\omega}$ obtained from $\G_{n+1,\omega}$ by deleting $\tilde v$ along with its adjacent constraint nodes is distributed precisely as $\hat\G_{n,\omega}$, and therefore $\G''_{n,\omega}$ and $\G_{n+1,\omega}$ can be coupled identically.
Hence,
	\begin{equation}\Erw\brk{\ln Z(\G''_{n,\omega})\mid \vX+\vY\leq\sqrt n}
			=\Erw\brk{\ln \G_{n+1,\omega}\mid\cA,\,\vX+\vY\leq\sqrt n}.\end{equation}
If, on the other hand, $\vX+\vY\leq\sqrt n$ but $\cA$ does not occur, then we can couple $\G_{n+1,\omega}$ and $\G''_{n,\omega}$ such that both disagree on at most $2d$ constraint nodes.
Indeed, suppose that $v_{N_{n+1,\omega}}$ has $\tilde d<d$ adjacent constraints in $\G_{n+1,\omega}$.
Then the subgraph obtained by removing $v_{N_{n+1,\omega}}$, its $\tilde d$ neighbors and another $d-\tilde d$ random constraint nodes is distributed precisely as $\hat\G_{n,\omega}$.
Hence, we can obtain both $\G_{n+1,\omega}$ and $\G_{n,\omega}''$ from $\hat\G_{n,\omega}$ by adding $d$ (possibly distinct) constraint nodes.
Thus, (\ref{eqP}) ensures that
	\begin{equation}\Erw\brk{\ln Z(\G''_{n,\omega})\mid\vX+\vY\leq\sqrt n}=\Erw\brk{\ln Z(\G_{n+1,\omega})\mid\overline{\cA},\,\vX+\vY\leq\sqrt n}+O(1).\label{eqLemma_G'G''3}\end{equation}
Furthermore, (\ref{eqP}) ensures that
		\begin{equation}\Erw\brk{\ln Z(\G''_{n,\omega})\mid\vX,\vY},\Erw\brk{\ln Z(\G_{n+1,\omega})\mid\vX,\vY}=O(n).\label{eqLemma_G'G''3a}\end{equation}
Since $\pr\brk{X+Y>\sqrt n}=o(n^{-2})$, 
(\ref{eqLemma_G'G''1})--(\ref{eqLemma_G'G''3a}) yield the second assertion.

Matters get slightly more complicated once we condition on $\cS$.
Since $\vX+\vY\leq\ln n$ with probability $1-O(n^{-k})$, due to (\ref{eqP}) the event $\vX+\vY>\ln n$ contributes no more than an additive $o(1)$ to the difference of the free energies.
Hence, we may condition on $\vX+\vY\leq\ln n$.
Let $\vec d'$ be the vector comprising the variable degrees in $\hat\GG_{n,\omega}$.
Let $\cD$ be the set of all such sequences with entries either $d$ or $d-1$.
A standard moment calculation shows that given any possible $\vec d'$, the event $\hat\G_{n,\omega}\in\cS$ has probability $(1+o(1))\exp\brk{-(d-1)(k-1)/2-\vecone\{k=2\}(d-1)^2/4}$ (cf.\ Fact~\ref{Fact_simple}).
Therefore, with $\tilde O(\nix)$ hiding poly-logarithmic terms,
	\begin{align}\label{eqDealWithSimple1a}
	\pr\brk{\vec d'\in\cD\mid\vX+\vY\leq\ln n}&=1-\tilde O(1/n).
	\end{align}
Similarly, let $\vec d$ comprise the variable degrees of the factor graph $\GG_{n,\omega}^-$ obtained from $\GG_{n,\omega}$ by deleting the last $d$ constraint nodes.
Then
	\begin{align}\label{eqDealWithSimple1b}
	\pr\brk{\vec d\in\cD\mid\vX+\vY\leq\ln n}&=1-\tilde O(1/n).
	\end{align}
Additionally, let $\cE$ be the set of all factor graphs that have a constraint node that is adjacent to variable nodes of degree less than $d$ only.
Then
	\begin{align}\label{eqDealWithSimple1c}
	\pr\brk{\hat\GG_{n,d}\in\cE\mid\vX+\vY\leq\ln n},\pr\brk{\GG_{n,\omega}^-\in\cE\mid\vX+\vY\leq\ln n}&=\tilde O(n^{1-k}).
	\end{align}
Further, on the event $\cD\setminus\cE$ we can couple $\GG_{n,\omega}'$ and $\GG_{n,\omega}$ identically,
because there is no way of adding the missing constraint nodes to $\hat\GG_{n,\omega}$ without obtaining a simple factor graph.
Hence,
	\begin{align}\label{eqDealWithSimple1}
	\Erw\brk{\ln Z(\GG_{n,\omega})\mid \vec d\in\cD,\,\GG_{n,\omega}^-\not\in\cE,\vX+\vY\leq\ln n}
		=\Erw\brk{\ln Z(\GG'_{n,\omega})\mid\vec{\hat d}\in\cD,\,\hat\GG_{n,\omega}\not\in\cE,\vX+\vY\leq\ln n}.
	\end{align}

But \eqref{eqDealWithSimple1} does not yet suffice to prove~\eqref{eqLemma_G'G''cpl1} because outside the event $\cD\setminus\cE$ the free energies of the two factor graphs may differ by $\Omega(n)$.
Hence, we also need to consider the event $\cD'$ that $\vec d$ has a single $d-2$ entry; this suffices because
	\begin{align}\label{eqDealWithSimple2}
	\pr\brk{\vec d\not\in\cD\cup\cD'\mid\vX+\vY\leq\ln n},\pr\brk{\vec{\hat d}\not\in\cD\cup\cD'\mid\vX+\vY\leq\ln n}&=1-\tilde O(n^{-2})
	\end{align}
and thus the contribution of the complement of $\cD\cup\cD'$ to the free energy difference is $o(1)$ due to (\ref{eqP}).
Considering the event $\cD'$ is indeed necessary because $\pr[\vec{\hat d}\in\cD'\mid\vX+\vY\leq\ln n]>\pr[\vec d\in\cD'\mid\vX+\vY\leq\ln n]$.
Indeed, while $\hat\GG_{n,\omega}$ is just a uniformly random simple factor graph with $N_{n,\omega}$ variable and $\hat M_{n,\omega}$ constraint nodes, $\GG_{n,\omega}^-$ has a tilted distribution, with each possible simple graph being weighed according to the number of extensions into a simple graph with $M_{n,\omega}$ constraints.
In effect, since variable nodes of degree less than $d-1$ leave us with fewer extensions, the event $\cD'$ is less likely in $\GG_{n,\omega}^-$.
Yet because $M_{n,\omega}-\hat M_{n,\omega}$ is bounded, 
on the event $\cD'$ we can couple $\GG_{n,\omega}$ and $\GG_{n,\omega}'$ such that both differ only in a bounded number of constraint nodes.
As a consequence,
	\begin{align}\label{eqDealWithSimple3}
	\Erw\brk{\ln Z(\GG_{n,\omega})\mid \vec d\in\cD',\,\GG_{n,\omega}^-\not\in\cE,\mid\vX+\vY\leq\ln n}
		=\Erw\brk{\ln Z(\GG'_{n,\omega})\mid\vec{\hat d}\in\cD',\,\hat\GG_{n,\omega}\not\in\cE,\mid\vX+\vY\leq\ln n}+O(1).
	\end{align}
Additionally, we claim that also $\GG_{n,\omega}$ given $\vec d\in\cD$ and $\GG'_{n,\omega}$ given $\vec{\hat d}\in\cD'$ can be coupled such that with probability $1-\tilde O(1/n)$ both differ only in $\tilde O(1)$ constraint nodes and that, in effect,
	\begin{align}\label{eqDealWithSimple4}
	\Erw\brk{\ln Z(\GG_{n,\omega})\mid \vec d\in\cD,\,\GG_{n,\omega}^-\not\in\cE,\mid\vX+\vY\leq\ln n}
			=\Erw\brk{\ln Z(\GG'_{n,\omega})\mid\vec{\hat d}\in\cD',\,\hat\GG_{n,\omega}\not\in\cE,\mid\vX+\vY\leq\ln n}+\tilde O(1).
	\end{align}
To see this, let $u_1,\ldots,u_\ell$ be the variables nodes of degree less than $d$ in $\GG_{n,\omega}^-$; suppose, indeed, that all of them have degree $d-1$.
Similarly, let $u_1',\ldots,u_{\ell-1}'$ be the cavities of $\hat\GG_{n,\omega}$, all of degree $d-1$ except for $u_{\ell-1}'$, which has degree $d-2$.
Pick a further variable node $u_\ell'$ of degree $d$ randomly.
Then with probability $1-\tilde O(n^{-1})$ the second neighborhoods $\partial^2\{u_1,\ldots,u_\ell\}$,
	$\partial^2\{u_1',\ldots,u_\ell'\}$ both have size $\ell(k-1)(d-1)$.
Consequently, the subgraphs of $\GG_{n,\omega}^-$ and $\hat\GG_{n,\omega}$ obtained by removing $u_1,\ldots,u_\ell$ and $u_1',\ldots,u_\ell'$ along with their neighbors, respectively, can be coupled such that both coincide with probability $1-\tilde O(1/n)$.
Thus, $\GG_{n,\omega}$ and $\GG_{n,\omega}'$ can be coupled such that the expected number of constraint nodes on which the two factor graphs differ is $\tilde O(1)$, whence we obtain (\ref{eqDealWithSimple4}).

To deal with the event $\cE$, we may assume that $k=2$ due to (\ref{eqDealWithSimple1a}).
Furthermore, because of (\ref{eqDealWithSimple1c}) and as
	\begin{align}\label{eqDealWithSimple8}
	\pr[\vec d'\not\in\cD\mid\hat\GG_{n,\omega}\in\cE,\mid\vX+\vY\leq\ln n]&=\tilde O(1/n),&
	\pr[\vec d\not\in\cD\mid\GG_{n,\omega}^-\in\cE,\mid\vX+\vY\leq\ln n]&=\tilde O(1/n),
	\end{align}	
we may assume that $\vec d,\vec d'\in\cD$.
Since 
	\begin{align}\label{eqDealWithSimple9}
	\pr\brk{\hat\GG_{n,\omega}\in\cE\mid\vec d'\in\cD,\vX+\vY\leq\ln n}&\geq
			\pr\brk{\GG_{n,\omega}^-\in\cE\mid\vec d\in\cD,\vX+\vY\leq\ln n}
	\end{align}	
because the event $\cE$ precludes certain extensions into a simple factor graph with $M_{n,\omega}$ constraints,
we just need to consider the case that $\hat\GG_{n,\omega}\in\cE$ and $\GG_{n,\omega}^-\not\in\cE$ given that $\vec d,\vec d'\in\cD$.
Let $u_1,\ldots,u_\ell$  and $u_1',\ldots,u_{\ell}'$ be the cavities of $\GG_{n,\omega}^-$ and $\hat\GG_{n,\omega}$, respectively.
Pick one further constraint node $b$ of $\hat\GG_{n,\omega}$.
Then with probability $1-\tilde O(1/n)$ the set $\partial^2\{u_1,\ldots,u_\ell\}$ has size $\ell(k-1)(d-1)$, and all variable nodes in this set have pairwise distance at least four.
The same is true of the set $\partial^2\{u_1',\ldots,u_\ell'\}\cup\partial b$ with probability $1-\tilde O(1/n)$.
If these two events occur, then $\GG_{n,\omega}$ and $\GG_{n,\omega'}$ can be coupled such that they only differ on the $\tilde O(1)$ constraint nodes that are adjacent to $u_1,\ldots,u_\ell$ and $u_1',\ldots,u_\ell'$ and $b$.
Hence, we obtain a coupling such that $\GG_{n,\omega}$ and $\GG_{n,\omega'}$  only differ on $\tilde O(1/n)$ variable nodes in expectation, and thus
	\begin{align}\label{eqDealWithSimple6}
	\Erw\brk{\ln Z(\GG_{n,\omega}')\mid \vec d'\in\cD,\,\hat\GG_{n,\omega}\in\cE,\vX+\vY\leq\ln n}
		&=\Erw\brk{\ln Z(\GG_{n,\omega})\mid \vec d\in\cD,\,\GG_{n,\omega}^-\not\in\cE,\vX+\vY\leq\ln n}+\tilde O(1).
	\end{align}
Moreover, because given $\cE$ there is precisely one constraint involving variables of degree $d-1$ only with probability $1-\tilde O(1/n)$, we obtain
	\begin{align}\label{eqDealWithSimple7}
	\Erw\brk{\ln Z(\GG_{n,\omega}')\mid \vec d'\in\cD,\,\hat\GG_{n,\omega}\in\cE,\vX+\vY\leq\ln n}
		&=\Erw\brk{\ln Z(\GG_{n,\omega})\mid \vec d\in\cD,\,\GG_{n,\omega}^-\in\cE,\vX+\vY\leq\ln n}+\tilde O(1).
	\end{align}
Combining (\ref{eqDealWithSimple6})--(\ref{eqDealWithSimple7}), we obtain the left bound stated in (\ref{eqLemma_G'G''cpl1}).

We proceed similarly to derive the right bound in \eqref{eqLemma_G'G''cpl1}.
Indeed, in this case we do not need to consider the event $\cE$ separately, because all additional constraint nodes are connected with a variable node that does not belong to $\hat\GG_{n,\omega}$ or $\GG_{n,\omega}^-$, respectively.
Hence, on the event $\cD$ we can couple $\GG_{n,\omega}''$ and $\GG_{n+1,\omega}$ identically, and thus
	\begin{align}\label{eqDealWithSimple10}
	\Erw\brk{\ln Z(\GG_{n+1,\omega})\mid\vec d\in\cD,\vX+\vY\leq\ln n}&=
		\Erw\brk{\ln Z(\GG_{n,\omega}'')\mid\vec d'\in\cD,\vX+\vY\leq\ln n}.
	\end{align}
In effect, due to \eqref{eqDealWithSimple2} we just need to construct a coupling in the event that $\hat\GG_{n,\omega}\in\cD'$ and $\hat\GG_{n,\omega}^-\in\cD$.
To this end, we proceed as above by coupling $\hat\GG_{n,\omega}$, $\GG_{n,\omega}^-$ given the second neighborhoods of the cavities such that both only differ in an expected $\tilde O(1/n)$ constraint nodes.
Since $\pr\brk{\vec d'\in\cD'\mid\vX+\vY\leq\ln n},\pr\brk{\vec d\in\cD'\mid\vX+\vY\leq\ln n}=\tilde O(1/n)$ and 
	$\pr\brk{\vec d\in\cD'\mid\vX+\vY\leq\ln n}\geq\pr\brk{\vec d'\in\cD'\mid\vX+\vY\leq\ln n}$, the second part of (\ref{eqLemma_G'G''cpl1}) follows from (\ref{eqDealWithSimple10}).
\end{proof}

We are ready to compare the free energies of $Z(\G'')$, $Z(\hat\G)$ and $Z(\G')$, $Z(\hat\G)$ and of the corresponding simple graphs.
We will carry the proofs out for the case of the simple random factor graph $\hat\GG$; the other case is simply obtained by skipping any deliberations pertinent to the event $\cS$.

\begin{lemma}\label{Lemma_G''}
We have
	\begin{align*}
\Erw\brk{\ln\frac{Z(\G_{n,\omega}'')}{Z(\hat\G_{n,\omega})}}&=\Erw\brk{\ln\scal{\VARPHI_1}{\hat\rho_{n,\omega}}}+o_\omega(1),&
\Erw\brk{\ln\frac{Z(\GG_{n,\omega}'')}{Z(\hat\GG_{n,\omega})}}&=\Erw\brk{\ln\scal{\VARPHI_1}{\tilde\rho_{n,\omega}}}+o_\omega(1).
\end{align*}
\end{lemma}
\begin{proof}
Let $\cA$ be the event that $\hat\GG_{n,\omega}$ has at least $\omega/2$ cavities, that all variable nodes have degree either $d$ or $d-1$ and that no two variable nodes of degree $d-1$ are adjacent to the same constraint node.
Then $\pr\brk{\cA}=1-\exp(-\Omega_\omega(\omega))$.
Hence, (\ref{eqP}) ensures that
	\begin{align}\label{eqLemma_G''_1}
	\Erw[\ln(Z(\GG_{n,\omega}'')/Z(\hat\GG_{n,\omega}))]&=\Erw[\ln(Z(\GG_{n,\omega}'')/Z(\hat\GG_{n,\omega}))|\cA]+o_\omega(1).
	\end{align}
Moreover, on $\cA$ the random factor graph $\GG_{n,\omega}''$ is obtained from $\GG_{n,\omega}$ by adding one variable node $\hat v$ along with $d$ constraint nodes $\hat a_1,\ldots,\hat a_d$, whose weight functions are drawn from $P$ independently.
Further, on the event $\cA$ all neighbors of the $\hat a_i$ except $\hat v$ belong to the set $\hat\cC$ of cavities.
Therefore, we have the exact formula
	\begin{align}
	\frac{Z(\GG_{n,\omega}'')}{Z(\hat\GG_{n,\omega})}&=
		\scal{\sum_{\chi\in\Omega}p(\chi)\prod_{i=1}^d\sum_{\tau\in\Omega^{\partial\hat a_i}}\psi_{\hat a_i}
			\vecone\{\tau_{\hat x}=\chi,\,\forall y\in\partial\hat a_i\setminus\hat x:\SIGMA_y=\tau_y\}}{\mu_{\hat\GG_{n,\omega},\hat\cC}}.
			\label{eqLemma_G''_2}
	\end{align}
To proceed, let $(\vv_{i,j})_{i,j}$ be a sequence of uniformly and independently chosen cavities $\vv_{i,j}\in\hat\cC$.
We claim that on the event $\cA$,
	\begin{align}\nonumber
	\Erw&\brk{\log\frac{Z(\GG_{n,\omega}'')}{Z(\hat \GG_{n,\omega})}\,\bigg|\,\hat\GG_{n,\omega}}\\
		&\qquad=o_\omega(1)+
		\Erw\brk{\ln\scal{\sum_{\chi\in\Omega}p(\chi)\prod_{i=1}^d\sum_{\tau\in\Omega^{k}}\PSI_{i}(\tau)
			\vecone\{\tau_{\vh_i}=\chi,\,\forall h\in[k]\setminus \vh_i:\SIGMA_{\vv_{i,h}}=\tau_h\}}{\mu_{\hat\GG_{n,\omega},\hat\cC}}\,\bigg|\,\hat\GG_{n,\omega}}.			\label{eqLemma_G''_3}
	\end{align}
Indeed, the only difference between
(\ref{eqLemma_G''_2}) and (\ref{eqLemma_G''_3}) is that in the former the neighbours $\partial\hat a_i\setminus\hat v$ are chosen from $\hat\cC$ without replacement, whereas the $\vv_{i,j}$ are chosen independently, i.e., with replacement.
But since we choose a mere $dk$ cavities $(\vv_{i,j})_{i\in[d], j\in[k]}$ out of a total of at least $\omega/2$, the probability of hitting the same cavity twice is $o_\omega(1)$, and thus (\ref{eqLemma_G''_3}) follows from (\ref{eqP}).
Further, unravelling the definitions of $\tilde\rho_{n,\omega}$ and $\VARPHI_1$, we see that 
	\begin{align}
	\Erw&\brk{\ln\scal{\sum_{\chi\in\Omega}p(\chi)\prod_{i=1}^d\sum_{\tau\in\Omega^{k}}\PSI_{i}(\tau)
			\vecone\{\tau_{\vh_i}=\chi,\,\forall h\in[k]\setminus \vh_i:\SIGMA_{\vv_{i,h}}=\tau_h\}}{\mu_{\hat\GG_{n,\omega},\hat\cC}}\,\bigg|\,\hat\GG_{n,\omega}}\nonumber\\
			&\qquad\qquad\qquad=\Erw\brk{\log\scal{\VARPHI_1}{\tilde\rho_{n,\omega}}\mid\hat\GG_{n,\omega}}.
						\label{eqLemma_G''_4}
	\end{align}
Finally, the assertion follows from (\ref{eqLemma_G''_1})--(\ref{eqLemma_G''_4}) by taking the expectation on $\hat\GG_{n,\omega}$.
\end{proof}

\begin{lemma}\label{Lemma_G'}
We have
	\begin{align*}
	\Erw\brk{\ln\frac{Z(\G_{n,\omega}')}{Z(\hat\G_{n,\omega})}}&=
			\Erw\scal{\ln\bigoplus_{\hat M_{n,\omega}<i\leq M_{n,\omega}}\PSI_{1,i}}{\hat\rho_{n,\omega}}+o_\omega(1),&
	\Erw\brk{\ln\frac{Z(\GG_{n,\omega}')}{Z(\hat\GG_{n,\omega})}}&=
			\Erw\scal{\ln\bigoplus_{\hat M_{n,\omega}<i\leq M_{n,\omega}}\PSI_{1,i}}{\tilde\rho_{n,\omega}}+o_\omega(1).
	\end{align*}
\end{lemma}
\begin{proof}
The proof is similar in spirit to the previous one.
Once more we consider the event $\cA$ that $\hat\GG_{n,\omega}$ has at least $\omega/2$ cavities, that all variable nodes have degree either $d$ or $d-1$ and that no two variable nodes of degree $d-1$ are adjacent to a common constraint node.
Then $\pr\brk{\cA}=1-\exp(-\Omega_\omega(\omega))$ and
	\begin{align}\label{eqLemma_G'_1}
	\Erw[\ln(Z(\GG_{n,\omega}')/Z(\hat\GG_{n,\omega}))]&=\Erw[\ln(Z(\GG_{n,\omega}'')/Z(\hat\GG_{n,\omega}))|\cA]+o_\omega(1).
	\end{align}
Moreover, we have the pointwise exact formula
	\begin{align}\label{eqLemma_G'_2}
	\frac{Z(\GG_{n,\omega}')}{Z(\hat\GG_{n,\omega})}&=\scal{\prod_{\hat M_{n,\omega}<i\leq M_{n,\omega}}
		\sum_{\tau\in\Omega^{\partial a_i}}\psi_{a_i}(\tau)\vecone\{\SIGMA_{\partial a_i}=\tau\}}{\mu_{\hat\GG_{n,\omega},\hat\cC}}.
	\end{align}
With $(\vv_{i,j})_{i,j}$ a sequence of independently chosen cavities $\vv_{i,j}\in\hat\cC$, we claim that on $\cA$,	
	\begin{align}			\label{eqLemma_G'_3}
	\Erw\brk{\log\frac{Z(\GG_{n,\omega}')}{Z(\hat \GG_{n,\omega})}\,\bigg|\,\hat\GG_{n,\omega}}&=o_\omega(1)+
		\Erw\brk{\ln\scal{\prod_{i=1}^{M_{n,\omega}-\hat M_{n,\omega}}
		\sum_{\tau\in\Omega^{k}}\PSI_{i}(\tau)\prod_{j=1}^k\vecone\{\SIGMA_{\vv_{i,j}}=\tau_j\}}{\mu_{\hat\GG_{n,\omega},\hat\cC}}\,\bigg|\,\hat\GG_{n,\omega}}.
	\end{align}
Indeed, the only difference is that in (\ref{eqLemma_G'_3}) the $\vv_{i,j}$ are chosen independently, whereas in (\ref{eqLemma_G'_2}) the neighbors of the
$a_i$ are chosen without replacement.
But since $M_{n,\omega}-\hat M_{n,\omega}$ is bounded while there are at least $\omega/2$  cavities, the two terms coincide up to $o_\omega(1)$.
Finally, the construction of $\tilde\rho_{n,\omega}$ ensures that
	\begin{align}			\label{eqLemma_G'_4}
	\Erw\brk{\ln\scal{\prod_{i=1}^{M_{n,\omega}-\hat M_{n,\omega}}
		\sum_{\tau\in\Omega^{k}}\PSI_{i}(\tau)\prod_{j=1}^k\vecone\{\SIGMA_{\vv_{i,j}}=\tau_j\}}{\mu_{\hat\GG_{n,\omega},\hat\cC}}\,\bigg|\,\hat\GG_{n,\omega}}
			&=\Erw\brk{\ln\scal{\bigoplus_{i=1}^{M_{n,\omega}-\hat M_{n,\omega}}\PSI_{1,i}}{\hat\rho_{n,\omega}}\,\bigg|\,\hat\GG_{n,\omega}},
	\end{align}
and thus the assertion follows from (\ref{eqLemma_G'_1})--(\ref{eqLemma_G'_4}) by taking the expectation.
\end{proof}

\noindent
Finally, \Prop~\ref{Prop_Aizenman} is an immediate consequence of \Lem s~\ref{Lemma_G'G''cpl}, \ref{Lemma_G''} and~\ref{Lemma_G'}.

\subsection{Proof of \Prop~\ref{Prop_apxInv}}\label{Sec_apxInv}
Once more we will carry the proof out for the simple random factor graph, which is the (slightly) more intricate case;
the unconditional case follows by skipping any considerations pertaining to the conditioning.
The basic idea behind the proof of \Prop~\ref{Prop_apxInv} is quite simple.
With probability $1-o_\omega(1)$
the random graph $\hat\GG_{n,\omega}$ consists of $N_{n,\omega}$ variable and $\hat M_{n,\omega}$ constraint nodes and we have $N_{n,\omega}=n-\vX$ and 
	$$\hat M_{n,\omega}=M_{n+1,\omega}-d=\lfloor d N_{n,\omega}/k\rfloor\wedge\bc{\lfloor d N_{n,\omega}/k\rfloor+\Delta_{n,\omega}-d\vX-\vY}-d$$
with independent $\Po(\omega)$ variables $\vX,\vY$.
Fix two integers $\ell,\ell'\geq0$.
Given that $\vX\geq\ell$ and $k\hat M_{n,\omega}\leq dN_{n,\omega}-k\ell'-d(k-1)\ell$, let $\hat\GG_{n,\omega}\brk{\ell,\ell'}$ be the random factor graph obtained from $\hat\GG_{n,\omega}$ by adding
	\begin{itemize}
	\item $\ell$ more variable nodes $\hat v_1=v_{N_{n,\omega}+1},\ldots,\hat v_\ell=v_{N_{n,\omega}+\ell}$ along with $d\ell$ new constraint nodes
		$\hat a_{i,j}$, $i\in[\ell]$, $j\in[d]$, each with a weight function chosen from $P$ independently;
		connect a random clone $\hat\vh_{i,j}$ of each $\hat a_{i,j}$ with a random clone of $\hat v_i$ and pair the other $k-1$ clones of $\hat a_{i,j}$ with random cavities of $\hat\GG_{n,\omega}$ left pending by the previous additions.
	\item $\ell'$ more constraint nodes $\hat a_1,\ldots,\hat a_{\ell'}$, each endowed with a weight function chosen from $P$ independently and each connected with $k$ random cavities of $\hat\GG_{n,\omega}$ left vacant by the previous operations.
	\end{itemize}
The resulting random factor graph $\hat\GG_{n,\omega}\brk{\ell,\ell'}$ is not necessarily simple.
Yet the key insight behind \Prop~\ref{Prop_apxInv} is that for any $\ell,\ell'$ the distribution of $\hat\GG_{n,\omega}\brk{\ell,\ell'}$ is close to that of the original graph $\hat\GG_{n,\omega}$, provided that $\omega$ is big enough.
Moreover, the perturbation of the Boltzmann distribution that ensues upon going from $\hat\GG_{n,\omega}$ to $\hat\GG_{n,\omega}\brk{\ell,\ell'}$ is close to the perturbation induced by the $\Join(\ell,\ell')$-operation.
We introduce similar notation $\hat\G_{n,\omega}[\ell,\ell']$ for the random graph without the conditioning on $\cS$.

To formalize this idea, we first compare the distributions of $\hat\GG_{n,\omega}$ and $\hat\GG_{n,\omega}\brk{\ell,\ell'}$.
For integers $x,y$ we denote by $\hat\GG_{n,\omega,x,y}$ the conditional $\hat\GG_{n,\omega}$ given that $\vX=x$ and $\vY=y$.

\begin{lemma}\label{Lemma_perturb1}
For any $\ell,\ell'\geq0$  we have
	$\dTV\bc{\hat\GG_{n,\omega},\hat\GG_{n,\omega}\brk{\ell,\ell'}}=o_\omega(1)$ and analogously
 $\dTV\bc{\hat\G_{n,\omega},\hat\G_{n,\omega}\brk{\ell,\ell'}}=o_\omega(1)$.
\end{lemma}
\begin{proof}
The event 
	$\cA=\{\omega/2\leq\vX\leq2\omega,\, \omega/2\leq\vY\leq2\omega\}$
 has probability $1-o_\omega(1)$.
Further, because $\vX,\vY$ are independent Poisson variables with a large mean $\omega$ while $\ell,\ell'$ are fixed, the total variation distance of the pairs $(\vX,\vY)$ and $(\vX-\ell,\vY-\ell')$ is of order $O_\omega(\omega^{-1/2})$.
Hence, given $\cA$ the total variation distance of $\hat\GG_{n,\omega}$ and $\hat\GG_{n,\omega,\vX-\ell,\vY-\ell'}$ is $o_\omega(1)$;  in symbols,
	\begin{equation}\label{eqLemma_perturb1_1}
	\dTV\bc{\hat\GG_{n,\omega}\mid\cA,\hat\GG_{n,\omega,\vX-\ell,\vY-\ell'}\mid\cA}=o_\omega(1).
	\end{equation}
Further, let $\cA'$ be the event that $\hat\GG_{n,\omega}$ enjoys the following additional properties.
	\begin{enumerate}[(i')]
	\item The last $\ell$ variable nodes of $\hat\GG_{n,\omega}$ satisfy
			$\abs{\partial^2\{v_{N_{n,\omega}-\ell+1},\ldots,v_{N_{n,\omega}}\}\setminus\hat\cC}=\ell d(k-1).$
		Hence, there are $\ell d(k-1)$ distinct second neighbors, none of which is a cavity.
	\item The last $\ell'$ constraint nodes of $\hat\GG_{n,\omega}$ satisfy
			$\abs{\partial\{a_{\hat M_{n,\omega}-\ell'+1},\ldots,a_{\hat M_{n,\omega}}\}\setminus\hat\cC}=k\ell'$.
				Hence, there are $k\ell' $ distinct second neighbors, none of them a cavity.
	\item We have $\partial\{v_{N_{n,\omega}-\ell+1},\ldots,v_{N_{n,\omega}}\}\cap\{a_{\hat M_{n,\omega}-\ell'+1},\ldots,a_{\hat M_{n,\omega}}\}=\emptyset$.
	\item Let 
			$$\cU=\hat\cC\cup\partial^2\{v_{N_{n,\omega}-\ell+1},\ldots,v_{N_{n,\omega}}\}\cup
					\partial\{a_{\hat M_{n,\omega}-\ell'+1},\ldots,a_{\hat M_{n,\omega}}\}.$$
		Then for any constraint node 
		$a\not\in \partial\{v_{N_{n,\omega}-\ell+1},\ldots,v_{N_{n,\omega}}\}\cup a_{\hat M_{n,\omega}-\ell'+1},\ldots,a_{\hat M_{n,\omega}}$ we have $|\partial a\cap\cU|\leq1$.
		Thus, only the constraint nodes adjacent to the last $\ell$ variable nodes or the $a_i$ with $i>\hat M_{n,\omega}-\ell'$ may be adjacent to more than one variable node in $\cU$.
	\item All variable nodes $u\in\cU$ have degree $d$ or $d-1$.
	\end{enumerate}
Additionally, let $\cA''$ be the event that  $\hat\GG_{n,\omega,\vX-\ell,\vY-\ell'}$ has the following properties.
	\begin{enumerate}[(i'')]
	\item all variable nodes have degree either $d$ or $d-1$.
	\item no two variable nodes of degree $d-1$ are adjacent to the same constraint node.
	\end{enumerate}
Then
	\begin{align}\label{eqLemma_perturb1_1a}
	\pr\brk{\hat\GG_{n,\omega}\in\cA'\mid\cA}&=1-o_\omega(1),&
			\pr\brk{\hat\GG_{n,\omega,\vX-\ell,\vY-\ell'}\in\cA''\mid\cA}&=1-o_\omega(1).
	\end{align}	
Furthermore,  given $\cA''\cap\cA$, the random factor graph $\hat\GG_{n,\omega,\vX-\ell,\vY-\ell'}[\ell,\ell']$ obtained by attaching $\ell$ new variable nodes and $\ell'$ new constraint nodes is distributed precisely as
$\hat\GG_{n,\omega}$ given $\cA'\cap\cA$.
Indeed, the construction of the enhanced factor graph $\hat\GG_{n,\omega,\vX-\ell,\vY-\ell'}[\ell,\ell']$ expressly ensures that (i')--(iii') are satisfied, and (iv')--(v') follow from (i'')--(ii'').
Hence, (\ref{eqLemma_perturb1_1a}) yields
	\begin{equation}\label{eqLemma_perturb1_2}
	\dTV\bc{\hat\GG_{n,\omega}\mid\cA,\hat\GG_{n,\omega,\vX-\ell,\vY-\ell'}[\ell,\ell']\mid\cA}=o_\omega(1).
	\end{equation}
Finally, since $\pr\brk{\cA}=1-o_\omega(1)$, the assertion follows from (\ref{eqLemma_perturb1_1}) and (\ref{eqLemma_perturb1_2}).
\end{proof}

Let $\hat\cC[\ell,\ell']$ be the set of cavities of $\hat\G_{n,\omega}\brk{\ell,\ell'}$ and
 let $\hat\rho_{n,\omega}[\ell,\ell']\in\meas$ be the kernel representing $\mu_{\hat\G_{n,\omega}\brk{\ell,\ell'},\hat\cC[\ell,\ell']}$.
Let $\hat\pi_{n,\omega}\brk{\ell,\ell'}\in\Meas$ be the distribution of $\hat\rho_{n,\omega}[\ell,\ell']$.
Define $\tilde\cC[\ell,\ell']$, $\tilde\rho_{n,\omega}[\ell,\ell']$ analogously for $\hat\GG_{n,\omega}$.

\begin{lemma}\label{Lemma_perturb2}
For any $\ell,\ell'\geq0$  we have 
	\begin{align*}
	\Cutm\bc{\hat\pi_{n,\omega}^{\Join(\ell,\ell')},\hat\pi_{n,\omega}\brk{\ell,\ell'}}&=o_\omega(1),&
\Cutm\bc{\tilde\pi_{n,\omega}^{\Join(\ell,\ell')},\tilde\pi_{n,\omega}\brk{\ell,\ell'}}=o_\omega(1).
\end{align*}
\end{lemma}
\begin{proof}
The event $\cA_1=\cbc{\omega/2\leq X,Y\leq 2\omega}$ occurs with probability $1-o_\omega(1)$.
So does the event $\cA_2$ that all variable nodes of $\hat\GG_{n,\omega}$ have degree either $d$ or $d-1$, and thus the same is true of $\cA=\cA_1\cap\cA_2$.
Moreover, the construction of $\hat\GG_{n,\omega}[\ell,\ell']$ is such that on the event $\cA$ we have the exact formula
	\begin{align*}
	\frac{Z(\hat\GG_{n,\omega}[\ell,\ell'])}{Z(\hat\GG_{n,\omega})}&=\scal{\prod_{i=1}^\ell\varphi_i(\SIGMA)\prod_{i=1}^{\ell'}\psi_{\hat a_i}(\SIGMA)}{\mu_{\hat\GG_{n,\omega},\hat\cC}},\qquad\mbox{where}\\
		\varphi_i(\sigma)&=\sum_{\chi\in\Omega}p(\chi)\prod_{j=1}^d\sum_{\tau\in\Omega^{\partial \hat a_{i,j}}}
			\psi_{\hat a_{i,j}}(\tau)\vecone\{\tau_{\hat v_i}=\chi,\,\forall w\in\partial\hat a_{i,j}\setminus\hat v_i:\sigma_w=\tau_w \}.
	\end{align*}
Consequently, the joint distribution  of the cavities   $\tilde\cC[\ell,\ell']$ of $\hat\GG_{n,\omega}[\ell,\ell']$ reads
	\begin{align}\label{eqLemma_perturb2_1}
	\mu_{\hat\GG_{n,\omega}[\ell,\ell'],\tilde\cC[\ell,\ell']}(\sigma)&=\frac{
		\scal{\vecone\{\forall u\in\tilde\cC[\ell,\ell']:\SIGMA_u=\sigma_u\}\prod_{i=1}^\ell\varphi_i(\sigma)\prod_{i=1}^{\ell'}\psi_{\hat a_i}(\sigma)}{\mu_{\hat\GG_{n,\omega}}(\sigma)}}
			{\scal{\prod_{i=1}^\ell\varphi_i(\SIGMA)\prod_{i=1}^{\ell'}\psi_{\hat a_i}(\SIGMA)}{\mu_{\hat\GG_{n,\omega},\tilde\cC}}}\qquad(\sigma\in\Omega^{\tilde\cC[\ell,\ell']}).
	\end{align}

Thus,  with probability $1-o_\omega(1)$, namely on the event $\cA$,
$\tilde\rho_{n,\omega}[\ell,\ell']$ is just the kernel representing the right hand side of (\ref{eqLemma_perturb2_1}) 
We claim that in this case $\tilde\rho_{n,\omega}[\ell,\ell']$ and $\tilde\rho_{n,\omega}^{\Join(\ell,\ell')}$ can be coupled to coincide with probability $1-o_\omega(1)$.
Indeed, the weight functions associated with the $\hat a_i$ and the $\hat a_{i,j}$ are chosen from $P$ independently, and they are connected to the cavities of $\hat\GG_{n,\omega}$ by a random pairing.
By comparison, we construct $\tilde\rho_{n,\omega}^{\Join(\ell,\ell')}$ by adjoining $\VARPHI_1,\ldots,\VARPHI_\ell$ and $\PSI_1,\ldots,\PSI_{\ell'}$ that evaluate the kernel $\tilde\rho_{n,\omega}$ at independent uniformly random points of the unit interval.
Combinatorially, this is equivalent to attaching the new variable and constraint nodes to random cavities chosen with replacement,
rather than without replacement as in the construction of $\hat\GG_{n,\omega}[\ell,\ell']$.
But since the number of cavities of $\hat\GG$ is $\Omega_\omega(\omega)$, the two constructions have total variation distance $o_\omega(1)$.
\end{proof}

\begin{proof}[Proof of \Prop~\ref{Prop_apxInv}]
The proposition is immediate from \Lem s~\ref{Lemma_perturb1} and~\ref{Lemma_perturb2}.
\end{proof}

\subsection{Proof of \Prop~\ref{MainAizLemReg}}\label{Sec_MainAizLemReg}
We begin with the following lemma, whose proof is similar to the proof of \Lem~\ref{Lemma_Prop_regint2_B''}.

\begin{lemma}\label{Lemma_G'G''}
We have
	\begin{align*}
	\Erw\brk{\ln\scal{\VARPHI_1}{\hat\pi_{n,\omega}}
		-\ln\scal{\bigoplus_{\hat M_{n,\omega}<i\leq M_{n,\omega}}\PSI_{1,i}}{\hat\pi_{n,\omega}}}&=\cB(\hat\pi_{n,\omega})+o_\omega(1),\\
	\Erw\brk{\ln\scal{\VARPHI_1}{\tilde\pi_{n,\omega}}
		-\ln\scal{\bigoplus_{\hat M_{n,\omega}<i\leq M_{n,\omega}}\PSI_{1,i}}{\tilde\pi_{n,\omega}}}&=\cB(\tilde\pi_{n,\omega})+o_\omega(1).
	\end{align*}
\end{lemma}
\begin{proof}
Let $\pi=\hat\pi_{n,\omega}$ or $\pi=\tilde\pi_{n,\omega}$.
Since $\Erw[M_{n,\omega}-\hat M_{n,\omega}]=d(k-1)/k+o_\omega(1)$, due to \eqref{eqP} it suffices to show that
	\begin{align}\label{eqLemma_G'G''_1}
	\Erw\brk{\ln\scal{\bigoplus_{\hat M_{n,\omega}<i\leq M_{n,\omega}}\PSI_{1,i}}{\pi}}
			&=o_{\omega}(1)+\Erw[M_{n,\omega}-\hat M_{n,\omega}]\cdot
				\Erw\brk{\ln\scal{\PSI_{1,1}}{\pi}}.
	\end{align}
Thus, we need to cope with the correlations between $M_{n,\omega}-\hat M_{n,\omega}$ and $\hat\G_{n,\omega}$ or  $\hat\GG_{n,\omega}$, respectively.
In other words, we need to assess the correlations between $M_{n,\omega}-\hat M_{n,\omega}$ and $\hat N_{n,\omega}$, $\hat M_{n,\omega}$.
With probability  $1-\exp(-\Omega_\omega(\omega))$ we have
	\begin{align}\label{eqW}
	M_{n,\omega}-\hat M_{n,\omega}&=W,\qquad\mbox{where}&
	W&=\lfloor d N_{n,\omega}/k\rfloor-\lfloor d(1+N_{n,\omega})/k\rfloor+\Delta_{n,\omega}-\Delta_{n+1,\omega}+d,
	\end{align}
with independent Bernoulli variables $\Delta_{n,\omega},\Delta_{n+1,\omega}$.
Thus, $W\leq d+1$ and \eqref{eqP} ensures that
	\begin{align}	\nonumber
	\Erw&\brk{\ln\scal{\bigoplus_{\hat M_{n,\omega}<i\leq M_{n,\omega}}\PSI_{1,i}}{\pi}}=
		\Erw\brk{\ln\scal{\bigoplus_{i=1}^W\PSI_{1,i}}{\pi}}+o_\omega(1)\\
		&=\Erw\brk{\vecone\cbc{\abs{\hat N_{n,\omega}-\Erw[\hat N_{n,\omega}]},\abs{\hat M_{n,\omega}-\Erw[\hat M_{n,\omega}]}
			\leq\sqrt{\omega\ln\omega}}\ln\scal{\bigoplus_{i=1}^W\PSI_{1,i}}{\pi}}+o_\omega(1).
			\label{eqLemma_G'G''_2}
		\end{align}
Furthermore, since $\vX,\vY$ are independent Poisson variables with mean $\omega$ while $W$ is bounded, for any $\hat n,\hat m$ such that
$|\hat n-\Erw[\hat N_{n,\omega}]|,|\hat m-\Erw[\hat M_{n,\omega}]|\leq\sqrt{\omega\ln\omega}$ we obtain from (\ref{eqW}) that 
	\begin{align*}
	\pr\brk{\hat M_{n,\omega}=\hat m\mid\hat N_{n,\omega}=\hat n}&=(1+o_\omega(1))\pr\brk{\hat M_{n,\omega}=\hat m\mid N_{n,\omega}=\hat n,\, W=h}\qquad\mbox{for any }0\leq h\leq d+1.
	\end{align*}
Hence, introducing an independent copy $W'$ of $W$, we obtain from (\ref{eqP}) and (\ref{eqLemma_G'G''_2}) that
	\begin{align}\label{eqLemma_G'G''_3}
	\Erw\brk{\ln\scal{\bigoplus_{\hat M_{n,\omega}<i\leq M_{n,\omega}}\PSI_{1,i}}{\pi}}
		&=\Erw\brk{\ln\scal{\bigoplus_{i=1}^{W'}\PSI_{1,i}}{\pi}}+o_\omega(1).
		\end{align}
Additionally, we claim that for any $0\leq w\leq d+1$,
	\begin{align}\label{eqLemma_G'G''_4}
	\Erw\brk{\ln\frac{\scal{\bigoplus_{i=1}^{w+1}\PSI_{1,i}}{\pi}}{\scal{\bigoplus_{i=1}^{w}\PSI_{1,i}}{\pi}}}
		&=\Erw\brk{\ln\scal{\PSI_{1,1}}{\pi}}+o_\omega(1).
	\end{align}
Indeed, as in the proof of \Lem~\ref{Lemma_Prop_regint2_B''} we obtain
	\begin{align}\label{eqLemma_Prop_regint2_B''_5}
	\Erw\brk{\bc{\frac{\scal{\bigoplus_{i=1}^{w+1}\PSI_i}{\pi}}{\scal{\bigoplus_{i=1}^{w}\PSI_i}{\pi}}}^\ell}
		&=\Erw\brk{\scal{\PSI_{w+1}}{\bigoplus_{i=1}^w\PSI_i\Join\pi}^\ell}.
	\end{align}
Further, (\ref{eqP}), \Cor~\ref{Cor_cntFct2} and \Prop~\ref{Prop_apxInv} yield
	\begin{align}\label{eqLemma_Prop_regint2_B''_5a}
	\Erw\brk{\scal{\PSI_{w+1}}{\bigoplus_{i=1}^w\PSI_i\Join\pi}^\ell}&=\Erw\brk{\scal{\PSI_{w+1}}{\pi}^\ell}		+o_\omega(1).
	\end{align}
As the logarithm can be approximated arbitrarily well by polynomials due to (\ref{eqP}),
(\ref{eqLemma_G'G''_4}) follows from (\ref{eqLemma_Prop_regint2_B''_5})--(\ref{eqLemma_Prop_regint2_B''_5a}).
Finally, (\ref{eqLemma_G'G''_1}) follows from (\ref{eqLemma_G'G''_3}) and (\ref{eqLemma_G'G''_4}).
\end{proof}

\begin{proof}[Proof of \Prop~\ref{MainAizLemReg}]
\Prop~\ref{Prop_apxInv} and \Lem~\ref{Lemma_G'G''} show that for any $\ell\geq1$ there exists $\omega_{\ell}>\omega_{\ell-1}$ such that for all sufficiently large $n$ we have $\hat\pi_{n,\omega_\ell}\in\MEAS_{1/\ell,\ell,\ell}$ and
	\begin{align}\label{eqMainAizLemReg_1}
	\frac{1}{n} \Erw \log Z (\G) &\ge\mathcal B(\hat\pi_{n,\omega_\ell})-1/\ell.
	\end{align}
Since $\Meas$ is compact,  the sequence $(\hat\pi_{n,\omega_\ell})_n$ has a convergent subsequence, whose limit $\hat\pi^{(\ell)}$ lies in the closed set $\MEAS_{1/\ell,\ell,\ell}$.
Furthermore, because \Lem~\ref{Lemma_cntFct} shows that $\cB(\nix)$ is continuous, (\ref{eqMainAizLemReg_1}) yields
	\begin{align}\label{eqMainAizLemReg_2}
	\liminf_{n\to\infty}\frac{1}{n} \Erw \log Z (\G) &\ge\cB(\hat\pi^{(\ell)})-1/\ell.
	\end{align}
Additionally, $(\hat\pi^{(\ell)})_\ell$ has a subsequence that converges to $\tilde\pi\in\MEAS=\bigcap_\ell\MEAS_{1/\ell,\ell,\ell}$.
Thus, the first assertion follows from (\ref{eqMainAizLemReg_2}) and the continuity of $\cB(\nix)$ established by \Cor~\ref{Cor_Bcont}.

The second assertion concerning the simple random factor graph $\GG$ is immediate from the first.
Indeed, due to \eqref{eqP} a standard application of Azuma's inequality shows that
$n^{-0.51}|\ln Z(\G)-\Erw\ln Z(\G)|\to0$ in probability.
Hence, Fact~\ref{Fact_simple} and Bayes' rule imply that $\Erw\ln Z(\GG)-\Erw\ln Z(\G)=o(n)$.
\end{proof}

\subsection{Proof of \Thm~\ref{thmBethePlus}}\label{Sec_thmBethePlus}
We begin by showing that the free energy of $\G$ can be expressed in terms of the functional $\cB$ applied to $\hat\pi_{n,\omega}$ or $\tilde\pi_{n,\omega}$, respectively.
Once more we will carry the details out for $\GG$; the unconditioned random factor graph $\G$ is easier to deal with, and the proofs are just obtained from the $\GG$ case by dropping any considerations regarding multiple edges.

\begin{lemma}\label{Lemma_thmBethePlus1}
If {\bf POS} is satisfied, then
	\begin{align*}
	\lim_{n\to\infty}\frac1n\Erw\ln Z(\G)&=\liminf_{\omega\to\infty}\,\liminf_{n\to\infty}\cB(\hat\pi_{n,\omega}),&
\lim_{n\to\infty}\frac1n\Erw\ln Z(\GG)&=\liminf_{\omega\to\infty}\,\liminf_{n\to\infty}\cB(\tilde\pi_{n,\omega}).
	\end{align*}
\end{lemma}
\begin{proof}
\Prop~\ref{Prop_Aizenman} and \Lem~\ref{Lemma_G'G''} show that for any $\eps>0$ there exists $\omega_0$ such that for all $\omega>\omega_0$ there exists $n_0$ such that for all $n>n_0$ we have	$n^{-1}\Erw\log Z(\GG)\geq\cB(\tilde\pi_{n,\omega})-\eps.$
Hence, for any $\eps>0$ there is $\omega_0>0$ such that for all $\omega>\omega_0$ we have
	\begin{align}\label{eqLemma_thmBethePlus1_2}
	\liminf_{n\to\infty}\frac1n\Erw\log Z(\GG)&\geq\liminf_{n\to\infty}\cB(\tilde\pi_{n,\omega})-\eps.
	\end{align}
Indeed, since \Prop s~\ref{Prop_regint} and~\ref{MainAizLemReg} show that $(\frac1n\Erw\log Z(\GG))_n$ converges, (\ref{eqLemma_thmBethePlus1_2}) yields
	\begin{align}\label{eqLemma_thmBethePlus1_1}
	\lim_{n\to\infty}\frac1n\Erw\log Z(\GG)=\liminf_{n\to\infty}\frac1n\Erw\log Z(\GG)&\geq\liminf_{\omega\to\infty}\,\liminf_{n\to\infty}\cB(\tilde\pi_{n,\omega}).
	\end{align}

We are left to prove the converse inequality.
The space $\Meas$ is compact and separable.
Therefore, for any $\omega$ the sequence $(\tilde\pi_{n,\omega})_n$ has a subsequence that converges to $\tilde\pi^{(\omega)}\in\Meas$ such that
	$\liminf_{n\to\infty}\cB(\tilde\pi_{n,\omega})=\cB(\tilde\pi^{(\omega)})$.
Further, $(\tilde\pi^{(\omega)})_\omega$ has a subsequence that converges to $\tilde\pi^*$ such that
	\begin{align}\label{eqLemma_thmBethePlus1_3a}
	\liminf_{\omega\to\infty}\cB(\tilde\pi^{(\omega)})&=\cB(\tilde\pi^{*}).
	\end{align}
\Prop~\ref{Prop_apxInv} shows that $\tilde\pi^{*}\in\MEAS$.
Hence, \Prop~\ref{Prop_regint} implies that
	\begin{align}\label{eqLemma_thmBethePlus1_3}
	\lim_{n\to\infty}\frac1n\Erw\log Z(\GG)&\leq\cB(\tilde\pi^{*})=
		\liminf_{\omega\to\infty}\cB(\tilde\pi^{(\omega)})=
		\liminf_{\omega\to\infty}\,\liminf_{n\to\infty}\cB(\tilde\pi_{n,\omega}).
	\end{align}
Thus, the assertion follows from (\ref{eqLemma_thmBethePlus1_1})--(\ref{eqLemma_thmBethePlus1_3}).
\end{proof}

To proceed we need a small twist on \Lem~\ref{Lemma_thmBethePlus1}. 
Namely, instead of using $\hat\GG_{n,\omega}$ as our reference point, we are going to work with $\GG_{n,\omega}$.
Thus, let $\cC$ be the set of cavities of $\GG_{n,\omega}$ and let 
	$\rho_{n,\omega,\cS}\in\meas$ be the kernel representing $\mu_{\GG_{n,\omega},\cC}$.
Further, let $\pi_{n,\omega,\cS}\in\Meas$ be the distribution of $\rho_{n,\omega,\cS}$.
Define $\rho_{n,\omega}$, $\pi_{n,\omega}$ analogously with respect to $\G_{n,\omega}$.

\begin{corollary}\label{Cor_thmBethePlus1}
If {\bf POS} is satisfied, then
		\begin{align*}
\lim_{n\to\infty}\frac1n\Erw\ln Z(\G)&=\liminf_{\omega\to\infty}\,\liminf_{n\to\infty}\cB(\pi_{n,\omega}),&	\lim_{n\to\infty}\frac1n\Erw\ln Z(\GG)&=\liminf_{\omega\to\infty}\,\liminf_{n\to\infty}\cB(\pi_{n,\omega,\cS}).
\end{align*}
\end{corollary}
\begin{proof}
Since $X,Y$ are Poisson variables with a large mean $\omega$, $\hat M_{n,\omega}$ and $M_{n,\omega}$ can be coupled so that both coincide with probability $1-o_\omega(1)$.
This coupling naturally extends to a coupling of $\hat\GG_{n,\omega}$ and $\GG_{n,\omega}$ under which $\GG_{n,\omega}=\hat\GG_{n,\omega}$ with probability $1-o_\omega(1)$.
Consequently, recalling that $\Cutm(\nix,\nix)$ stands for the Wasserstein metric on $\Meas$, we have $\Cutm(\pi_{n,\omega,\cS},\tilde\pi_{n,\omega})=o_\omega(1)$.
Thus, the assertion follows from the \Cor~\ref{Cor_Bcont}.
\end{proof}

We remember the construction of the kernel $\check\mu_{\G,\vX,\vY}\in\meas$ from (\ref{eqmuGXY}).
Let $\check\pi_{n,\omega}\in\Meas$ be the distribution of $\check\mu_{\G,\vX,\vY}$, and define
$\check\mu_{\GG,\vX,\vY}\in\meas$, $\check\pi_{n,\omega,\cS}\in\Meas$ analogously with respect to $\GG$.
Due to the inevitable divisibility condition required to construct a regular factor graph, these kernels are defined whenever $k|dn$.
The following proposition summarizes the main step toward the proof of \Thm~\ref{thmBethePlus}.

\begin{proposition}\label{Prop_thmBethePlus}
For any $\alpha>0$, $\omega>0$ there exists $n_0>0$ such that for all $n>n_0$ such that $k|dn$ we have
	\begin{align*}
\Cutm(\check\pi_{n,\omega},\pi_{n,\omega})&<\alpha,&\Cutm(\check\pi_{n,\omega,\cS},\pi_{n,\omega,\cS})&<\alpha.
\end{align*}
\end{proposition}

To prove  \Prop~\ref{Prop_thmBethePlus}
 we let $\cV=\{v_{i}:i>N_{n,\omega}\}$ and $\cA=\{a_i:i>M_{n,\omega}\}\cup\bigcup_{v\in\cV}\partial v$ be the sets of  variable and constraint nodes, respectively, that are present in $\GG$ but not in $\GG_{n,\omega}$.
Similarly as in \Sec~\ref{Sec_apxInv}, conditioning on the event that $kM_{n,\omega}\leq dN_{n,\omega}-d(k-1)\vX-k\vY$, we define an enhanced random factor
graph $\GG^{\#}$  by 
	\begin{itemize}
	\item adding the variable nodes $\cV$ to $\GG_{n,\omega}$ along with 
		with $d\vX$ new constraint nodes $a_{v,j}^{\#}$, $v\in\cV$, $j\in[d]$.
	Each $a_{v,j}^{\#}$ is adjacent to $v$ and $k-1$ random cavities of $\GG_{n,\omega}$, 
	\item adding $\vY$ more constraint nodes $a_1^{\#},\ldots,a_{\vY}^{\#}$, each connected with $k$ random cavities of $\GG_{n,\omega}$.
	\end{itemize}
Of course, the cavities in the above construction are drawn without replacement and all weight functions are chosen from $P$ independently.
We do {\em not} require that the outcome $\GG^{\#}$ be simple.
Let 
	$$\cA^{\#}=\{a_{v,j}^{\#}:v\in\cV,j\in[d]\}\cup\{a_i^{\#}:i\leq\vY\}$$
comprise the new constraint nodes.

\begin{lemma}\label{Lemma_cplXY}
We have $\dTV(\GG,\GG^{\#})=o(1)$.
\end{lemma}
\begin{proof}
Similarly as in the proof of \Lem~\ref{Lemma_perturb1}, we consider the event
	$\cE=\{\omega/2\leq\vX\leq2\omega,\, \omega/2\leq\vY\leq2\omega\}$, which has
probability $1-o_\omega(1)$.
Further, let $\cE'$ be the event that $\GG$ enjoys the following additional properties.
	\begin{enumerate}[(i')]
	\item We have
			$|\partial^2\cV|=|\cV|d(k-1).$
	\item $|\partial\{a_{ M_{n,\omega}+1},\ldots,a_{m}\}|=k(m-M_{n,\omega})$
			and $\partial\cV\cap\{a_{M_{n,\omega}},\ldots,a_{ m}\}=\emptyset$.
	\item If  $a\not\in\cA$, then $a$ is  connected to the set $\partial\cA$ by at most one edge.
	\end{enumerate}
Additionally, let $\cE''$ be the event that  $\GG_{n,\omega}$ has the following properties.
	\begin{enumerate}[(i'')]
	\item all variable nodes have degree either $d$ or $d-1$.
	\item no two cavities are adjacent to the same variable node.
	\end{enumerate}
We have
	\begin{align}\label{eqLemma_cplXY}
	\pr\brk{\cE}&=1-o(1),&
	\pr\brk{\GG\in\cE'}&=1-o(1),&
	\pr\brk{\GG_{n,\omega}\in\cE''}&=1-o(1).	
	\end{align}
Moreover, $\GG$ given $\cE'$ is distributed precisely as $\GG^{\#}$ given $\cE''$.
Thus, the assertion follows from (\ref{eqLemma_cplXY}).
\end{proof}

Due to \Lem~\ref{Lemma_cplXY} we can apply \Thm~\ref{Thm_BP} to $\GG^{\#}$.
Let $S_1^{\#},\ldots,S_\ell^{\#}$ denote the resulting Bethe state decomposition of $\GG^{\#}$.
Let $T_i^{\#}=S_i^{\#}\cap\Omega^{V_n\setminus\cV}$ for $i\in[\ell]$.
Further, we introduce
	\begin{align}\nonumber
	\vz_{\GG^\#,i}&=
		\scal{\vecone\{\SIGMA\in S_{i}^\#\}\bigg/
			\sum_{\tau\in\Omega^{\cV}}
			\prod_{v\in\cV}p(\tau_{v})\prod_{a\in\cA^\#}\psi_a(\SIGMA_{V_n\setminus\cV},\tau)}{\mu_{\GG^\#}},\\
	\mu_{\GG^\#,i}(\sigma)&=\frac{\mu_{\GG^\#}(\sigma)\vecone\{\sigma\in T_i^\#\}}{\vz_{\GG^\#,i}
		\sum_{\tau\in\Omega^{\cV}}
			\prod_{v\in\cV}p(\tau_{v})\prod_{a\in\cA^\#}\psi_a(\sigma,\tau)}&(\sigma\in\Omega^{V_n\setminus\cV}).
		\label{eqmuGGhashi}
	\end{align}
Thus, $\mu_{\GG^\#,i}\in\cP(\Omega^{V_n\setminus\cV})$.

\begin{lemma}\label{Lemma_comp}
\Whp\ the sets $T_1^\#,\ldots,T_\ell^\#$ are pairswise disjoint and we have
	\begin{align}
	\mu_{\GG^\#,i}(\tau)&=\mu_{\GG_{n,\omega}}(\tau|T_i^\#)\qquad\mbox{for all $\tau\in T_i^\#$ and}&
	\mu_{\GG_{n,\omega}}(T_i^\#)&=\vz_{\GG^\#,i}\big/\sum_{j=1}^\ell\vz_{\GG^\#,j}.
		\label{eqLemma_comp}
	\end{align}
\end{lemma}
\begin{proof}
We recall from \Sec~\ref{Sec_Proof_BP} that the decomposition $S_1^\#,\ldots,S_\ell^\#$ is constructed by pinning the values of a random set $\vU_*$ of variables to specific spins.
Since the size of this set is bounded, with high probability we have $(\cC\cup\cV)\cap\vU_*=\emptyset$.
We will prove that in this case, $\mu_{\GG^\#,i}(\sigma)=\mu_{\GG_{n,\omega}}(\sigma|T_i^\#)$ for all $i,\sigma$.

If $(\cC\cup\cV)\cap\vU_*=\emptyset$, then $T_1^\#,\ldots,T_\ell^\#$ are pairwise disjoint.
Thus, fix $i\in[\ell]$ and $\sigma\in T_i^\#$.
Then by the construction of $\GG^\#$,
	\begin{align}
	\mu_{\GG_{n,\omega}}(\sigma)&=\frac{Z(\GG^\#)}{Z(\GG_{n,\omega})}\cdot\frac{\mu_{\GG^\#}(\sigma)}
		{\sum_{\tau\in\Omega^{\cV}}\prod_{v\in\cV}p(\tau_v)\prod_{a\in\cA^\#}\psi_a(\sigma,\tau)},
										\label{eqLemma_comp_1}\\
	\frac{Z(\GG_{n,\omega})}{Z(\GG^\#)}&=\scal{1\bigg/\sum_{\tau\in\Omega^{\cV}}\prod_{v\in\cV}p(\tau_v)\prod_{a\in\cA^\#}\psi_a(\SIGMA,\tau)}{\mu_{\GG}}=\sum_{j=1}^\ell\vz_{\GG^\#,j}.\label{eqLemma_comp_2}
	\end{align}
Combining (\ref{eqLemma_comp_1}) and (\ref{eqLemma_comp_2}), we obtain the second identity in (\ref{eqLemma_comp}).
Further, combining the second part of (\ref{eqLemma_comp}) with (\ref{eqLemma_comp_1}) and (\ref{eqLemma_comp_2}), we 
find
	\begin{align*}
	\mu_{\GG_{n,\omega}}(\sigma\mid T_i^\#)&=
		\frac{\mu_{\GG_{n,\omega}}(\sigma)}{\mu_{\GG_{n,\omega}}(T_i^\#)}
		=\frac{\sum_{j=1}^\ell\vz_{\GG^\#,j}}{\vz_{\GG^\#,i}}
		\cdot \frac{Z(\GG^\#)}{Z(\GG_{n,\omega})}\cdot \vz_{\GG^\#,i}\cdot\mu_{\GG^{\#},i}(\sigma)=\mu_{\GG^{\#},i}(\sigma),
	\end{align*}
thereby establishing the first part of (\ref{eqLemma_comp}).
\end{proof}

\Whp\ each cavity $v\in\cC$ of $\GG_{n,\omega}$ has degree $d-1$.
In this case, we denote by $b_v$ the unique neighbour of $v$ in $\GG^\#$ that is not present in $\GG_{n,\omega}$.
Further, for $i\in[\ell]$ let $\NU_{\GG,i}\in\cP(\Omega^{\cC})$ be the product measure
	$$\nu_{\GG^\#,i}=\bigotimes_{v\in\cC}\mu_{\GG^\#,v\to b_v}(\nix|S_i^\#).$$
In close analogy to the weights introduced in (\ref{eqzGi}),  we also define
	\begin{align}\nonumber
	\check\vz_{i}^\#=\mu_{\GG^\#}(S_i^\#)&\cdot
			\prod_{v\in\cV}\bc{\sum_{\chi\in\Omega}p(\chi)\prod_{a\in\partial v}\sum_{\tau\in\Omega^{\partial a}}\vecone\{\tau_{v}=\chi\}\psi_a(\tau)\prod_{w\in\partial a\setminus v}\mu_{\GG^\#,w\to a}(\tau_w|S_i^\#)}^{-1}
			\\
&\cdot\prod_{i=1}^{\vY}\bc{\sum_{\tau\in\Omega^{\partial a_i^\#}}\psi_{a_i^\#}(\tau)\prod_{w\in\partial a_i^\#}\mu_{\GG^\#,w\to a_i^\#}(\tau_w|S_i^\#)}^{-1}.\label{eqzGi_1}
	\end{align}

\begin{lemma}\label{Lemma_comp2}
\Whp\ we have $\sum_{h=1}^\ell\abs{\vz_{\GG^\#,h}-\check\vz_{h}^\#}=o(1)$ and 
	\begin{align*}
	\sum_{h=1}^\ell\mu_{\GG_{n,\omega}}(T_h^\#)\TV{\mu_{\GG_{n,\omega},\cC}(\nix|T_h^\#)-\nu_{\GG^\#,h}}&=o(1).
	\end{align*}
\end{lemma}
\begin{proof}
Fix $h\in[\ell]$ and suppose that $S_i^\#$ is an $o(1)$-Bethe state, which occurs with probability $1-o_\omega(1)$ due to \Thm~\ref{Thm_BP} and \Lem~\ref{Lemma_cplXY}.
Then by {\bf BS2} \whp\ we have for any $\sigma\in\Omega^{\cV\cup\cC}$,
	\begin{align}\nonumber
	\mu_{\GG^\#}(\sigma|S_h^\#)&\sim 
		\prod_{v\in\cV}\frac{p(\sigma_{v_i})\prod_{a\in\partial v}\psi_a(\sigma)\prod_{w\in\partial a}\mu_{\GG^\#,w\to a}(\sigma_w|S_h^\#)}
			{\sum_{\chi\in\Omega}\prod_{a\in\partial v_i}\sum_{\tau\in\Omega^{\partial a}}\vecone\{\tau_v=\chi\}\psi_a(\tau)
				\prod_{w\in\partial a}\mu_{\GG^\#,w\to a}(\tau_w|S_h^\#)}\\
		&\quad\cdot\prod_{i=1}^{\vY}\frac{\psi_{a_i^\#}(\sigma)\prod_{w\in\partial a_i^\#}\mu_{\GG^\#,w\to a_i^\#}(\sigma_w|S_h^\#)}
			{\sum_{\tau\in\Omega^{\partial a_i^\#}}\psi_{a_i^\#}(\tau)\prod_{w\in\partial a_i^\#}\mu_{\GG^\#,w\to a_i^\#}(\tau_w|S_h^\#)}.\label{eqLemma_comp2_1}
	\end{align}
Further, \whp\ each cavity of $\GG_{n,\omega}$ has degree $d-1$; in this case, denote by $c_v$ the unique neighbor of $v$ in $\GG^{\#}$ that is absent in $\GG_{n,\omega}$.
Then \whp\ we have
	\begin{align}
	\vz_{\GG^\#,h}&=\mu_{\GG^\#}(S_h^\#)
		\scal{1\bigg/
			\sum_{\tau\in\Omega^{\cV}}
			\prod_{v\in\cV}p(\tau_v)\prod_{a\in\cA^\#}\psi_a(\SIGMA_{V\setminus\cV},\tau)}{\mu_{\GG^\#}(\nix|S_h^\#)}\nonumber\\
		&\sim
				\mu_{\GG^\#}(S_h^\#)\sum_{\sigma\in\Omega^{\cV\cup\cC}}
					\prod_{v\in\cV}\frac{p(\sigma_v)\prod_{a\in\partial v}\psi_a(\sigma)\prod_{w\in\partial a\setminus v}\mu_{\GG^{\#},w\to a}(\sigma_w|S_h^\#)}{p(\sigma_v)\prod_{a\in\partial v}\psi_a(\sigma)
							\cdot\sum_{\chi\in\Omega}p(\chi)\prod_{a\in\partial v}\sum_{\tau\in\Omega^{\partial a}}\vecone\{\chi=\tau_v\}
							\psi_a(\tau)\prod_{w\in\partial a\setminus v}\mu_{\GG^{\#},w\to a}(\tau_w|S_h^\#)}
					\nonumber\\
			&\qquad\qquad\qquad\qquad\cdot\prod_{i=1}^{\vY}\frac{\psi_{a_i^{\#}}(\sigma)\prod_{w\in\partial a_i^{\#}}\mu_{\GG^{\#},w\to a_i^{\#}}(\sigma_w|S_h^\#)}
						{\psi_{a_i^{\#}}(\sigma)\cdot\sum_{\tau\in\Omega^{\partial a_i^{\#}}}\psi_{a_i^{\#}}(\tau)\prod_{w\in\partial a_i^{\#}}\mu_{\GG^{\#},w\to a_i^{\#}}(\tau_w|S_h^\#)}
			\nonumber		\\
			&=\mu_{\GG^\#}(S_h^\#)\prod_{v\in\cV}\frac{\sum_{\sigma\in\Omega^{\partial^2v}}\prod_{a\in\partial v}\prod_{w\in\partial a\setminus v}\mu_{\GG^{\#},w\to a}(\sigma_w|S_h^\#)}{\sum_{\chi\in\Omega}p(\chi)\prod_{a\in\partial v}\sum_{\tau\in\Omega^{\partial a}}\vecone\{\chi=\tau_v\}
							\psi_a(\tau)\prod_{w\in\partial a\setminus v}\mu_{\GG^{\#},w\to a}(\tau_w|S_h^\#)}\nonumber\\
			&\qquad\qquad\qquad\qquad\cdot\prod_{i=1}^{\vY}\frac{\sum_{\sigma\in\Omega^{\partial a_i^{\#}}}\prod_{w\in\partial a_i^{\#}}\mu_{\GG^{\#},w\to a_i^{\#}}(\sigma_w|S_h^\#)}
						{\psi_{a_i^{\#}}(\sigma)\cdot\sum_{\tau\in\Omega^{\partial a_i^{\#}}}\psi_{a_i^{\#}}(\tau)\prod_{w\in\partial a_i^{\#}}\mu_{\GG^{\#},w\to a_i^{\#}}(\tau_w|S_h^\#)}\nonumber\\
		&=\check\vz_{h}^{\#}.
			\label{eqcheck}
	\end{align}
Summing on $h$ completes the proof of the first assertion.

With respect to the second assertion, for $\sigma\in\Omega^{\cC}$ we have \whp
	\begin{align*}
	\mu_{\GG_{n,\omega},\cC}(\sigma|T_h^{\#})&=
	\mu_{\GG^\#,h}(\sigma)&\mbox{[by \Lem~\ref{Lemma_comp}]}\nonumber\\
			&=\mu_{\GG^{\#}}(S_h^{\#})\cdot\frac{\mu_{\GG^\#}(\sigma|S_h^{\#})}{\vz_{\GG^\#,h}
					\sum_{\tau\in\Omega^{\cV}}
						\prod_{v\in\cV}p(\tau_{v})\prod_{a\in\cA^\#}\psi_a(\sigma,\tau)}
								&\mbox{[by (\ref{eqmuGGhashi})]}\nonumber\\
		&\sim\frac{\mu_{\GG^{\#}}(S_h^{\#})}{\vz_{\GG^\#,h}}
				\bc{\sum_{\tau\in\Omega^{\cV}}
						\prod_{v\in\cV}p(\tau_{v})\prod_{a\in\cA^\#}\psi_a(\sigma,\tau)}^{-1}\\
			&\quad\cdot\prod_{i=1}^{\vY}\frac{\psi_{a_i^\#}(\sigma)\prod_{w\in\partial a_i^\#}\mu_{\GG^\#,w\to a_i^\#}(\sigma_w|S_h^\#)}
			{\sum_{\tau\in\Omega^{\partial a_i^\#}}\psi_{a_i^\#}(\tau)\prod_{w\in\partial a_i^\#}\mu_{\GG^\#,w\to a_i^\#}(\tau_w|S_h^\#)}\nonumber\\
			&\quad\cdot
			\prod_{v\in\cV}\frac{p(\sigma_{v_i})\prod_{a\in\partial v}\psi_a(\sigma)\prod_{w\in\partial a}\mu_{\GG^\#,w\to a}(\sigma_w|S_h^\#)}
			{\sum_{\chi\in\Omega}\prod_{a\in\partial v_i}\sum_{\tau\in\Omega^{\partial a}}\vecone\{\tau_v=\chi\}\psi_a(\tau)
				\prod_{w\in\partial a}\mu_{\GG^\#,w\to a}(\tau_w|S_h^\#)}
			&\mbox{[by~\eqref{eqLemma_comp2_1}]}\nonumber\\
		&=\frac{\mu_{\GG^{\#}}(S_h^{\#})}{\vz_{\GG^\#,h}}
				\prod_{i=1}^{\vY}\frac{\prod_{w\in\partial a_i^\#}\mu_{\GG^\#,w\to a_i^\#}(\sigma_w|S_h^\#)}
			{\sum_{\tau\in\Omega^{\partial a_i^\#}}\psi_{a_i^\#}(\tau)\prod_{w\in\partial a_i^\#}\mu_{\GG^\#,w\to a_i^\#}(\tau_w|S_h^\#)}\nonumber\\
					&\qquad\qquad\qquad\cdot
			\prod_{v\in\cV}\frac{\prod_{a\in\partial v}\prod_{w\in\partial a}\mu_{\GG^\#,w\to a}(\sigma_w|S_h^\#)}
			{\sum_{\chi\in\Omega}\prod_{a\in\partial v_i}\sum_{\tau\in\Omega^{\partial a}}\vecone\{\tau_v=\chi\}\psi_a(\tau)
				\prod_{w\in\partial a}\mu_{\GG^\#,w\to a}(\tau_w|S_h^\#)}\\
		&=\nu_{\GG^\#,h}(\sigma),&\mbox{[by~\eqref{eqcheck}]}
	\end{align*}
as claimed.
\end{proof}

\begin{proof}[Proof of \Prop~\ref{Prop_thmBethePlus}]
Let 
	$\nu_{\GG^{\#}}={\sum_{i=1}^\ell\check\vz^{\#}_i\nu_{\GG^{\#},i}}/{\sum_{i=1}^\ell\check\vz_{\GG^{\#},i}}$
and let $\pi^\#_{n,\omega,\cS}$ be the distribution of the kernel representation $\dot\nu_{\GG^{\#}}\in\meas$.
Then up to a renumbering of the variable and constraint nodes, $\check\mu_{\GG^{\#},\vX,\vY}\in\meas$ is distributed as the representation of $\nu_{\GG^{\#}}$.
Specifically, in (\ref{eqmuGXY}) we renumbered the nodes such that $\cV$ comprises the first $\vX$ variable nodes and such that the $a_i^{\#}$, $i\in[\vY]$, are the first $\vY$ constraint nodes.
Due to \Lem~\ref{Lemma_cplXY} and because $\GG$ and $\G$ are invariant under node permutations, 
we conclude that $\Cutm(\pi^\#_{n,\omega,\cS},\check\pi_{n,\omega,\cS})=o(1)$.
Furthermore, combining \Lem s~\ref{Lemma_triangle}, \ref{Lemma_comp} and~\ref{Lemma_comp2}, we see that
$\Erw[\cutm(\mu_{\GG_{n,\omega},\cC},\nu_{\GG^{\#}})]=o(1)$.
Hence, invoking~\eqref{eqCutmcutm}, we conclude that $\Erw[\Cutm(\rho_{n,\omega,\cS},\dot\nu_{\GG^{\#}})]=o(1)$.
Thus, the triangle inequality yields $\Cutm(\pi_{n,\omega,\cS},\check\pi_{n,\omega,\cS})=o(1)$.
The same argument applies to $\pi_{n,\omega}$ and $\check\pi_{n,\omega}$.
\end{proof}

As a final preparation toward the proof of \Thm~\ref{thmBethePlus}, we need the following simple lemma.

\begin{lemma}\label{Lemma_smallShift}
For any fixed integer $\ell$ we have
	\begin{align*}
	\Cutm(\pi_{n,\omega},\pi_{n+\ell,\omega})&=o_\omega(1),&
	\Cutm(\pi_{n,\omega,\cS},\pi_{n+\ell,\omega,\cS})&=o_\omega(1).
\end{align*}	
\end{lemma}
\begin{proof}
The random factor graph $\GG_{n,\omega}$ or $\G_{n,\omega}$, respectively, has $n-\vX$ variable nodes with probability $1-o_\omega(1)$.
Similarly, the number of variable nodes of $\GG_{n+\ell,\omega}$ or $\G_{n+\ell,\omega}$ is $n+\ell-\vX$
with probability $1-o_\omega(1)$.
Since $\vX$ is a Poisson variable with mean $\omega$, we have $\dTV(n+\ell-\vX,n-\vX)=o_\omega(1)$.
Hence, we can couple $\GG_{n+\ell,\omega}$ and $\GG_{n,\omega}$ as well as 
$\G_{n+\ell,\omega}$ and $\G_{n,\omega}$ in such a way that both coincide \whp\
This coupling extends to the distributions $\rho_{n,\omega,\cS},\rho_{n+\ell,\omega,\cS}$ and
$\rho_{n,\omega},\rho_{n+\ell,\omega}$.
\end{proof}

\begin{proof}[Proof of \Thm~\ref{thmBethePlus}]
\Cor~\ref{Cor_thmBethePlus1} yields the free energy formula in terms of the distributions
$\pi_{n,\omega}$ and $\pi_{n,\omega,\cS}$, respectively.
Furthermore, \Prop~\ref{Prop_thmBethePlus} implies together with \Cor~\ref{Cor_Bcont} that
	\begin{align}\label{eqthmBethePlus1}
	\liminf_{\omega\to\infty}\,\liminf_{n\to\infty, k|dn}\cB(\pi_{n,\omega})&=
			\liminf_{\omega\to\infty}\,\liminf_{n\to\infty, k|dn}\cB(\check\pi_{n,\omega}),\\
	\liminf_{\omega\to\infty}\,\liminf_{n\to\infty, k|dn}\cB(\pi_{n,\omega,\cS})&=
		\liminf_{\omega\to\infty}\,\liminf_{n\to\infty, k|dn}\cB(\check\pi_{n,\omega,\cS}),\label{eqthmBethePlus2}
	\end{align}
with the limit on $n$ confined to integers such that $k|dn$ each time.
But \Lem~\ref{Lemma_smallShift} implies with \Cor~\ref{Cor_Bcont}  that this divisibility condition does not alter the limits on the left hand side of these equations, i.e.,
	\begin{align}
	\liminf_{\omega\to\infty}\,\liminf_{n\to\infty, k|dn}\cB(\pi_{n,\omega})&=
			\liminf_{\omega\to\infty}\,\liminf_{n\to\infty}\cB(\pi_{n,\omega}),\label{eqthmBethePlus3}\\
	\liminf_{\omega\to\infty}\,\liminf_{n\to\infty,k|dn}\cB(\pi_{n,\omega,\cS})&=
		\liminf_{\omega\to\infty}\,\liminf_{n\to\infty}\cB(\pi_{n,\omega,\cS}).\label{eqthmBethePlus4}
	\end{align}
Thus, combining (\ref{eqthmBethePlus1})--(\ref{eqthmBethePlus4}) and invoking \Cor~\ref{Cor_thmBethePlus1}, we obtain
	\begin{align*}
	\lim_{n\to\infty}\frac1n\Erw\ln Z(\G)&=\liminf_{\omega\to\infty}\,\liminf_{n\to\infty}\cB(\check\pi_{n,\omega}),&	\lim_{n\to\infty}\frac1n\Erw\ln Z(\GG)&=\liminf_{\omega\to\infty}\,\liminf_{n\to\infty}\cB(\check\pi_{n,\omega,\cS}),
	\end{align*}
where, of course, the limit is confined to $n$ such that $k|dn$ because $\GG$, $\G$ and $\check\pi_{n,\omega}$, $\check\pi_{n,\omega,\cS}$ are defined only for such $n$; this is the assertion.
\end{proof}

\section{Applications}\label{Sec_app}

\noindent
In \Sec~\ref{Sec_app_sg} we prove that the spin glass model from \Sec~\ref{Sec_diluted_sg} satisfies the condition {\bf POS};
the results stated in \Sec~\ref{Sec_diluted_sg} are then immediate from those in \Sec~\ref{Sec_results}.
Further, in \Sec s~\ref{Sec_app_Potts} and~\ref{Sec_app_kSAT} we apply the results from \Sec~\ref{Sec_results} to two further models, the Potts antiferromagnet and the random regular $k$-SAT model.
Finally, in \Sec~\ref{secHCintro} we show how the theorems from \Sec~\ref{Sec_results} can be brought to bear on the hard-core model, thereby proving the results stated in \Sec~\ref{Sec_introhc}.

\subsection{The spin glass}\label{Sec_app_sg}
To derive the results on the spin glass model stated in \Sec~\ref{Sec_intro} from the general theorems in \Sec~\ref{Sec_results}, we just need to verify the condition {\bf POS} for the spin glass model.
In Example~\ref{Ex_sg} we introduced the relevant weight function even in the more general case of the $k$-spin model; the case $k=2$ corresponds to the spin glass on the Bethe lattice.

\begin{lemma}\label{Lemma_kspin}
The $k$-spin model satisfies {\bf POS} for all $d\geq3$, $\beta>0$ and all even $k\geq2$.
\end{lemma}
\begin{proof}
The lemma is already implicit in~\cite{FranzLeone,PanchenkoTalagrand}; but let us carry the simple proof out for completeness.
Let $\vec J$ be a standard Gaussian.
Upon substituting the weight functions from Example~\ref{Ex_sg} into {\bf POS} and multiplying by $2^\ell$, {\bf POS} reads
	\begin{align}\nonumber
	\Erw&\brk{\bc{1-\tanh(\beta\vJ)\int_0^1\prod_{i=1}^k\bc{2\mu_{s,\vx_i}-1}\dd s}^\ell}
		+(k-1)\Erw\brk{\bc{1-\tanh(\beta\vJ)\int_0^1\prod_{i=1}^k\bc{2\mu_{s,\vx_i}-1}\dd s}^\ell}\\
		&-k\Erw\brk{\bc{1-\tanh(\beta\vJ)\int_0^1(2\mu_{s,\vx_1}-1)\prod_{i=2}^k(2\mu'_{s,\vx_i}-1)\dd s}^\ell}\geq0.
		\label{eqLemma_sg1}
	\end{align}
for all measurable $\mu,\mu':[0,1]^2\to[0,1]$.
Expanding the first expectation yields
	\begin{align*}
	\Erw&\brk{\bc{1-\tanh(\beta\vJ)\int_0^1\prod_{i=1}^k\bc{2\mu_{s,\vx_i}-1}\dd s}^\ell}
		=\sum_{j=0}^\ell\bink\ell j(-1)^j\Erw\brk{\tanh(\beta\vJ)^j\bc{\int_0^1\prod_{i=1}^k\bc{2\mu_{s,\vx_i}-1}\dd s}^j}
	\end{align*}
Since $\vJ$ is independent of the $\vx_i$, the last expectation vanishes if $j$ is odd, while $\tanh(\beta\vJ)^j\geq0$ if $j$ is even.
Thus, in order to establish (\ref{eqLemma_sg1}) it suffices to show that for any even $j\geq2$,
	\begin{align}
	\Erw&\brk{\bc{\int_0^1\prod_{i=1}^k\bc{2\mu_{s,\vx_i}-1}\dd s}^j
		+(k-1)\bc{\int_0^1\prod_{i=1}^k\bc{2\mu_{s,\vx_i}-1}\dd s}^j
		-k\bc{\int_0^1(2\mu_{s,\vx_1}-1)\prod_{i=2}^k(2\mu'_{s,\vx_i}-1)\dd s}^j}\geq0.
		\label{eqLemma_sg2}
	\end{align}
Let $\vs_1,\ldots,\vs_j\in[0,1]$ be uniformly distribution and mutually independent as well as independent of the $\vx_i$.
Then Fubini's theorem yields
	\begin{align}\label{eqLemma_sg3}
	\Erw\brk{\bc{\int_0^1\prod_{i=1}^k\bc{2\mu_{s,\vx_i}-1}\dd s}^j},
	&=\Erw\brk{\Erw\brk{\prod_{h=1}^j(2\mu_{\vs_h,\vx_1}-1)\bigg|\vs_1,\ldots,\vs_j}^k},\\
				\label{eqLemma_sg4}
	\Erw\brk{\bc{\int_0^1\prod_{i=1}^k\bc{2\mu'_{s,\vx_i}-1}\dd s}^j}
		&=\Erw\brk{\Erw\brk{\prod_{h=1}^j(2\mu'_{\vs_h,\vx_1}-1)\bigg|\vs_1,\ldots,\vs_j}^k},\\
	\Erw\brk{\bc{\int_0^1\bc{2\mu_{s,\vx_1}-1)}\prod_{i=2}^k(2\mu'_{s,\vx_2}-1)\dd s}^j}
		&=		\Erw\brk{\prod_{h=1}^j(2\mu_{\vs_h,\vx_1}-1)\bigg|\vs_1,\ldots,\vs_\ell}
			\Erw\brk{\prod_{h=1}^j(2\mu'_{\vs_h,\vx_1}-1)\bigg|\vs_1,\ldots,\vs_\ell}^{k-1}.
				\label{eqLemma_sg5}
	\end{align}
Since for even $k$ we have $X^k+(k-1)Y^k-kXY^{k-1}\geq0$ for all $X,Y\in\RR$,  (\ref{eqLemma_sg3})--(\ref{eqLemma_sg5}) yield (\ref{eqLemma_sg2}).
\end{proof}

\noindent
Due to \Lem~\ref{Lemma_kspin}, \Thm~\ref{Thm_sg} follows from \Thm~\ref{Thm_BP}, 
\Thm~\ref{Thm_sgZ} follows from \Thm~\ref{thmBethePlus} and
\Thm~\ref{Thm_sgvariational} follows from \Thm~\ref{Thm_freeEng}.

\begin{remark}
Indeed, together with \Lem~\ref{Lemma_kspin} the results from \Sec~\ref{Sec_results} yield the Bethe state decomposition and the corresponding formulas for the free energy for the $k$-spin model for any even $k\geq2$. 
\end{remark}

\subsection{The Potts model}\label{Sec_app_Potts}
For an integer $q\geq2$ let $\Omega=\{1,\ldots,q\}$ be a set of $q$ distinct colors
Also let $\beta>0$ be a real parameter, the {\em inverse temperature}.
The {\em Potts antiferromagnet} on $\GG$ is the distribution on $\Omega^{V_n}$ defined by
	\begin{align*}
	\mu_{\GG,\beta}(\sigma)&=\frac1{Z_\beta(\GG)}\exp\brk{-\beta\sum_{1\leq i<j\leq n}\vecone\{v_i\in\partial v_j,\,
		\sigma(v_i)=\sigma(v_j)\}},&(\sigma\in\Omega^{V_n}),
	\end{align*}
where the partition function $Z_\beta(\GG)$ provides normalization; we omit the reference to $\beta$ where possible.
Thus, for a given $\sigma\in\Omega^{V_n}$ each monochromatic edge of $\GG$ incurs an $\exp(-\beta)$ penalty factor.

The Potts antiferromagnet and the associated optimization problems, the {\sc Max $q$-Cut} problem, are of fundamental importance in combinatorics.
Krzakala and Zdeborov\'a~\cite{FLPotts} brought the cavity method to bear on this model.
In the following we show how the main results of the present paper apply to this model to underpin the predictions from~\cite{FLPotts} rigorously.
In particular, we specialize the Belief Propagation equations to the Potts model, work out the variational formula for the free energy and apply this formula to the {\sc Max $q$-Cut} problem on the random regular graph.

The Potts model on $\GG(n,d)$ can be cast as a random factor graph model with a single weight function
	\begin{align*}
	\psi_\beta&:\Omega^2\to(0,1),&(\sigma,\tau)\mapsto\exp(-\beta\vecone\{\sigma=\tau\}).
	\end{align*}
Thus, $k=2$, $\Psi=\{\psi_\beta\}$ and $P(\psi_\beta)=1$ and the prior distribution $p$ is uniform on $\Omega$.
Since the constraints are binary, the random regular factor graph $\G$ can be identified with the usual random $d$-regular graph $\GG$, with the edges representing the factor nodes.

\begin{lemma}\label{Lemma_PottsPOS}
The Potts model satisfies condition {\bf POS} for all $\beta>0$.
\end{lemma}
\begin{proof}
We plug the definition of $\psi_\beta$ into {\bf POS} and notice that the $1-\eul^{-\beta}$ factors cancel.
Hence, the desired inequality reads
	\begin{align}\label{eqLemma_PottsPOS1}
	\Erw\brk{\bc{
		\sum_{\sigma=1}^q\int_0^1\mu_{s,\vx_1}(\sigma)\mu_{s,\vx_2}(\sigma)}^\ell
		+\bc{\sum_{\sigma=1}^q\int_0^1\mu'_{s,\vx_1}(\sigma)\mu'_{s,\vx_2}(\sigma)}^\ell-
			2\bc{\sum_{\sigma=1}^q\int_0^1\mu_{s,\vx_1}(\sigma)\mu'_{s,\vx_2}(\sigma)}^\ell}&\geq0
			&(\mu,\mu'\in\meas).
	\end{align}
Applying Fubini's theorem to take the expectation on $\vx_1,\vx_2$ inside, we find
	\begin{align}\nonumber
	\Erw\brk{\bc{\sum_{\sigma=1}^q\int_0^1\mu_{s,\vx_1}(\sigma)\mu_{s,\vx_2}(\sigma)}^\ell}&=
		\sum_{\sigma_1,\ldots,\sigma_\ell=1}^q\Erw\brk{\prod_{h=1}^\ell\mu_{\vs_h,\vx_1}(\sigma_h)\mu_{\vs_h,\vx_2}(\sigma_h)}\\
		&=\sum_{\sigma_1,\ldots,\sigma_\ell=1}^q
			\Erw\brk{\Erw\brk{\prod_{h=1}^\ell\mu_{\vs_h,\vx_1}(\sigma_h)\bigg|\vs_1,\ldots,\vs_\ell}
			\Erw\brk{\prod_{h=1}^\ell\mu_{\vs_h,\vx_2}(\sigma_h)\bigg|\vs_1,\ldots,\vs_\ell}}\nonumber\\
		&=\sum_{\sigma_1,\ldots,\sigma_\ell=1}^q
			\Erw\brk{\Erw\brk{\prod_{h=1}^\ell\mu_{\vs_h,\vx_1}(\sigma_h)\bigg|\vs_1,\ldots,\vs_\ell}^2}.
			\label{eqLemma_PottsPOS2}
	\end{align}
Similar manipulations yield
	\begin{align}			\label{eqLemma_PottsPOS3}
	\Erw\brk{\bc{\sum_{\sigma=1}^q\int_0^1\mu'_{s,\vx_1}(\sigma)\mu'_{s,\vx_2}(\sigma)}^\ell}
		&=\sum_{\sigma_1,\ldots,\sigma_\ell=1}^q
			\Erw\brk{\Erw\brk{\prod_{h=1}^\ell\mu'_{\vs_h,\vx_1}(\sigma_h)\bigg|\vs_1,\ldots,\vs_\ell}^2},\\
	\Erw\brk{\bc{\sum_{\sigma=1}^q\int_0^1\mu_{s,\vx_1}(1)\mu'_{s,\vx_2}(1)}^\ell}
		&=\sum_{\sigma_1,\ldots,\sigma_\ell=1}^q\Erw\brk{
			\Erw\brk{\prod_{h=1}^\ell\mu_{\vs_h,\vx_1}(\sigma_h)\bigg|\vs_1,\ldots,\vs_\ell}
			\Erw\brk{\prod_{h=1}^\ell\mu'_{\vs_h,\vx_1}(\sigma_h)\bigg|\vs_1,\ldots,\vs_\ell}}.
				\label{eqLemma_PottsPOS4}
	\end{align}
Combining (\ref{eqLemma_PottsPOS2})--(\ref{eqLemma_PottsPOS4}), we conclude that the l.h.s.\ of (\ref{eqLemma_PottsPOS1}) is the expectation of a sum of squares, and thus non-negative.
\end{proof}

The message space $\cS(\GG)$ of the Potts model boils down to the set of all families $(\mu_{v\to w})_{v\in V_n,w\in\partial w}$, with $\mu_{v\to w}\in\cP(\Omega)$.
With this simplification the Belief Propagation operator $\BP:\cS(\GG)\to\cS(\GG)$, $\nu\mapsto\hat\nu$ of the Potts model reads
	\begin{align}\label{eqBPPotts}
	\hat\nu_{v\to u}(\sigma)&=\frac{\prod_{w\in\partial v\setminus u}1-(1-\eul^{-\beta})\mu_{w\to v}(\sigma)}
		{\sum_{\tau\in\Omega}\prod_{w\in\partial v\setminus u}1-(1-\eul^{-\beta})\mu_{w\to v}(\tau)}
		&(\sigma\in\Omega).
	\end{align}

With respect to Bethe states, we expect that the phase space $\Omega^n$ decomposes into $S_1,\ldots,S_\ell$ such that the conditional distribution $\mu_{\GG}[\nix|S_i]$ are free of long-range correlations, that their standard messages form an approximate fixed point of BP and that the conditional marginals derive from the messages.
In formulas, with high probability over the choice of the graph 
and with $(\hat\mu_{\GG,v\to u}[\nix|S_h])_{u\in\partial v}=\BP(\mu_{\GG,v\to u}[\nix|S_h])_{u\in\partial v}$,
we aim to show that
	\begin{align}\label{eqPureState1}
	\frac1{n^2}\sum_{1\leq i<j\leq n} \TV{\mu_{\GG,v_i,v_j}[\nix|S_h]-\mu_{\GG,v_i}[\nix|S_h]\tensor\mu_{\GG,v_j}[\nix|S_h]}&=o(1),\\
				\label{eqPureState2}
	\frac1n\sum_{i=1}^n\sum_{u\in\partial v_i}\TV{\mu_{\GG,v_i\to u}[\nix|S_h]
		-\hat\mu_{\GG,v_i\to u}[\nix|S_h]}&=o(1),\\
		\label{eqPureState2}
	\frac1n\sum_{i=1}^n\sum_{\sigma\in\Omega}\abs{\mu_{\GG,v_i}[\sigma|S_h]-
		\frac{\prod_{w\in\partial v_i}1-(1-\eul^{-\beta})\mu_{\GG,v_i\to w}(\sigma)}
			{\sum_{\tau\in\Omega}\prod_{w\in\partial v_i}1-(1-\eul^{-\beta})\mu_{\GG,v_i\to w}(\tau)}}&=o(1).
	\end{align}
The following theorem establishes these facts.

\begin{theorem}\label{Thm_Potts}
For any sequence $L=L(n)\to\infty$ and all $d\geq3$, $\beta>0$ the following is true.
With high probability $\GG$ admits a decomposition $S_0,S_1,\ldots,S_\ell$, $\ell\leq L$, of the phase space $\Omega^n$ such that
$\mu_{\GG}(S_0)=o(1)$ and such that (\ref{eqPureState1})--(\ref{eqPureState2}) are satisfied for $h=1,\ldots,\ell$.
\end{theorem}	
\begin{proof}
This is immediate from \Thm~\ref{Thm_BP} applied to the factor graph representation of the Potts model.
\end{proof}

With respect to the free energy,  let $\vX,\vY$ be two independent Poisson variables with mean $\omega$.
Let $\vu_1,\ldots,\vu_{\vX}$ and $\vv_1\vw_1,\ldots,\vv_{\vY}\vw_{\vY}$ be uniformly random vertices and edges of $\GG$, chosen independently.
With $S_1,\ldots,S_\ell$ the decomposition from \Thm~\ref{Thm_Potts}, we introduce the weights 
	\begin{align*}
	\vz_{\GG,h}&=\mu_{\GG}(S_h)\prod_{i=1}^{\vX}\bc{
		\sum_{\sigma\in\Omega}\prod_{v\in\partial\vu_i}1-(1-\eul^{-\beta})\mu_{\GG,v\to\vu_i}(\sigma|S_h)}^{-1}
		\prod_{i=1}^{\vY}\bc{1-(1-\eul^{-\beta})\sum_{\sigma\in\Omega}
			\mu_{\GG,\vv_i\to\vw_i}(\sigma|S_h)\mu_{\GG,\vw_i\to\vv_i}(\sigma|S_h)}^{-1}
	\end{align*}
and $\vz_{\GG}=\sum_{h=1}^\ell\vz_{\GG,h}$.
Further, let $\cC(\GG)$ be the set of all vertices of degree less than $d$ in the graph obtained from $\GG$ by removing $\vv_1,\ldots,\vv_{\vX}$ and $\vv_1\vw_1,\ldots,\vv_{\vY}\vw_{\vY}$.
Then with high probability each $c\in\cC(\GG)$ has degree precisely $d-1$, and we write $c'$ for the missing $d$'th neighbor of $c$.
Then with $\vc_1,\vc_2,\ldots$ a sequence of uniformly and independently chosen elements of $\cC(\GG)$, we let
	\begin{align*}
	\cB(\GG)&=\Erw\brk{
		\ln{\sum_{h=1}^\ell\frac{\vz_{\GG,h}}{\vz_{\GG}}
			\sum_{\sigma\in\Omega}\prod_{i=1}^d1-(1-\eul^{-\beta})\mu_{\GG,\vc_i\to\vc_i'}(\sigma)}
			+\frac d2\ln{
			\sum_{h=1}^\ell\frac{\vz_{\GG,h}}{\vz_{\GG}}1-(1-\eul^{-\beta})
			\sum_{\sigma\in\Omega}\mu_{\GG,\vc_1\to\vc_1'}(\sigma)\mu_{\GG,\vc_2\to\vc_2'}(\sigma)
			}\Bigg|\GG}.
	\end{align*}

\begin{theorem}\label{Thm_PottsZ}
For all $d\geq3,\beta>0$ we have
	$\lim_{n\to\infty}\frac1n\Erw[\ln Z(\GG)]=\liminf_{\omega\to\infty}\,\liminf_{n\to\infty}\,\Erw[\cB(\GG)].$
\end{theorem}
\begin{proof}
This is an immediate consequence of \Thm~\ref{thmBethePlus} and \Lem~\ref{Lemma_PottsPOS}.
\end{proof}

Additionally, \Thm~\ref{Thm_freeEng} yields a variational formula for the free energy.
Writing out the specifics of the Potts case, we see that $\MEAS$ consists of all $\pi\in\Meas$ that satisfy the following property.
For a measurable $\mu:[0,1]^2\to\cP(\Omega)$ with $\Omega=[q]$ and integers $N,M\geq0$ let 
	\begin{align*}
	\vz_{\mu,s}^{N,M}&=\prod_{i=1}^N\bc{\sum_{\sigma=1}^q\prod_{j=1}^d1-(1-\eul^{-\beta})\mu_{s,\vx_{i,j}}(\sigma)}
		\prod_{i=1}^M\bc{1-(1-\eul^{-\beta})\sum_{\sigma=1}^q\mu_{s,\vx_{i+N,1}}(\sigma)\mu_{s,\vx_{i+N,2}}(\sigma)},
		&\mbox{and set}\\
	\vt&=\vt(s)=
	\inf\cbc{u\in[0,1]:\int_0^u\vz_{\mu,u}^{N,M}\dd s\geq s\int_0^1\vz_{\mu,u}^{N,M}\dd u}.
	\end{align*}
Then we let $\mu^{*(N,M)}_{s,x}=\mu_{\vt,x}$.
Now $\MEAS_\beta$ is the set of all $\pi\in\Meas$ such that for a random $\MU^{\pi}\in\meas$ drawn from $\pi$, the perturbed
$\MU^{\pi*(N,M)}\in\meas$ again has distribution $\pi$.
Furthermore, in the Potts model the functional $\cB(\nix)$ reads
	\begin{align*}
	\cB_\beta(\pi)&=\Erw\brk{\ln\bc{\sum_{\sigma=1}^q\int_0^1
		\prod_{j=1}^d1-(1-\eul^{-\beta})\mu_{s,\vx_{1,j}}(\sigma)
		\dd s}
		+\frac d2\log\bc{1-(1-\eul^{-\beta})\sum_{\sigma=1}^q\int_0^1\mu_{s,\vx_{1,1}}(\sigma)\mu_{s,\vx_{1,2}}(\sigma)\dd s}}.
	\end{align*}

\begin{theorem}\label{Thm_PottsZ'}
For all $d\geq3$, $\beta>0$ we have
	\begin{align*}
\lim_{n\to\infty}\frac1n\Erw[\ln Z(\GG)]&=\Phi_{d,\beta}
	&\mbox{with }\qquad\Phi_{d,\beta}&=\inf_{\pi\in\MEAS_\beta}\cB_\beta(\pi).
\end{align*}
\end{theorem}

As a further application we obtain a variational formula for the {\sc Max $q$-Cut} of the random regular graph, which is defined as
	\begin{align}\label{eqMaxqCut}
	\MC_q(\GG)&=\frac{dn}2-\frac12\min_{\sigma:[n]\to[q]}\sum_{v,w=1}^n\vecone\{w\in\partial v,\,\sigma(v)=\sigma(w)\}.
	\end{align}
Thus, $\MC_q(\GG)$ equals the total number of edges of $\GG$ minus the ground state energy of the Potts model.
In other words, $\MC_q(\GG)$ is the maximum, over the choice of $\sigma:[n]\to[q]$, of the number of edges that link vertices of different colors
The {\sc Max $q$-Cut} problem is well-studied in combinatorics and computer science.
In particular, the problem is well known to be NP-hard on worst-case instances.

\begin{corollary}\label{Thm_MaxqCut}
For all $d\geq3$ we have
	$\displaystyle\MC_q(\GG)/n\ \stacksign{$n\to\infty$}\longrightarrow\ 
		\frac d2+\lim_{\beta\to\infty}\Phi_{d,\beta+1}-\Phi_{d,\beta}$
	in probability.
\end{corollary}
\begin{proof}
Since Azuma's inequality shows that $\MC_q(\GG)$ is concentrated within $O(\sqrt{n\ln n}$) about its mean, it suffices to prove that
	\begin{equation}\label{eqThm_MaxqCut1}
	\lim_{n\to\infty}\frac1n\Erw[\MC_q(\GG)]=\frac d2+\lim_{\beta\to\infty}\Phi_{d,\beta+1}-\Phi_{d,\beta}.
	\end{equation}
Further, introducing $\cH_{\GG}(\sigma)=\frac12\sum_{v,w=1}^n\vecone\{w\in\partial v,\,\sigma(v)=\sigma(w)\}$ and recalling (\ref{eqMaxqCut}), we can rewrite (\ref{eqThm_MaxqCut1}) as 
	\begin{equation}\label{eqThm_MaxqCut2}
	\lim_{n\to\infty}\frac1n\Erw\brk{\min_{\sigma:[n]\to[q]}\cH_{\GG}(\sigma)}=\lim_{\beta\to\infty}\Phi_{d,\beta}-\Phi_{d,\beta+1}.
	\end{equation}

To prove (\ref{eqThm_MaxqCut2}) we write $\mu_{G,\beta}\in\cP([q]^{V(G)})$ for the Potts distribution induced by a $d$-regular graph $G=(V(G),E(G))$.
Moreover, let us denote the Potts Hamiltonian by $\cH_G$ and the partition function by $Z_\beta(G)$.
It is well known that for any $\eps>0$ there exists $\beta_0(\eps)>0$ such that for all $\beta>\beta_0(\eps)$ and all $d$-regular graphs $G$ we have
	\begin{align}\label{eqThm_MaxqCut3}
	\scal{\cH_G}{\mu_{G,\beta}}-\eps |V(G)|&\leq\min_{\sigma:V(G)\to[q]}\cH_G(\sigma)\leq\scal{\cH_G}{\mu_{G,\beta}}.
	\end{align}
Consequently, for all $\beta>\beta_0(\eps)$ we have
	\begin{align}\label{eqThm_MaxqCut4}
	\int_\beta^{\beta+1}\scal{\cH_G}{\mu_{G,b}}\dd b
		-\eps |V(G)|&\leq\min_{\sigma:V(G)\to[q]}\cH_G(\sigma)\leq\int_\beta^{\beta+1}\scal{\cH_G}{\mu_{G,b}}\dd b.
	\end{align}
Since $\scal{\cH_G}{\mu_{G,\beta}}=-\frac{\partial}{\partial\beta}\ln Z_\beta(G)$, (\ref{eqThm_MaxqCut4}) yields
	\begin{align}\label{eqThm_MaxqCut5}
	\ln Z_\beta(G)-\ln Z_{\beta+1}(G)-\eps |V(G)|\leq\min_{\sigma:V(G)\to[q]}\cH_G(\sigma)\leq\ln Z_\beta(G)-\ln Z_{\beta+1}(G).
	\end{align}
Applying (\ref{eqThm_MaxqCut5}) to the random regular graph $\GG$ and taking expectations, we obtain
	\begin{align}\label{eqThm_MaxqCut6}
	\frac1n\Erw[\ln Z_\beta(\GG)]-\frac1n\Erw[\ln Z_{\beta+1}(\GG)]-\eps\leq\frac1n\Erw\brk{\min_{\sigma:[n]\to[q]}\cH_{\GG}(\sigma)}\leq\frac1n\Erw[\ln Z_\beta(\GG)]-\frac1n\Erw[\ln Z_{\beta+1}(\GG)].
	\end{align}
Hence, taking $n\to\infty$, we obtain for all $\beta>\beta_0(\eps)$, 
	\begin{align}\label{eqThm_MaxqCut7}
	\Phi_{d,\beta}-\Phi_{d,\beta+1}-\eps\leq
		\liminf_{n\to\infty}\frac1n\Erw\brk{\min_{\sigma:[n]\to[q]}\cH_{\GG}(\sigma)}\leq
			\limsup_{n\to\infty}\frac1n\Erw\brk{\min_{\sigma:[n]\to[q]}\cH_{\GG}(\sigma)}\leq
		\Phi_{d,\beta}-\Phi_{d,\beta+1}.
	\end{align}
Finally, there exists a subsequence $(n_l)$ along which $\Erw\brk{\min_{\sigma:[n_l]\to[q]}\cH_{\GG(n_l,d)}(\sigma)}/n_l$ converges to a number $\xi\geq0$.
Taking the limit of (\ref{eqThm_MaxqCut6}) along this subsequence, we obtain $\xi\leq\Phi_{d,\beta}-\Phi_{d,\beta+1}\leq\xi+\eps$  for all $\beta>\beta_0(\eps)$.
Consequently, the limit $\lim_{\beta\to\infty}\Phi_{d,\beta}-\Phi_{d,\beta+1}$ exists.
Therefore, taking $\beta\to\infty$ in (\ref{eqThm_MaxqCut7}), we conclude that
	$$\lim_{n\to\infty}n^{-1}\Erw\brk{\min_{\sigma:[n]\to[q]}\cH_{\GG}(\sigma)}$$ exists as well and that (\ref{eqThm_MaxqCut2}) is satisfied.
\end{proof}

\subsection{The regular $k$-SAT model}\label{Sec_app_kSAT}
The $k$-SAT problem play a major role in computer science, particularly in computational complexity theory.
In its optimization version, known as the {\sc Max $k$-SAT} problem asks for the largest number of clauses of a propositional formula in conjunctive normal form with clauses of length $k$ that can be satisfied simultaneously.
Random instances of $k$-SAT and {\sc Max $k$-SAT} have been studied extensively as instructive benchmarks~\cite{ANP}.

We can express the {\sc Max $k$-SAT} problem as a factor graph model with spins $\Omega=\{-1,1\}$ corresponding to the Boolean values `true' and `false' as follows.
With $k\geq2$ an integer and $\beta>0$ be a real parameter, we introduce the weight functions
	\begin{align*}
	\psi_{\beta,\chi}:&\cbc{\pm1}^k\to(0,1),&\sigma&\mapsto\frac{{1-\tanh\bc{\beta\prod_{i=1}^k\chi_i\sigma_i}}}{2}
	&(\chi\in\cbc{\pm1}^k).
	\end{align*}
Let $p$ be the uniform distribution on $\Omega$ and let $P$ be uniform on $\Psi_\beta=\{\psi_{\beta,\chi}:\chi\in\Omega^k\}$.
In terms of propositional formulas, the semantics is that $\psi_{\beta,\chi}$ encodes a $k$-clause whose $i$th literal is negated if $\chi_i=1$ and positive if $\chi_i=-1$.
Thus, $\prod_{i=1}^k\chi_i\sigma_i=1$ if the truth assignment $\sigma$ fails to satisfy the clause, and $\prod_{i=1}^ks_i\sigma_i=-1$ otherwise.
In effect, $\psi_{\beta,s}(\sigma)=(1-\tanh \beta)/2\to0$ as $\beta\to\infty$ if $\sigma$ fails to satisfy the clause,
whereas $\psi_{\beta,s}(\sigma)=(1+\tanh \beta)/2\to1$ if $\sigma$ is satisfying.
Hence, the random factor graph $\G$ models a random $k$-SAT formula in which every variable appears precisely $d$ times, the {\em regular $k$-SAT model}.
We are going to derive variational formulas for its free energy and its ground state energy.

\begin{lemma}\label{Lemma_POSkSAT}
The regular $k$-SAT model satisfies {\bf POS} for all $d,k\geq3$ and all $\beta>0$.
\end{lemma}
\begin{proof}
Once more this is already implicit in~\cite{FranzLeone,PanchenkoTalagrand}, but we carry out the argument here for completeness.
Let us write $\vec\chi$ for a uniformly random element of $\{\pm1\}^k$.
Substituting $\psi_{\beta,\chi}$ into {\bf POS} and cancelling positive constants, we are left to verify the inequality
	\begin{align}\label{eqLemma_POSkSAT1}
	\Erw\brk{\bc{\int_0^1\prod_{i=1}^k\mu_{s,\vx_i}(\CHI_i)\dd s}^\ell
		+\bc{\int_0^1\prod_{i=1}^k\mu_{s,\vx_i}'(\CHI_i)\dd s}^\ell-
			(k-1)\bc{\int_0^1\mu_{s,\vx_1}(\CHI_1)\prod_{i=2}^k\mu'_{s,\vx_2}(\CHI_i)\dd s}^\ell}&\geq0
			&(\mu,\mu'\in\meas).
	\end{align}
Fubini's theorem yields
	\begin{align}\label{eqLemma_POSSAT2}
	\Erw\brk{\bc{\int_0^1\prod_{i=1}^k\mu_{s,\vx_i}(\CHI_i)\dd s}^\ell}
		&=\Erw\brk{\Erw\brk{\prod_{h=1}^\ell\mu_{\vs_h,\vx_1}(\vec\chi_1)\bigg|\vs_1,\ldots,\vs_\ell}^k},\\
				\label{eqLemma_POSSAT3}
	\Erw\brk{\bc{\int_0^1\prod_{i=1}^k\mu_{s,\vx_i}'(\CHI_i)\dd s}^\ell}
		&=\Erw\brk{\Erw\brk{\prod_{h=1}^\ell\mu'_{\vs_h,\vx_1}(\CHI_1)\bigg|\vs_1,\ldots,\vs_\ell}^k},\\
	\Erw\brk{\bc{\int_0^1\mu_{s,\vx_1}(\CHI_2)\prod_{i=2}^k\mu'_{s,\vx_2}(\CHI_i)\dd s}^\ell}
		&=\Erw\brk{
			\Erw\brk{\prod_{h=1}^\ell\mu_{\vs_h,\vx_1}(\CHI_1)\bigg|\vs_1,\ldots,\vs_\ell}
			\Erw\brk{\prod_{h=1}^\ell\mu'_{\vs_h,\vx_1}(\CHI_2)\bigg|\vs_1,\ldots,\vs_\ell}^{k-1}}.
				\label{eqLemma_POSSAT4}
	\end{align}
Since $X^k+(k-1)Y^k-kXY^{k-1}\geq0$ for $X,Y\geq0$,  (\ref{eqLemma_POSSAT2})--(\ref{eqLemma_POSSAT4}) yield (\ref{eqLemma_POSkSAT1}).
\end{proof}

Due to \Lem~\ref{Lemma_POSkSAT} we can bring the results from \Sec~\ref{Sec_results} to bear on the random regular $k$-SAT model.
Specifically,  for a measurable $\mu:[0,1]^2\to\cP(\Omega)$ with $\Omega=\{\pm1\}$ and integers $N,M\geq0$ let 
$(\CHI_{i,j})_{i,j\geq1}$ be independent uniformly random elements of $\Omega$ and let
	\begin{align*}
	\vz_{\mu,s}^{N,M}&=\prod_{i=1}^N\bc{\sum_{\sigma\in\Omega}\prod_{j=1}^d1-\tanh(\beta)\sum_{\tau\in\Omega^{k-1}}\CHI_{i,k}\sigma\prod_{j=1}^{k-1}\CHI_{i,j}\tau_j\mu_{s,\vx_{i,j}}(\tau_j)}
		\prod_{i=1}^M\bc{
			1-\tanh\bc{\beta}\sum_{\tau\in\Omega^k}
				\prod_{j=1}^k\CHI_{i+N,j}\tau_j
				\mu_{s,\vx_{i+N,j}}(\tau_j)}.
	\end{align*}
Further, let
	\begin{align*}
	\vt&=\vt(s)=
	\inf\cbc{u\in[0,1]:\int_0^u\vz_{\mu,u}^{N,M}\dd s\geq s\int_0^1\vz_{\mu,u}^{N,M}\dd u}.
	\end{align*}
and $\mu^{N,M}_{s,x}=\mu_{\vt,x}$.
Then $\MEAS_\beta$ consists of all $\pi\in\Meas$ such that $\MU^{\pi,N,M}$ has distribution $\pi$.
Furthermore, the functional $\cB(\nix)$ reads
	\begin{align*}
	\cB_\beta(\pi)&=\Erw\brk{\ln\bc{\sum_{\sigma\in\Omega}^q\int_0^1\sum_{\sigma\in\Omega}\prod_{j=1}^d1-\tanh(\beta)\sum_{\tau\in\Omega^{k-1}}\CHI_{i,k}\sigma\prod_{j=1}^{k-1}\CHI_{i,j}\tau_j\mu_{s,\vx_{i,j}}(\tau_j)\dd s}}\\
		&\qquad\qquad+\frac{d(k-1)}k\Erw\brk{\bc{\sum_{\tau\in\Omega^k}1-\tanh\beta\int_0^1\prod_{j=1}^k\CHI_{1,j}\tau_j\mu_{s,\vx_{1,j}}(\tau_j)\dd s}}-dk\ln 2.
	\end{align*}
Let
	\begin{align*}
	\Phi_{d,\beta}&=\inf_{\pi\in\MEAS_\beta}\cB_\beta(\pi).
	\end{align*}

\begin{theorem}\label{Thm_ZkSAT}
For all $d,k\geq3$, $\beta>0$ we have $\lim_{n\to\infty}\frac1n\Erw[\ln Z(\G)]=\Phi_{d,\beta}$.
\end{theorem}
\begin{proof}
This follows immediately from \Thm~\ref{Thm_freeEng} and \Lem~\ref{Lemma_POSkSAT}.
\end{proof}

As a further application we also obtain a variational formula for the {\sc Max $k$-SAT} problem.
Specifically, with the interpretation of $\sigma\in\Omega^n$ as a truth assignment, define $\cH_{\G}(\sigma)$ as the number of propositional clauses of $\G$ that $\sigma$ fails to satisfy.
Further, let $\OPT(\G)=dn/k-\min_{\sigma\in\Omega^k}\cH_{\G}(\sigma)$ be the maximum number of clauses that can be satisfied simultaneously.
Following the steps of the proof of \Cor~\ref{Thm_MaxqCut} precisely, we obtain the following result.

\begin{corollary}\label{Cor_ZkSAT}
For all $d,k\geq3$ we have 
	$$\frac1n\OPT(\G)\ \stacksign{$n\to\infty$}\longrightarrow\ \frac{d}k+\lim_{\beta\to\infty}\Phi_{d,\beta+1}-\Phi_{d,\beta}
		\qquad\mbox{in probability}.$$
\end{corollary}

\subsection{The hard-core model}
\label{secHCintro}

The proofs of \Thm ~\ref{Thm_hc} and \Cor~\ref{Cor_hc} are not entirely straightforward because the hard-core model cannot be cast directly as a factor graph model as in \Sec~\ref{Sec_results}.
This is because of the `hard' constraint that $\SIGMA_v\SIGMA_w=0$ for any adjacent $v,w$.
We therefore prove \Thm~\ref{Thm_hc} and \Cor~\ref{Cor_hc} by way of a relaxed `soft-core model' and taking two limits, first in the `softness' and then in the fugacity.
Specifically, we obtain a random factor graph model with $\Omega=\{0,1\}$ and the prior $p(0)=1/(1+\lambda)$ and $p(1)=\lambda/(1+\lambda)$.
In addition, to mimic the hard-core constraints we would like to introduce a binary weight function that forbids its two adjacent variable nodes from both taking the spin $1$.
But since it would take values $\{0,1\}$, we instead introduce
	\begin{align*}
	\psi_\beta&:\Omega^2\to(0,1),&(\sigma_1,\sigma_2)\mapsto 1-(1-\eul^{-\beta})\sigma_1\sigma_2.
	\end{align*}
Thus, $\beta>0$ is a `softness parameter', and upon taking $\beta\to\infty$ we recover the hard-core constraint:
	$\psi_\infty(\sigma_1,\sigma_2)=1-\sigma_1\sigma_2.$
For any $\beta,\lambda$ and $d\geq3$ we obtain the random factor factor graph model $\G_{\lambda,\beta}$ with the single binary weight function $\psi_\beta$.

\begin{lemma}\label{Lemma_POShc}
The model $\G_{\lambda,\beta}$ satisfies {\bf POS} for all $d\geq3,\lambda>0,\beta\in(0,\infty]$.
\end{lemma}
\begin{proof}
Substituting $\psi_\beta$ into {\bf POS} and noticing that $1-\eul^{-\beta}>0$, we see that it suffices to verify the inequality
	\begin{align}\label{eqLemma_POShc1}
	\Erw\brk{\bc{\int_0^1\mu_{s,\vx_1}(1)\mu_{s,\vx_2}(1)\dd s}^\ell
		+\bc{\int_0^1\mu'_{s,\vx_1}(1)\mu'_{s,\vx_2}(1)\dd s}^\ell-2\bc{\int_0^1\mu_{s,\vx_1}(1)\mu'_{s,\vx_2}(1)\dd s}^\ell}&\geq0
			&(\mu,\mu'\in\meas).
	\end{align}
By Fubini's theorem,
	\begin{align}\nonumber
	\Erw\brk{\bc{\int_0^1\mu_{s,\vx_1}(1)\mu_{s,\vx_2}(1)}^\ell}&=
		\Erw\brk{\prod_{h=1}^\ell\mu_{\vs_h,\vx_1}(1)\mu_{\vs_h,\vx_2}(1)}
		=\Erw\brk{\Erw\brk{\prod_{h=1}^\ell\mu_{\vs_h,\vx_1}(1)\bigg|\vs_1,\ldots,\vs_\ell}
			\Erw\brk{\prod_{h=1}^\ell\mu_{\vs_h,\vx_2}(1)\bigg|\vs_1,\ldots,\vs_\ell}}\\
		&=\Erw\brk{\Erw\brk{\prod_{h=1}^\ell\mu_{\vs_h,\vx_1}(1)\bigg|\vs_1,\ldots,\vs_\ell}^2},
			\label{eqLemma_POShc2}
	\end{align}
and analogously
	\begin{align}			\label{eqLemma_POShc3}
	\Erw\brk{\bc{\int_0^1\mu'_{s,\vx_1}(1)\mu'_{s,\vx_2}(1)}^\ell}
		&=\Erw\brk{\Erw\brk{\prod_{h=1}^\ell\mu'_{\vs_h,\vx_1}(1)\bigg|\vs_1,\ldots,\vs_\ell}^2},\\
	\Erw\brk{\bc{\int_0^1\mu_{s,\vx_1}(1)\mu'_{s,\vx_2}(1)}^\ell}
		&=\Erw\brk{
			\Erw\brk{\prod_{h=1}^\ell\mu_{\vs_h,\vx_1}(1)\bigg|\vs_1,\ldots,\vs_\ell}
			\Erw\brk{\prod_{h=1}^\ell\mu'_{\vs_h,\vx_1}(1)\bigg|\vs_1,\ldots,\vs_\ell}}.
				\label{eqLemma_POShc4}
	\end{align}
Combining (\ref{eqLemma_POShc2})--(\ref{eqLemma_POShc4}), we conclude that the l.h.s.\ of (\ref{eqLemma_POShc1}) is the expectation of a square.
\end{proof}

We proceed to prove \Thm~\ref{Thm_hc}.
In light of \Lem~\ref{Lemma_POShc}, \Thm~\ref{Thm_freeEng} readily yields a variational formula for $\G_{\lambda,\beta}$.
The main issue that we have to confront is that the resulting variational problem for given $\lambda,\beta$ ranges over a spaces that depends on these parameters.
In effect, it is not a priori clear that these variational problems bear any relationship to the one stated in \Thm~\ref{Thm_freeEng}.
To deal with this issue, let  $\Meas_{\lambda}$ be the set of all $\pi\in\Meas$ that are supported
on $\mu\in\meas$ such that $\mu_{s,x}(1)\leq\lambda/(1+\lambda)$ for all $s,x\in[0,1]$.
Further, for $\pi\in\Meas_{\lambda}$ we let $\pi^{\Join_\beta(N,M)}$ be the distribution obtained by the adjoining operation with respect to the weight function $\psi_\beta$.
Finally, let 
	$$\MEAS_{\lambda,\beta}=\cbc{\pi\in\Meas_{\lambda}:\mbox{for all $N,M\geq0$ we have $\pi^{\Join_\beta(N,M)}=\pi$}}.$$

\begin{lemma}\label{Lemma_inftyCont}
For any $N,M\geq0$ the map $\pi\in\Meas_\lambda\mapsto\pi^{\Join_\infty(N,M)}$ is continuous.
\end{lemma}

Like in the case of \Lem~\ref{Lemma_JoinNM}, the proof is based on fairly arguments revolving around the cut metric. The details can be found in Appendix~\ref{Sec_Lemma_JoinCont}.

\begin{lemma}\label{Lemma_betaCont}
Let $N,M\geq0$ be integers.
Uniformly for all $\pi\in\Meas_\lambda$ we have $\pi^{\Join_\beta(N,M)}\to\pi^{\Join_\infty(N,M)}$ as $\beta\to\infty$.
\end{lemma}
\begin{proof}
Let $\eps>0$
For any $\mu\in\states$ let $\vZ^{N,M}_{\mu,\beta}(s)$ be the weight from \eqref{eqZNMmus} with respect to the weight function $\psi_\beta$.
Then we see that, uniformly for all $\mu$ and $s$,
	\begin{align}\label{eqLemma_betaCont1}
	\vZ^{N,M}_{\mu,\beta}(s)&\to\vZ^{N,M}_{\mu,\infty}(s)\qquad\mbox{ as }\qquad\beta\to\infty.
	\end{align}
Furthermore, if $\mu_{s,x}\leq\lambda/(1+\lambda)$ for all $s,x$, then for all $\beta\in(0,\infty]$ we have
	\begin{align}\label{eqLemma_betaCont2}
	\vZ^{N,M}_{\mu,\beta}(s)&\geq\bcfr{1}{1+\lambda}^N\bc{1-\bcfr{\lambda}{1+\lambda}^2}^M>0.
	\end{align}
Combining (\ref{eqLemma_betaCont1}) and (\ref{eqLemma_betaCont2}) and recalling the construction of $\mu^{\Join_\beta(N,M)}$, we can construct a measurable map $\xi:[0,1]\to[0,1]$ that preserves the Lebesgue measure such that for large enough $\beta$ for all $S,X\subset[0,1]$,
	\begin{align*}
	\abs{\int_S\int_X\mu_{s,x}^{\Join_\beta(N,M)}-\mu_{\xi(s),x}^{\Join_\infty(N,M)}\,\dd x\,\dd s}&<\eps.
	\end{align*}
Thus, $\Cutm(\mu^{\Join_\beta(N,M)},\mu^{\Join_\infty(N,M)})<\eps$ for large $\beta$.
Since $\Meas_\lambda$ is endowed with the $W_1$-metric, the assertion follows.
\end{proof}

\begin{lemma}\label{Lemma_PlambdaClosed}
The set $\meas_\lambda$ is closed.
\end{lemma}
\begin{proof}
We can view $\meas_\lambda$ as a scaled version of the space of weak kernels.
Therefore, since $\meas$ is complete, so is $\meas_\lambda$ is complete.
Hence, any Cauchy sequence in $\meas_\lambda$ has a limit within this set, and thus
$\meas_\lambda$ is a closed subspace of $\meas$.
\end{proof}

\begin{corollary}\label{Cor_PlambdaClosed}
The set $\Meas_\lambda$ is closed.
\end{corollary}
\begin{proof}
By \Lem~\ref{Lemma_PlambdaClosed} there exists an increasing sequence of continuous functions $u_n:\meas\to[0,1]$ that converges pointwise to $1-\vecone\meas_\lambda$.
Thus,  $\Meas_\lambda=\bigcap_{n\geq1}\cbc{\pi\in\Meas:\int u_n\dd \pi=0}$ is closed in the weak topology.
\end{proof}

\begin{corollary}\label{Cor_hcMargs}
We have $\liminf_{n \to \infty} \frac{1}{n} \Erw\brk{\log Z (\G_{\lambda,\beta})}\ge \inf_{\pi\in\MEAS_\lambda}\cB(\pi).$
\end{corollary}
\begin{proof}
Since $\MEAS$ is compact,  \Prop~\ref{MainAizLemReg} shows that there exists $\pi\in\MEAS$ such that
	\begin{align}\label{eqCor_hcMargs2}	
	\liminf_{n \to \infty} \frac{1}{n} \Erw\brk{\log Z (\G_{\lambda,\beta})}\ge \cB(\pi).
	\end{align}
The construction of the $\pi$ for which the lower bound is attained is based on \Prop~\ref{Prop_Aizenman}, whose proof shows that the measure $\pi_{\lambda,\beta}$
for which the lower bound is attained in the limit of a sequence of distributions $(\pi_{\lambda,\beta,n})_{n\geq1}$ that come from random factor graphs with the weight function $\psi_\beta$.
Specifically, we considered a random factor graph $\G_{\lambda,\beta,n,\omega}$ with a random number of `cavities' for a slowly growing $\omega=\omega_n\to\infty$.
With $\mu_{n}\in\cP(\Omega^{\cC})$ the joint Boltzmann distribution of the spins of the cavities $\cC$, the measure $\pi_{\lambda,\beta,n}$ is defined as the distribution of the representation of $\mu_{n}$ as an element of $\cM$.
Thus, we just need to show that these representations converge to points in $\meas_\lambda$.

The proof of this fact is based on \Cor~\ref{Lemma_rl}.
Specifically, let $\eps>0$.
We obtain a decomposition $S_1,\ldots,S_\ell$ of $\Omega^{\cC}$ into classes by pinning a random set $\Theta_\eps$ of cavities.
The size $|\Theta_\eps|$ of this set depends on $\eps$ only and
	\begin{align}\label{eqCor_hcMargs1}
	\cutm(\mu_n,\bar\mu_n)&<\eps,\quad\mbox{where}\qquad
		\bar\mu_n=\sum_{i=1}^\ell\mu(S_i)\bigotimes_{v\in\cC}\mu_v(\nix|S_i).
	\end{align}
Now, consider a cavity $v\in\cC\setminus\Theta_\eps$, let $1\leq i\leq\ell$ and consider a configuration $\sigma\in S_i$ with $\sigma_v=1$.
Obtain $\sigma'$ by setting $\sigma'_v=0$ and $\sigma'_w=\sigma_w$ for all $w\neq v$.
Then $\sigma'\in S_i$ and the construction of the Boltzmann distribution ensures that $\mu_n(\sigma|S_i)\leq \lambda\mu_n(\sigma'|S_i)$.
Hence, $\mu_v(1|S_i)\leq\lambda/(1+\lambda)$.
Since $|\Theta_\eps|$ is bounded in terms of $\eps$ only, whereas $|\cC|\geq\omega_n/2\to\infty$ with high probability, 
we deduce from (\ref{eqCor_hcMargs1}) that the representation $\check\mu_n\in\meas$ satisfies $\Cutm(\check\mu_n,\meas_\lambda)<\eps$ with high probability.
Since, furthermore, the Wasserstein metric induces the weak topology on $\Meas$, we conclude that $\pi_{\lambda,\beta,n}$ converges to a point $\pi$ on in the closure of $\Meas_\lambda$; but since $\Meas_\lambda$ is closed, we conclude that $\pi\in\Meas_\lambda$.
Finally, \Cor~\ref{Cor_PlambdaClosed} implies that $\pi\in\Meas_\lambda\cap\MEAS=\MEAS_\lambda$.
Thus, the assertion follows from (\ref{eqCor_hcMargs2}).
\end{proof}

We are ready to establish the lower bound on the free energy.

\begin{proposition}\label{Prop_hcLower}
For all $d\geq3,\lambda>0$ we have $\liminf_{n\to\infty}\frac1n\Erw[\ln Z(\GG_{\lambda,\infty})]\geq\Phi_{d,\lambda}.$
\end{proposition}

\begin{proof}
For any $\beta,\lambda>0$ \Cor~\ref{Cor_hcMargs} supplies $\pi_{\lambda,\beta}\in\MEAS_\lambda$ such that
	\begin{equation}\label{eqProp_hcLower_1}
	\liminf_{n \to \infty} \frac{1}{n} \Erw\brk{\log Z (\G_{\lambda,\beta})}\geq\cB_{d,\lambda,\beta}(\pi_{\lambda,\beta}).
	\end{equation}
Now consider the sequence $(\pi_{\lambda,\beta})_{\beta=1,2,\ldots}$.
Since $\Meas_\lambda$ is compact, a subsequence $(\pi_{\lambda,\beta_j})_j$ converges to $\pi_{\lambda}\in\Meas_\lambda$, i.e.,
	\begin{equation}\label{eqProp_hcLower_2}
	\lim_{j\to\infty}\Cutm(\pi_{\lambda,\beta_j},\pi_{\lambda})=0.
	\end{equation}
Further, since $\pi_{\lambda,\beta_j}^{\Join_{\beta_j}(N,M)}=\pi_{\lambda,\beta_j}$ for all $j$ and $N,M\geq0$, \Lem~\ref{Lemma_betaCont} implies that for all pairs $N,M\geq0$,
	\begin{equation}\label{eqProp_hcLower_3}
	\lim_{j\to\infty}\Cutm(\pi_{\lambda,\beta_j},\pi_{\lambda,\beta_j}^{\Join_{\infty}(N,M)})=0.
	\end{equation}
Combining (\ref{eqProp_hcLower_2}) and (\ref{eqProp_hcLower_3}) with \Lem~\ref{Lemma_inftyCont}, we conclude that $\pi_\lambda\in\MEAS_\lambda$.
Finally, since for every $\beta>0$ we have $\cB_{d,\lambda,\beta}(\nix)\geq\cB_{d,\lambda,\infty}(\nix)$ on $\Meas_\lambda$, the assertion follows from (\ref{eqProp_hcLower_1}) and the continuity of the functional $\cB_{d,\lambda,\infty}(\nix)$.
\end{proof}

A separate argument is needed to derive the upper bound on the free energy.
Basically, we will prove the following proposition by checking that the interpolation argument from \Sec~\ref{secInterpolate} goes through for the hard-core model.

\begin{proposition}\label{Prop_hcUpper}
For all $d\geq3,\lambda>0$ we have $\limsup_{n\to\infty}\frac1n\Erw[\ln Z(\GG_{\lambda,\infty})]\leq\Phi_{d,\lambda}.$
\end{proposition}

With $\VARPHI_i$ and $\PSI_{1,i}$ defined with respect to the hard-core weight function $\psi_\infty$, let
	\begin{align*}
	\cB'(\mu)&=\Erw\ln\scal{\bigoplus_{i=1}^n\VARPHI_i}\mu,&
	\cB''(\mu)&=\Erw\ln\scal{\bigoplus_{1\leq i\leq dn/2}\PSI_{1,i}}\mu.
	\end{align*}

\begin{lemma}\label{Lemma_hcInt}
For any $\lambda>0$ and any $\mu\in\meas_\lambda$ we have
	$\Erw\brk{\ln Z(\G_{\lambda,\infty})}\leq \cB'(\mu)-\cB''(\mu)+o(n)$.
\end{lemma}
\begin{proof}
This follows along the lines of the proof of \Prop~\ref{Prop_regint1}.
In that proof we required the assumption that all weight functions are strictly positive, but only in one place.
Namely, we required positivity in order expand the logarithm into a power series in equations  (\ref{eqLemma_intderiv7})--(\ref{eqLemma_intderiv9}).
Yet this approximation is still valid in the hardcore model.
Indeed, the term $\scal{\psi_{a_{\vm_t+1}}}{\mu_{\G_t}}$, whose logarithm we calculate in (\ref{eqLemma_intderiv7}), is lower-bounded by $1-\lambda/(1+\lambda)$, because in the hard-core model the marginal probability that a single variable node has spin one is upper-bounded by  $\lambda/(1+\lambda)$.
Similarly, the arguments of the logarithms in (\ref{eqLemma_intderiv8}) and (\ref{eqLemma_intderiv9}) are lower-bounded by $1-\lambda/(1+\lambda)$ because $\mu\in\meas_\lambda$.
\end{proof}

\begin{proof}[Proof of \Prop~\ref{Prop_hcUpper}]
Based on \Lem~\ref{Lemma_hcInt}, we follow the proof of \Prop~\ref{Prop_regint2} to complete the proof of \Prop~\ref{Prop_hcUpper}.
Specifically, we claim that for any $\pi\in\MEAS_\lambda$,
	\begin{align}\label{eqProp_hcUpper1}
	\Erw[\cB''(\MU^\pi)]&=\frac{dn}{2}\Erw\brk{\ln\scal{\psi_\infty}{\pi}},&&\mbox{and}&
	\Erw[\cB'(\MU^\pi)]&=\Erw\log\scal{\VARPHI_1}{\pi}.
	\end{align}
This follows along the lines of \Lem s~\ref{Lemma_Prop_regint2_B''} and~\ref{Lemma_Prop_regint2_B'}.
In both cases we assumed that the weight functions are strictly positive in order to ensure that the arguments of the logarithms on the l.h.s.\ are bounded away from zero so that the logarithmic series applies.
But the condition $\pi\in\MEAS_\lambda$ guarantees that
	\begin{align*}
	\scal{\bigoplus_{i=1}^n\VARPHI_i}\mu&\geq(1/(1+\lambda))^n&&\mbox{ and }&\scal{\bigoplus_{1\leq i\leq dn/2}\PSI_{1,i}}\mu&\geq(1/(1+\lambda))^{dn/2}.
	\end{align*}
Thus, the same manipulations as before yield (\ref{eqProp_hcUpper1}).
Finally, the assertion follows from (\ref{eqProp_hcUpper1}) and \Lem~\ref{Lemma_hcInt}.
\end{proof}

\begin{proof}[Proof of \Thm~\ref{Thm_hc}]
The theorem is an immediate consequence of \Prop s~\ref{Prop_hcLower} and~\ref{Prop_hcUpper}.
\end{proof}

\begin{proof}[Proof of \Cor~\ref{Cor_hc}]
For a graph $G=(V(G),E(G))$ let $\mu_{G,\lambda}\in\cP(\{0,1\}^{V(G)})$ denote the hard-core model on $G$ with fugacity $\lambda$, and let $Z_\lambda(G)$ be the corresponding partition function.
Further, let $\alpha_\lambda(G)=\sum_{v\in V(G)}\scal{\SIGMA_v}{\mu_{G,\lam}}$ be the average size of an independent set drawn from $\mu_{G,\lambda}$.
Additionally, we write $\alpha(G)$ for the maximum independent set size.
It is well known that
	\begin{align}\label{eqCor_hc1}
	\alpha_\lambda(G)=\lambda\frac{\partial}{\partial\lambda}\ln Z_\lambda(G)
	\end{align}
and that
	\begin{align}\label{eqCor_hc2}
	\frac{\alpha_\lambda(G)}{|V(G)|}&\quad\stacksign{$\lambda\to\infty$}\longrightarrow\quad\frac{\alpha(G)}{|V(G)|}
	&&\mbox{uniformly for all $G$}.
	\end{align}
As an immediate consequence of (\ref{eqCor_hc1}) we obtain
	\begin{align*}
	\ln Z_{\lambda+1}(G)-\ln Z_{\lambda}(G)&=\int_\lambda^{\lambda+1}\frac{\alpha_t(G)}{t}\dd t\begin{cases}\leq\alpha_{\lambda+1}(G)/\lambda,\\
	\geq\alpha_\lambda(G)/(\lambda+1).
		\end{cases}
	\end{align*}
Hence, (\ref{eqCor_hc2}) shows that for any $\eps>0$ there exists $\lambda_0>0$ such that for all $\lambda\geq\lambda_0$ and all
	$d$-regular graphs $G$ we have
	\begin{align}\label{eqCor_hc3}
	(1-\eps)\alpha(G)\leq\frac\lambda{1+\lambda}\alpha_\lambda(G)&\leq
		\lambda(\ln Z_{\lambda+1}(G)-\ln Z_{\lambda}(G))\leq\alpha_\lambda(G)\leq\alpha(G).
	\end{align}
Applying (\ref{eqCor_hc3}) to the random graph $\G_\lambda$ and taking expectations, we obtain
	\begin{align}\label{eqCor_hc4}
	(1-\eps)\Erw\brk{\frac{\alpha(\GG)}n}&\leq
		\Erw\brk{\lambda(\ln Z_{\lambda+1}(\GG)-\ln Z_{\lambda}(\GG))}\leq
			\Erw\brk{\frac{\alpha(\GG)}n}.
	\end{align}
\Thm~\ref{Thm_hc} guarantees that the sequence $\bc{\Erw\brk{\lambda(\ln Z_{\lambda+1}(\GG)-\ln Z_{\lambda}(\GG))}}_n$ converges, and thus (\ref{eqCor_hc4}) yields
	\begin{align}\label{eqCor_hc5}
	(1-\eps)\limsup_{n\to\infty}\Erw\brk{\frac{\alpha(\GG)}n}&\leq
		\lambda(\Phi_{d,\lambda+1}-\Phi_{d,\lambda})\leq
			\liminf_{n\to\infty}\,\Erw\brk{\frac{\alpha(\GG)}n}.
	\end{align}
Further, there exists a subsequence $(n_l)_{l\geq1}$ along which $\Erw[\alpha(\GG)/n]$ converges to $\alpha_*\in[0,1]$, whence (\ref{eqCor_hc4}) yields
	\begin{align}\label{eqCor_hc6}
	(1-\eps)\alpha_*&\leq\lambda(\Phi_{d,\lambda+1}-\Phi_{d,\lambda})\leq\alpha_*.
	\end{align}
Since (\ref{eqCor_hc6}) holds for every $\eps>0$ for large enough $\lambda$, we conclude that
$\lim_{\lambda\to\infty}\lambda(\Phi_{d,\lambda+1}-\Phi_{d,\lambda})$ exists.
Hence, taking the limit $\eps\to0$, and thus $\lambda\to\infty$, in (\ref{eqCor_hc5}) completes the proof.
\end{proof}

\subsection*{Acknowledgement}
The first author thanks Max Hahn-Klimroth for helpful discussions on the cut metric.

\begin{appendix}

\section{Proof of \Lem s~\ref{Lemma_JoinCont} and~\ref{Lemma_inftyCont}%
	}\label{Sec_Lemma_JoinCont}

\noindent
The proof of \Lem~\ref{Lemma_JoinCont} requires the regularity lemma for measures from~\cite{Limits}.\footnote{The arguments in the appendix are special cases of more general results on the cut metric from~\cite{MHK}.}
Let $\lambda$ denote the Lebesgue measure.
For $\mu\in\states$ and measurable $S,X\subset[0,1]$ we write
	$$\mu_{S,X}=\frac1{\lambda(S)\lambda(X)}\int_S\int_X\mu_{s,x}\dd x\dd s\in\cP(\Omega),$$
with the convention that $\mu_{S,X}$ is uniform if $\lambda(S)\lambda(X)=0$.
Further, let $\vX=(X_1,\ldots,X_K),\vS=(S_1,\ldots,S_L)$ be a partitions of $[0,1)$ into pairwise disjoint measurable sets.
We write $\#\vX,\#\vS$ for the number $K,L$ of classes, respectively.
Then $\mu$ is {\em $\eps$-regular} with respect to $(\vX,\vS)$
if there exists $R\subset[\#\vX]\times[\#\vS]$ such that the following conditions hold.
\begin{description}
\item[REG1] $\lambda(X_i)>0$ and $\lambda(S_j)>0$ for all $(i,j)\in R$.
\item[REG2] $\sum_{(i,j)\in R}\lambda(X_i)\lambda(S_j)>1-\eps$.
\item[REG3] for all $(i,j)\in R$ and almost all $s,s'\in S_j$ we have
	$\tv{\int_{X_i}\mu_{s,x}-\mu'_{s',x}\dd x}<\eps\lambda(X_i)$.
\item[REG4] if $(i,j)\in R$, then for every $U\subset X_i$ with $\lambda(U)\geq\eps\lambda(X_i)$
	and every $T\subset S_j$ with $\lambda(T)\geq\eps\lambda(S_j)$ we have
		$$\TV{\mu_{S,X_i}-\mu_{T,U}}<\eps.$$
\end{description}

A {\em refinement} of a partition $(\vX,\vS)$ is a partition $(\vX',\vS')$ such that
for every pair $(i',j')\in[\#\vX']\times[\vS']$ there is a pair $(i,j)\in [\#\vX]\times[\vS]$ such that $(X_{i'}',S_{j'}')\subset(X_i,S_j)$.

\begin{theorem}[\cite{Limits}]\label{Thm_reg}
For any $\eps>0$ there exists $N=N(\eps,\Omega)$ such that for every $\mu\in\states$ the following is true.
Every partition $(\vX_0,\vS_0)$ with $\#\vX_0+\#\vS_0\leq1/\eps$ has a refinement
$(\vX,\vS)$ such that $\#\vX+\#\vS\leq N$ with respect to which $\mu$ is $\eps$-regular.
\end{theorem}

Additionally, we need the {\em strong cut metric}, defined by
	\begin{align*}
	\CUTM(\mu,\nu)&=\sup_{S,X,\omega}\abs{\int_S\int_X\mu_{s,x}(\omega)-\nu_{s,x}(\omega)\dd x\dd s}
			\qquad(\mu,\nu\in\states),
	\end{align*}
where $S,X$ range over measurable subsets of the unit interval and $\omega\in\Omega$.
It is well known that $\CUTM(\nix,\nix)$ induces a metric on $\states$.

For $\mu,\nu\in\states$ we define $\mu\oplus\nu:[0,1]^3\to\cP(\Omega^2)$ by	$\mu\oplus\nu_{s,x_1,x_2}=\mu_{s,x_1}\tensor\mu_{s,x_2}$.
Since $[0,1]^2$ with the Lebesgue measure is isomorphic as a measure space to $[0,1]$ with the Lebesgue measure, we can view $\mu\oplus\nu$ as a strong $\cP(\Omega^2)$-valued kernel.
In particular, it makes sense to apply the strong cut metric to these kernels.

\begin{proposition}\label{Prop_oplusCont}
The map $(\mu,\nu)\mapsto\mu\oplus\nu$ is continuous with respect to the strong cut metric.
\end{proposition}
\begin{proof}
Given $\eps>0$ pick a small enough $\delta>0$ and assume that $\CUTM(\mu,\mu')<\delta$.
Due to the triangle inequality it suffices to prove that  $\CUTM(\mu\oplus\nu,\mu'\oplus\nu)<\eps$ for every $\nu$.
Thus, we need to show that for any $X\subset[0,1]^2$, $S\subset[0,1]$ and $\sigma,\tau\in\Omega$,
	\begin{align}\label{eqProp_oplusCont1}
	\abs{\int_X\int_S\bc{\mu_{s,x_1}(\sigma)-\mu'_{s,x_1}(\sigma)}\nu_{s,x_2}(\tau)\dd s\dd x_1\dd x_2}<\eps.
	\end{align}
To this end, we may assume that $\lambda(S)>\eps^2$ and that $\int_S\nu_{s,x_2}(\tau)\dd s>\eps^2$ for all $(x_1,x_2)\in X$.
Further, with $z=\int_0^1\nu_{s,x_2}(\tau)\dd s$ consider the variable transformation
	\begin{align}\label{eqProp_oplusCont2}
	\dd t=\frac{\nu_{s,x_2}(\tau)\dd s}{z}.
	\end{align}
Let $T$ be the inverse image of $S$ under the transformation \eqref{eqProp_oplusCont2}.
Then we obtain for any $X_1\subset[0,1]$,
	\begin{align}\label{eqProp_oplusCont3}
	\int_{X_1}\int_S\bc{\mu_{s,x_1}(\sigma)-\mu'_{s,x_1}(\sigma)}\nu_{s,x_2}(\tau)\dd s\dd x_1
		&=z\int_{X_1}\int_T\mu_{t,x_1}(\sigma)-\mu'_{t,x_1}(\sigma)\dd t\dd x_1.
	\end{align}
But the assumption $\CUTM(\mu,\mu')<\delta$ implies that the double integral on the r.h.s.\ of (\ref{eqProp_oplusCont3}) is bounded by $\eps^4$ in absolute value (providing $\delta$ is small enough).
Thus, (\ref{eqProp_oplusCont1}) follows.
\end{proof}

\begin{proof}[Proof of \Lem~\ref{Lemma_JoinCont}]
We may assume without loss that $f(\tau)=\vecone\{\tau=\sigma\}$ for some $\sigma\in\Omega^k$.
Let $\eps>0$, pick $\alpha=\alpha(\eps)$, $\xi=\xi(\alpha)>0$ small enough and assume that $\mu,\nu\in\states$ are such that $\CUTM(\mu,\nu)<\delta$ for a small enough $\delta=\delta(\xi)>0$.
Applying \Thm~\ref{Thm_reg} twice, we obtain $(\vX,\vS)$ with respect to which both $\mu,\nu$ are $\xi$-regular, and $L=\#\vX+\#\vS$ is bounded in terms of $\xi$ only.
Let $R'$ be the set of all pairs for which {\bf REG1--REG4} are satisfied for both $\mu,\nu$ and that satisfy $\lambda(\vX_i,\vS_j)>\xi^8/L$.
Assuming that $\delta$ is sufficiently small, we obtain 
	\begin{equation}\label{eqLemma_JoinCont_1a}
|\mu_{S_i,X_j}-\nu_{S_i,X_j}|<\xi^8\qquad\mbox{ for all $(i,j)\in R'$.}
\end{equation}
Furthermore, consider the random variables
	\begin{align*}
	z_i&=\prod_{h=1}^k\mu_{S_i,\vx_h}(\sigma_h),&z&=\sum_{i\leq\#\vS}z_i,\\
	z_i'&=\prod_{h=1}^k\nu_{S_i,\vx_h}(\sigma_h),&z'&=\sum_{i\leq\#\vS}z_i'
	\end{align*}
and define $\mu',\nu'\in\states$ as follows.
To construct $\mu'$, partition the interval $[0,1]$ into pairwise disjoint sets $T_i$, $i\in[\#\vS]$, of measure $z_i/z$ and fill the strip $T_i\times[0,1]$
with a suitably scaled copy of $(\mu_{s,x})_{s\in S_i,x\in[0,1]}$.
Construct $\nu'$ analogously from the $z_i'$.
Then $\Cutm(\mu',f\Join\mu)=\Cutm(\nu',f\Join\nu)=0$.
Furthermore, \Prop~\ref{Prop_oplusCont} shows that with probability at least $1-\alpha$ we have 
	$$\sum_{i=1}^{\#\vS}\lambda(S_i)|z_i-z_i'|<\alpha^2,$$
provided that $\xi,\delta$ are chosen small enough.
Since also $z\geq\alpha$ because the function $f$ is strictly positive, we conclude that with probability at least $1-\alpha$ we have $\Cutm(\mu',\nu')<\alpha$.
We thus obtain a coupling of the random variables $f\Join \mu,f\Join\nu$ under which the expected cut distance is bounded by $\eps$, as desired.
\end{proof}

\begin{proof}[Proof of \Lem~\ref{Lemma_inftyCont}]
We proceed precisely as in the proof of \Lem~\ref{Lemma_JoinCont}, up until the point where the positivity of $f$ is used.
In the setup of \Lem~\ref{Lemma_inftyCont}, the function $f$ may take the value $0$ on kernels that take the value $1$ with positive probability; however, since we are assuming that the values of the kernels are bounded by $\lambda/(1+\lambda)$.
Therefore, the function $f$ always attains values that are bounded away from $0$.
\end{proof}

\section{Proof of \Lem~\ref{Lemma_cntFct}}

\noindent
The proof of \Lem~\ref{Lemma_cntFct} requires the following operation.
For functions $f:\Omega^{M\times N}\to\RR$, $g:\Omega^{L\times N}\to\RR$ we define
	\begin{align*}
	f\otimes g&:\Omega^{(M+L)\times N}\to\RR,&
	\sigma&\mapsto f\bc{(\sigma_{i,j})_{i\in[M],j\in[N]}}\cdot g\bc{(\sigma_{i+M,j+N})_{i\in[L],j\in[N]}}.
	\end{align*}	
Thus, the first $M$ rows of $\sigma$ go into $f$, the last $L$ rows go into $g$ and we multiply the results.

We define a corresponding operation on kernels.
Namely, for $\mu,\nu\in\states$ we define $\mu\otimes\nu:[0,1]^3\to\cP(\Omega^2)$ by	$\mu\oplus\nu_{s,t,x}=\mu_{s,x}\tensor\nu_{t,x}$.
Since $([0,1]^2,\lambda\tensor\lambda)$ is isomorphic $([0,1],\lambda)$, we can view $\mu\otimes\nu$ as a  $\cP(\Omega^2)$-valued kernel, and the cut metric extends to these kernels.
Since the cut metric is invariant under swapping the axes, \Prop~\ref{Prop_oplusCont} readily yields the following.

\begin{proposition}\label{Prop_otimesCont}
The map $(\mu,\nu)\mapsto\mu\otimes\nu$ is continuous with respect to the cut metric.
\end{proposition}

As a final preparation toward the proof of \Lem~\ref{Lemma_cntFct} we need the following fact.

\begin{lemma}\label{Lemma_cntFct1d}
For any $f:\Omega\to\RR$ the map $\mu\in\meas\mapsto\Erw\scal f\mu$ is continuous.
\end{lemma}
\begin{proof}
We may assume without loss that $f(\tau)=\vecone\{\sigma=\tau\}$ for some $\sigma\in\Omega$.
Then 
	$$\Erw\scal f\mu=\int_0^1\int_0^1\mu_{s,x}(\sigma)\dd x\dd s,$$
and it is immediate from the definition of the cut metric that the integral on the right hand side is a continuous function of $\mu$.
\end{proof}

\begin{proof}[Proof of \Lem~\ref{Lemma_cntFct}]
Let $f:\Omega^{m\times n}\to\RR$ and let $\mu\in\meas$.
Define $\nu=(\mu^{\oplus n})^{\otimes m}$.
Then $\nu$ is a kernel with values in $\Omega^{mn}$ and the definition of $\scal\nix\nix$ ensures that $\Erw\scal{f}\mu=\Erw\scal{f}\nu$.
This already shows that the map $\mu\mapsto\Erw\scal{f}\mu$ is continuous, because the map
$\mu\mapsto\nu$ is continuous by \Prop~\ref{Prop_oplusCont} and~\ref{Prop_otimesCont} and the map
$\nu\mapsto\Erw\scal f\nu$ is continuous by \Lem~\ref{Lemma_cntFct1d}.
Now fix an integer $\ell\geq2$ and let $\eta=\nu^{\otimes\ell}$.
Then
		$$\Erw\brk{\scal{f}\mu^\ell}=\Erw\brk{\scal{f}\eta}$$
and thus the continuity of the map $\mu\mapsto\Erw\brk{\scal{f}\mu^\ell}$ follows from \Prop~\ref{Prop_otimesCont} and \Lem~\ref{Lemma_cntFct1d}.
\end{proof}

\end{appendix}

\end{document}